\documentclass{article}
\usepackage{hyperref}
\usepackage{verbatim}
\usepackage{amssymb}
\usepackage{amsthm}
\usepackage{amsmath}
\usepackage[dvipsnames]{xcolor}
\usepackage{tikz}
\usepackage{float}
\usetikzlibrary{fadings}
\usetikzlibrary{patterns}
\usepackage{biblatex}
\DeclareFieldFormat
  [article,inbook,incollection,inproceedings,patent,thesis,unpublished]
  {title}{#1\isdot}

\title{Turns in Hamilton cycles of rectangular grids}
\author{Ethan Y. Tan\thanks{
                  Trinity College, Cambridge CB21TQ, UK. Email: \texttt{eyet2@cam.ac.uk}}
        \and
        Guowen Zhang\thanks{
                      \'{E}cole Normale Sup\'{e}rieure, 45 Rue d'Ulm, Cedex 05, Paris, 75230, France. Email: \texttt{wen00@live.com.au}
        }}
\usepackage{graphicx}

\theoremstyle{plain}

\newtheorem{theorem}{Theorem}[section]
\newtheorem{lemma}[theorem]{Lemma}

\newtheorem{corollary}[theorem]{Corollary}
\newtheorem{claim}[theorem]{Claim}

\theoremstyle{definition}

\newtheorem{question}{Question}[section]
\newtheorem*{acknowledgement}{Acknowledgements}

\begin{document}
\setlength{\parindent}{0pt}
\setlength{\parskip}{0.5em}

\newpage
\maketitle
\begin{abstract}
For a Hamilton cycle in a rectangular $m \times n$ grid, what is the greatest number of turns that can occur? We give the exact answer in several cases and an answer up to an additive error of $2$ in all other cases. In particular, we give a new proof of the result of Beluhov for the case of a square $n \times n$ grid. Our main method is a surprising link between the problem of `greatest number of turns' and the problem of `least number of turns'.
\end{abstract}

\section{Introduction}

The grid $G(m,n)$ is the graph on vertex set $\{1, \cdots, m\} \times \{1, \cdots, n\}$, with vertices $(a,b)$ and $(c,d)$ joined by an edge if and only if $|a-c| + |b-d| = 1$. In other words, $G(m,n)$ is the Cartesian product of a path on $m$ vertices and a path on $n$ vertices. As usual, a \textit{Hamilton cycle} is a cycle which visits every vertex. For convenience, we refer to vertices as \textit{cells}, with vertices represented as squares.

    \begin{figure}[H]
        \centering
        \begin{tikzpicture}
\draw[step=1cm,gray,dashed] (0,0) grid (6,4);
\draw [black, thin] (0, 0) -- (6, 0) -- (6, 4) -- (0, 4) -- cycle;
\draw [gray,thin] (0.5, 0.5) circle (0.4);
\draw [gray,thin] (3.5, 0.5) circle (0.4);
\draw [gray,thin] (4.5, 0.5) circle (0.4);
\draw [gray,thin] (5.5, 0.5) circle (0.4);
\draw [gray,thin] (1.5, 1.5) circle (0.4);
\draw [gray,thin] (2.5, 1.5) circle (0.4);
\draw [gray,thin] (3.5, 2.5) circle (0.4);
\draw [gray,thin] (4.5, 2.5) circle (0.4);
\draw [gray,thin] (0.5, 3.5) circle (0.4);
\draw [gray,thin] (1.5, 3.5) circle (0.4);
\draw [gray,thin] (2.5, 3.5) circle (0.4);
\draw [gray,thin] (5.5, 3.5) circle (0.4);
\draw [black, thick] (0.5, 0.5) -- (1.5, 0.5);
\draw [black, thick] (1.5, 0.5) -- (2.5, 0.5);
\draw [black, thick] (2.5, 0.5) -- (3.5, 0.5);
\draw [black, thick] (4.5, 0.5) -- (5.5, 0.5);
\draw [black, thick] (1.5, 1.5) -- (2.5, 1.5);
\draw [black, thick] (3.5, 2.5) -- (4.5, 2.5);
\draw [black, thick] (0.5, 3.5) -- (1.5, 3.5);
\draw [black, thick] (2.5, 3.5) -- (3.5, 3.5);
\draw [black, thick] (3.5, 3.5) -- (4.5, 3.5);
\draw [black, thick] (4.5, 3.5) -- (5.5, 3.5);
\draw [black, thick] (0.5, 0.5) -- (0.5, 1.5);
\draw [black, thick] (3.5, 0.5) -- (3.5, 1.5);
\draw [black, thick] (4.5, 0.5) -- (4.5, 1.5);
\draw [black, thick] (5.5, 0.5) -- (5.5, 1.5);
\draw [black, thick] (0.5, 1.5) -- (0.5, 2.5);
\draw [black, thick] (1.5, 1.5) -- (1.5, 2.5);
\draw [black, thick] (2.5, 1.5) -- (2.5, 2.5);
\draw [black, thick] (3.5, 1.5) -- (3.5, 2.5);
\draw [black, thick] (4.5, 1.5) -- (4.5, 2.5);
\draw [black, thick] (5.5, 1.5) -- (5.5, 2.5);
\draw [black, thick] (0.5, 2.5) -- (0.5, 3.5);
\draw [black, thick] (1.5, 2.5) -- (1.5, 3.5);
\draw [black, thick] (2.5, 2.5) -- (2.5, 3.5);
\draw [black, thick] (5.5, 2.5) -- (5.5, 3.5);
        \end{tikzpicture}
        \caption{A Hamilton cycle of the $6 \times 4$ grid $G(6,4)$. It has $12$ turns, marked with circles.}
        \label{examplegrid}
    \end{figure}

For a given cycle, a {\it turn} is a vertex at which the cycle turns $90$ degrees.

In this paper, we will study the following problem: over all Hamilton cycles of an $m \times n$ grid, what is the  maximum number of turns that can be attained? Of course, $mn$ must be even for such a cycle to exist; we will assume that $mn$ is even from here on. The case where $m$ or $n$ is $2$ is trivial.

The earliest reference to the minimum turns variant of this problem goes back to Loyd \cite{loyd}, who asked the problem of minimum turns on an $8 \times 8$ chessboard, and provided an example of a minimal cycle. Jelliss \cite{jelliss} outlined a proof that the minimum number of turns on $G(m,n)$ where $m \geq n$ is $2m$ if $n$ is odd and $2n$ otherwise.

The earliest reference known to the problem of maximum turns in a rectangle is Gik's book \cite{gik}, which poses the problem on an $8 \times 8$ chessboard, and gives a construction with $56$ turns without proof; the case for squares of any size was solved by Beluhov \cite{beluhov}, in the following theorem:

\begin{theorem}[Beluhov \cite{beluhov}]\label{square beluhov}
The maximum number of turns in a Hamilton cycle on $G(n,n)$ with $n>2$ is $n^2-n$ if $4|n$, and $n^2-n-2$ otherwise.
\end{theorem}

The methods used in \cite{beluhov} are tailored to square grids, and do not apply for general rectangular grids. Our main aim in this paper is to give an answer for a general rectangular grid, to within an additive error of $2$:

\begin{theorem}\label{final result}
The maximum number of turns possible over all Hamilton cycles on $G(m,n)$ with $m > n > 2$ is
$$f_{\max}(m,n) = \begin{cases}mn - n & \textnormal{if } n \equiv 0 \pmod 4 \\ mn - m & \textnormal{if } n \equiv 1 \textnormal{ or }3 \pmod 4 \\ mn-g(m,n) \textnormal{ or } mn-g(m,n)-2 & \textnormal{if } n \equiv 2 \pmod 4\end{cases}$$
where $g$ is given by
$$g(m,n) = 2\left\lfloor \frac{2m+n+2}{6}\right\rfloor.$$
\end{theorem}

We begin with some preliminary observations in section \ref{prelim}. We then solve the problem of the minimum number of turns in a grid in section \ref{min}. This least turns problem is significantly easier than the problem of maximum turns. In section \ref{max}, we prove Theorem \ref{square beluhov}, using the solution for the minimum turns case as the basis of a proof. Finally, we generalise our method to general rectangles in section \ref{rect}, obtaining an exact answer for many cases and an answer within an additive error of $2$ for all other cases.


\section{Preliminary results}\label{prelim}
We collect here some simple facts that we will need later.

\begin{lemma} \label{evenrowcol}
For any cycle $C$ on a grid, each row and column has an even number of turns.
\end{lemma}
\begin{proof}
We prove the result for rows; columns follow analogously.

Orient the cycle clockwise, and choose some row $R$. Start from a cell not in $R$, and travel along the entire cycle once. Label every turn in $R$ by whether the cycle enters or exits the row at that point. Every turn where the cycle enters must be followed by a turn where the cycle exits, and similarly, every turn where a cycle exits has an associated turn where the cycle last entered.

Hence, turns where the cycle enters can be paired with turns where the cycle exits, and thus the number of turns in $R$ is even.
\end{proof}

\begin{lemma}\label{turnexistence}
For a cycle $C$, and a cell $c$ in this cycle, there is at least one turn in the same row or column as $c$.
\end{lemma}
\begin{proof}
Follow the cycle as it leaves $c$; wherever it goes next, it must turn at some point.
\end{proof}

Combining Lemmas \ref{evenrowcol} and \ref{turnexistence}, we have the following corollary.
\begin{corollary}\label{2turns}
For any cell $c$ in a cycle $C$, there are at least two turns in the union of the row and column containing $c$. \qed
\end{corollary}

\begin{lemma}\label{turninoddrow}
In a Hamilton cycle on $G(m,n)$, where $m$ is odd, there is at least one turn in every row.
\end{lemma}
\begin{proof}
Assume for contradiction that some row $R$ contains no turns. Note that the cycle cannot pass through any cell in row $R$ horizontally.

Orient the cycle clockwise. The cycle passes through each square in $R$ in one of two ways: moving upward or moving downward. Choose a cell $c$ not in $R$. Without loss of generality, assume that $c$ is below $R$. We now travel along the loop exactly once, starting and ending at $c$.

The first cell in $R$ that we visit must be moved through upward, and the last cell in $R$ we visit must be moved through downward. The direction we pass through $R$ alternates as we travel, and hence there are the same number of squares moved upward as downward. But this implies that there are an even number of squares in $R$, contradicting the assumption that $m$ is odd.
\end{proof}

By symmetry, if $n$ is odd, we have at least one turn every column.

Combining Lemmas \ref{evenrowcol} and \ref{turninoddrow}, we have the following corollary. 
\begin{corollary}\label{2turnsodd}
In a Hamilton cycle on $G(m,n)$, where $m$ is odd, then there are at least two turns in every row. \qed
\end{corollary}

\section{Minimum turns}\label{min}
In this section we solve the problem of the minimum number of turns over all cycles of a grid; this problem is significantly easier than the problem of maximum turns, and the result will be used in the proof of specific cases of the maximum turn question. The main result we prove in this section is the following theorem.

\begin{theorem}\label{minturns}
The minimum number of turns of any Hamilton cycle of $G(m,n)$, where $m \geq n$, is $2m$ if $n$ is odd, and $2n$ otherwise.
\end{theorem}

We start with the following lemma.

\begin{lemma}\label{minab}
In a Hamilton cycle on $G(m,n)$ where $m \geq n$, there are at least $2n$ turns.
\end{lemma}
\begin{proof}
The union of any given row and column contains at least two turns, by Corollary \ref{2turns}. If there exists a row containing no turns, then every column has at least two turns, and there are at least $2m \geq 2n$ turns. Otherwise, every row has at least two turns by Lemma \ref{evenrowcol}, and there are at least $2n$ turns.
\end{proof}

\begin{lemma}\label{oddturn}
In a Hamilton cycle on $G(m,n)$, if $m$ is odd, there are at least $2n$ turns.
\end{lemma}
\begin{proof}
This follows immediately from Corollary \ref{2turnsodd}.
\end{proof}

Lemmas \ref{minab} and \ref{oddturn} give us bounds on $G(m,n)$. In fact, these bounds are tight; for any $m, n$, the stronger of these bounds is always attainable, which we will demonstrate below.

\begin{claim}\label{minconstruct}
In a Hamilton cycle on $G(m,n)$, where $m \geq n$, there is a cycle with $2n$ turns if $n$ is even, and $2m$ turns if $n$ is odd.
\end{claim}
\begin{proof}
If $m$ is even, we can construct a cycle that contains exactly two turns in every row, giving a total of $2n$ turns, by using a ``pronged" shape as shown in the below two figures:

    \begin{figure}[H]
        \centering
        \begin{tikzpicture}
\draw[step=1cm,gray,thin,dashed] (0,0) grid (6,4);
\draw [black, thin] (0, 0) -- (6, 0) -- (6, 4) -- (0, 4) -- cycle;
\draw [black, thick] (0.5, 0.5) -- (1.5, 0.5);
\draw [black, thick] (1.5, 0.5) -- (2.5, 0.5);
\draw [black, thick] (2.5, 0.5) -- (3.5, 0.5);
\draw [black, thick] (3.5, 0.5) -- (4.5, 0.5);
\draw [black, thick] (4.5, 0.5) -- (5.5, 0.5);
\draw [black, thick] (1.5, 1.5) -- (2.5, 1.5);
\draw [black, thick] (2.5, 1.5) -- (3.5, 1.5);
\draw [black, thick] (3.5, 1.5) -- (4.5, 1.5);
\draw [black, thick] (4.5, 1.5) -- (5.5, 1.5);
\draw [black, thick] (1.5, 2.5) -- (2.5, 2.5);
\draw [black, thick] (2.5, 2.5) -- (3.5, 2.5);
\draw [black, thick] (3.5, 2.5) -- (4.5, 2.5);
\draw [black, thick] (4.5, 2.5) -- (5.5, 2.5);
\draw [black, thick] (0.5, 3.5) -- (1.5, 3.5);
\draw [black, thick] (1.5, 3.5) -- (2.5, 3.5);
\draw [black, thick] (2.5, 3.5) -- (3.5, 3.5);
\draw [black, thick] (3.5, 3.5) -- (4.5, 3.5);
\draw [black, thick] (4.5, 3.5) -- (5.5, 3.5);
\draw [black, thick] (0.5, 0.5) -- (0.5, 1.5);
\draw [black, thick] (5.5, 0.5) -- (5.5, 1.5);
\draw [black, thick] (0.5, 1.5) -- (0.5, 2.5);
\draw [black, thick] (1.5, 1.5) -- (1.5, 2.5);
\draw [black, thick] (0.5, 2.5) -- (0.5, 3.5);
\draw [black, thick] (5.5, 2.5) -- (5.5, 3.5);
        \end{tikzpicture}
        \caption{$8$ turns in $G(6,4)$.}
        \label{6x4min}
    \end{figure}

    \begin{figure}[H]
        \centering
        \begin{tikzpicture}
\draw[step=1cm,gray,thin,dashed] (0,0) grid (11,8);
\draw [black, thin] (0, 0) -- (11, 0) -- (11, 8) -- (0, 8) -- cycle;
\draw [black, thick] (0.5, 0.5) -- (1.5, 0.5);
\draw [black, thick] (1.5, 0.5) -- (2.5, 0.5);
\draw [black, thick] (2.5, 0.5) -- (3.5, 0.5);
\draw [black, thick] (3.5, 0.5) -- (4.5, 0.5);
\draw [black, thick] (4.5, 0.5) -- (5.5, 0.5);
\draw [black, thick] (5.5, 0.5) -- (6.5, 0.5);
\draw [black, thick] (6.5, 0.5) -- (7.5, 0.5);
\draw [black, thick] (7.5, 0.5) -- (8.5, 0.5);
\draw [black, thick] (8.5, 0.5) -- (9.5, 0.5);
\draw [black, thick] (9.5, 0.5) -- (10.5, 0.5);
\draw [black, thick] (1.5, 1.5) -- (2.5, 1.5);
\draw [black, thick] (2.5, 1.5) -- (3.5, 1.5);
\draw [black, thick] (3.5, 1.5) -- (4.5, 1.5);
\draw [black, thick] (4.5, 1.5) -- (5.5, 1.5);
\draw [black, thick] (5.5, 1.5) -- (6.5, 1.5);
\draw [black, thick] (6.5, 1.5) -- (7.5, 1.5);
\draw [black, thick] (7.5, 1.5) -- (8.5, 1.5);
\draw [black, thick] (8.5, 1.5) -- (9.5, 1.5);
\draw [black, thick] (9.5, 1.5) -- (10.5, 1.5);
\draw [black, thick] (1.5, 2.5) -- (2.5, 2.5);
\draw [black, thick] (2.5, 2.5) -- (3.5, 2.5);
\draw [black, thick] (3.5, 2.5) -- (4.5, 2.5);
\draw [black, thick] (4.5, 2.5) -- (5.5, 2.5);
\draw [black, thick] (5.5, 2.5) -- (6.5, 2.5);
\draw [black, thick] (6.5, 2.5) -- (7.5, 2.5);
\draw [black, thick] (7.5, 2.5) -- (8.5, 2.5);
\draw [black, thick] (8.5, 2.5) -- (9.5, 2.5);
\draw [black, thick] (9.5, 2.5) -- (10.5, 2.5);
\draw [black, thick] (1.5, 3.5) -- (2.5, 3.5);
\draw [black, thick] (2.5, 3.5) -- (3.5, 3.5);
\draw [black, thick] (3.5, 3.5) -- (4.5, 3.5);
\draw [black, thick] (4.5, 3.5) -- (5.5, 3.5);
\draw [black, thick] (5.5, 3.5) -- (6.5, 3.5);
\draw [black, thick] (6.5, 3.5) -- (7.5, 3.5);
\draw [black, thick] (7.5, 3.5) -- (8.5, 3.5);
\draw [black, thick] (8.5, 3.5) -- (9.5, 3.5);
\draw [black, thick] (9.5, 3.5) -- (10.5, 3.5);
\draw [black, thick] (1.5, 4.5) -- (2.5, 4.5);
\draw [black, thick] (2.5, 4.5) -- (3.5, 4.5);
\draw [black, thick] (3.5, 4.5) -- (4.5, 4.5);
\draw [black, thick] (4.5, 4.5) -- (5.5, 4.5);
\draw [black, thick] (5.5, 4.5) -- (6.5, 4.5);
\draw [black, thick] (6.5, 4.5) -- (7.5, 4.5);
\draw [black, thick] (7.5, 4.5) -- (8.5, 4.5);
\draw [black, thick] (8.5, 4.5) -- (9.5, 4.5);
\draw [black, thick] (9.5, 4.5) -- (10.5, 4.5);
\draw [black, thick] (1.5, 5.5) -- (2.5, 5.5);
\draw [black, thick] (2.5, 5.5) -- (3.5, 5.5);
\draw [black, thick] (3.5, 5.5) -- (4.5, 5.5);
\draw [black, thick] (4.5, 5.5) -- (5.5, 5.5);
\draw [black, thick] (5.5, 5.5) -- (6.5, 5.5);
\draw [black, thick] (6.5, 5.5) -- (7.5, 5.5);
\draw [black, thick] (7.5, 5.5) -- (8.5, 5.5);
\draw [black, thick] (8.5, 5.5) -- (9.5, 5.5);
\draw [black, thick] (9.5, 5.5) -- (10.5, 5.5);
\draw [black, thick] (1.5, 6.5) -- (2.5, 6.5);
\draw [black, thick] (2.5, 6.5) -- (3.5, 6.5);
\draw [black, thick] (3.5, 6.5) -- (4.5, 6.5);
\draw [black, thick] (4.5, 6.5) -- (5.5, 6.5);
\draw [black, thick] (5.5, 6.5) -- (6.5, 6.5);
\draw [black, thick] (6.5, 6.5) -- (7.5, 6.5);
\draw [black, thick] (7.5, 6.5) -- (8.5, 6.5);
\draw [black, thick] (8.5, 6.5) -- (9.5, 6.5);
\draw [black, thick] (9.5, 6.5) -- (10.5, 6.5);
\draw [black, thick] (0.5, 7.5) -- (1.5, 7.5);
\draw [black, thick] (1.5, 7.5) -- (2.5, 7.5);
\draw [black, thick] (2.5, 7.5) -- (3.5, 7.5);
\draw [black, thick] (3.5, 7.5) -- (4.5, 7.5);
\draw [black, thick] (4.5, 7.5) -- (5.5, 7.5);
\draw [black, thick] (5.5, 7.5) -- (6.5, 7.5);
\draw [black, thick] (6.5, 7.5) -- (7.5, 7.5);
\draw [black, thick] (7.5, 7.5) -- (8.5, 7.5);
\draw [black, thick] (8.5, 7.5) -- (9.5, 7.5);
\draw [black, thick] (9.5, 7.5) -- (10.5, 7.5);
\draw [black, thick] (0.5, 0.5) -- (0.5, 1.5);
\draw [black, thick] (10.5, 0.5) -- (10.5, 1.5);
\draw [black, thick] (0.5, 1.5) -- (0.5, 2.5);
\draw [black, thick] (1.5, 1.5) -- (1.5, 2.5);
\draw [black, thick] (0.5, 2.5) -- (0.5, 3.5);
\draw [black, thick] (10.5, 2.5) -- (10.5, 3.5);
\draw [black, thick] (0.5, 3.5) -- (0.5, 4.5);
\draw [black, thick] (1.5, 3.5) -- (1.5, 4.5);
\draw [black, thick] (0.5, 4.5) -- (0.5, 5.5);
\draw [black, thick] (10.5, 4.5) -- (10.5, 5.5);
\draw [black, thick] (0.5, 5.5) -- (0.5, 6.5);
\draw [black, thick] (1.5, 5.5) -- (1.5, 6.5);
\draw [black, thick] (0.5, 6.5) -- (0.5, 7.5);
\draw [black, thick] (10.5, 6.5) -- (10.5, 7.5);
        \end{tikzpicture}
        \caption{$16$ turns in $G(11,8)$.}
        \label{11x8min}
    \end{figure}

If $n$ is odd, we can construct $2m$ by simply rotating the direction of the ``prong", with two turns in every column instead of row, shown in the below example.

    \begin{figure}[H]
        \centering
        \begin{tikzpicture}
\draw[step=1cm,gray,thin,dashed] (0,0) grid (10,7);
\draw [black, thin] (0, 0) -- (10, 0) -- (10, 7) -- (0, 7) -- cycle;
\draw [black, thick] (0.5, 0.5) -- (1.5, 0.5);
\draw [black, thick] (1.5, 0.5) -- (2.5, 0.5);
\draw [black, thick] (2.5, 0.5) -- (3.5, 0.5);
\draw [black, thick] (3.5, 0.5) -- (4.5, 0.5);
\draw [black, thick] (4.5, 0.5) -- (5.5, 0.5);
\draw [black, thick] (5.5, 0.5) -- (6.5, 0.5);
\draw [black, thick] (6.5, 0.5) -- (7.5, 0.5);
\draw [black, thick] (7.5, 0.5) -- (8.5, 0.5);
\draw [black, thick] (8.5, 0.5) -- (9.5, 0.5);
\draw [black, thick] (1.5, 1.5) -- (2.5, 1.5);
\draw [black, thick] (3.5, 1.5) -- (4.5, 1.5);
\draw [black, thick] (5.5, 1.5) -- (6.5, 1.5);
\draw [black, thick] (7.5, 1.5) -- (8.5, 1.5);
\draw [black, thick] (0.5, 6.5) -- (1.5, 6.5);
\draw [black, thick] (2.5, 6.5) -- (3.5, 6.5);
\draw [black, thick] (4.5, 6.5) -- (5.5, 6.5);
\draw [black, thick] (6.5, 6.5) -- (7.5, 6.5);
\draw [black, thick] (8.5, 6.5) -- (9.5, 6.5);
\draw [black, thick] (0.5, 0.5) -- (0.5, 1.5);
\draw [black, thick] (9.5, 0.5) -- (9.5, 1.5);
\draw [black, thick] (0.5, 1.5) -- (0.5, 2.5);
\draw [black, thick] (1.5, 1.5) -- (1.5, 2.5);
\draw [black, thick] (2.5, 1.5) -- (2.5, 2.5);
\draw [black, thick] (3.5, 1.5) -- (3.5, 2.5);
\draw [black, thick] (4.5, 1.5) -- (4.5, 2.5);
\draw [black, thick] (5.5, 1.5) -- (5.5, 2.5);
\draw [black, thick] (6.5, 1.5) -- (6.5, 2.5);
\draw [black, thick] (7.5, 1.5) -- (7.5, 2.5);
\draw [black, thick] (8.5, 1.5) -- (8.5, 2.5);
\draw [black, thick] (9.5, 1.5) -- (9.5, 2.5);
\draw [black, thick] (0.5, 2.5) -- (0.5, 3.5);
\draw [black, thick] (1.5, 2.5) -- (1.5, 3.5);
\draw [black, thick] (2.5, 2.5) -- (2.5, 3.5);
\draw [black, thick] (3.5, 2.5) -- (3.5, 3.5);
\draw [black, thick] (4.5, 2.5) -- (4.5, 3.5);
\draw [black, thick] (5.5, 2.5) -- (5.5, 3.5);
\draw [black, thick] (6.5, 2.5) -- (6.5, 3.5);
\draw [black, thick] (7.5, 2.5) -- (7.5, 3.5);
\draw [black, thick] (8.5, 2.5) -- (8.5, 3.5);
\draw [black, thick] (9.5, 2.5) -- (9.5, 3.5);
\draw [black, thick] (0.5, 3.5) -- (0.5, 4.5);
\draw [black, thick] (1.5, 3.5) -- (1.5, 4.5);
\draw [black, thick] (2.5, 3.5) -- (2.5, 4.5);
\draw [black, thick] (3.5, 3.5) -- (3.5, 4.5);
\draw [black, thick] (4.5, 3.5) -- (4.5, 4.5);
\draw [black, thick] (5.5, 3.5) -- (5.5, 4.5);
\draw [black, thick] (6.5, 3.5) -- (6.5, 4.5);
\draw [black, thick] (7.5, 3.5) -- (7.5, 4.5);
\draw [black, thick] (8.5, 3.5) -- (8.5, 4.5);
\draw [black, thick] (9.5, 3.5) -- (9.5, 4.5);
\draw [black, thick] (0.5, 4.5) -- (0.5, 5.5);
\draw [black, thick] (1.5, 4.5) -- (1.5, 5.5);
\draw [black, thick] (2.5, 4.5) -- (2.5, 5.5);
\draw [black, thick] (3.5, 4.5) -- (3.5, 5.5);
\draw [black, thick] (4.5, 4.5) -- (4.5, 5.5);
\draw [black, thick] (5.5, 4.5) -- (5.5, 5.5);
\draw [black, thick] (6.5, 4.5) -- (6.5, 5.5);
\draw [black, thick] (7.5, 4.5) -- (7.5, 5.5);
\draw [black, thick] (8.5, 4.5) -- (8.5, 5.5);
\draw [black, thick] (9.5, 4.5) -- (9.5, 5.5);
\draw [black, thick] (0.5, 5.5) -- (0.5, 6.5);
\draw [black, thick] (1.5, 5.5) -- (1.5, 6.5);
\draw [black, thick] (2.5, 5.5) -- (2.5, 6.5);
\draw [black, thick] (3.5, 5.5) -- (3.5, 6.5);
\draw [black, thick] (4.5, 5.5) -- (4.5, 6.5);
\draw [black, thick] (5.5, 5.5) -- (5.5, 6.5);
\draw [black, thick] (6.5, 5.5) -- (6.5, 6.5);
\draw [black, thick] (7.5, 5.5) -- (7.5, 6.5);
\draw [black, thick] (8.5, 5.5) -- (8.5, 6.5);
\draw [black, thick] (9.5, 5.5) -- (9.5, 6.5);
        \end{tikzpicture}
        \caption{$20$ turns in $G(10,7)$.}
        \label{10x7min}
    \end{figure}

These constructions prove the claim.
\end{proof}

The same constructions, but rotated, work for $m < n$. Hence we have proven Theorem \ref{minturns}.

The constructions given in Claim \ref{minconstruct} are by no means the only ones. Many other constructions exist, including a snake-like construction, similar to Figure \ref{examplegrid}, which has exactly two turns in every column. Jelliss \cite{jelliss} gives all such constructions for the $6 \times 6$ and $8 \times 8$ cases.

\section{Maximum turns in squares}\label{max}

The square case for maximum turns was solved by Beluhov \cite{beluhov}. We provide a significantly different proof for the square cases, which will build important methods for working on the general rectangle case.

Call a cell a {\it straight} if it is not a turn; maximising the number of turns is equivalent to minimising the number of straights.

We now come to the connection between maximum turns and minimum turns. Define the {\it square overlay} to be a set of $\frac{n^2}{4}$ cycles going between cells $(2i-1,2j-1),(2i,2j-1),(2i,2j)$ and $(2i-1,2j)$ for $1\le i,j\le \frac n2$.

    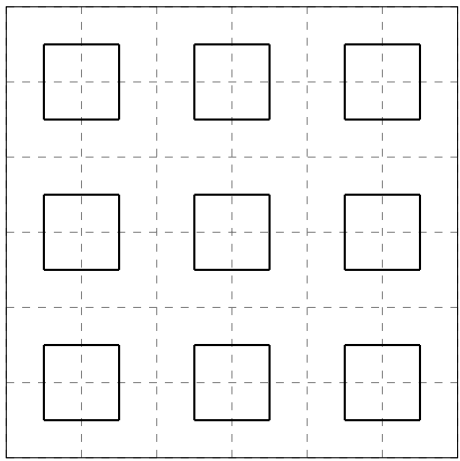
\begin{figure}[H]
        \centering
        \begin{tikzpicture}
\draw[step=1cm,gray,thin,dashed] (0,0) grid (6,6);
\draw [black, thin] (0, 0) -- (6, 0) -- (6, 6) -- (0, 6) -- cycle;
\draw [black, thick] (0.5, 0.5) -- (1.5, 0.5);
\draw [black, thick] (2.5, 0.5) -- (3.5, 0.5);
\draw [black, thick] (4.5, 0.5) -- (5.5, 0.5);
\draw [black, thick] (0.5, 1.5) -- (1.5, 1.5);
\draw [black, thick] (2.5, 1.5) -- (3.5, 1.5);
\draw [black, thick] (4.5, 1.5) -- (5.5, 1.5);
\draw [black, thick] (0.5, 2.5) -- (1.5, 2.5);
\draw [black, thick] (2.5, 2.5) -- (3.5, 2.5);
\draw [black, thick] (4.5, 2.5) -- (5.5, 2.5);
\draw [black, thick] (0.5, 3.5) -- (1.5, 3.5);
\draw [black, thick] (2.5, 3.5) -- (3.5, 3.5);
\draw [black, thick] (4.5, 3.5) -- (5.5, 3.5);
\draw [black, thick] (0.5, 4.5) -- (1.5, 4.5);
\draw [black, thick] (2.5, 4.5) -- (3.5, 4.5);
\draw [black, thick] (4.5, 4.5) -- (5.5, 4.5);
\draw [black, thick] (0.5, 5.5) -- (1.5, 5.5);
\draw [black, thick] (2.5, 5.5) -- (3.5, 5.5);
\draw [black, thick] (4.5, 5.5) -- (5.5, 5.5);
\draw [black, thick] (0.5, 0.5) -- (0.5, 1.5);
\draw [black, thick] (1.5, 0.5) -- (1.5, 1.5);
\draw [black, thick] (2.5, 0.5) -- (2.5, 1.5);
\draw [black, thick] (3.5, 0.5) -- (3.5, 1.5);
\draw [black, thick] (4.5, 0.5) -- (4.5, 1.5);
\draw [black, thick] (5.5, 0.5) -- (5.5, 1.5);
\draw [black, thick] (0.5, 2.5) -- (0.5, 3.5);
\draw [black, thick] (1.5, 2.5) -- (1.5, 3.5);
\draw [black, thick] (2.5, 2.5) -- (2.5, 3.5);
\draw [black, thick] (3.5, 2.5) -- (3.5, 3.5);
\draw [black, thick] (4.5, 2.5) -- (4.5, 3.5);
\draw [black, thick] (5.5, 2.5) -- (5.5, 3.5);
\draw [black, thick] (0.5, 4.5) -- (0.5, 5.5);
\draw [black, thick] (1.5, 4.5) -- (1.5, 5.5);
\draw [black, thick] (2.5, 4.5) -- (2.5, 5.5);
\draw [black, thick] (3.5, 4.5) -- (3.5, 5.5);
\draw [black, thick] (4.5, 4.5) -- (4.5, 5.5);
\draw [black, thick] (5.5, 4.5) -- (5.5, 5.5);
        \end{tikzpicture}
        \caption{A square overlay of $G(6,6)$}
        \label{6x6overlay}
    \end{figure}

Consider any Hamilton cycle in $G(n,n)$. We obtain the {\it reduced form} by taking the symmetric difference of each unit segment in the original cycle with that in the square overlay. Observe that this process is reversible.
    \begin{figure}[H]
        \centering
        \begin{tikzpicture}
\draw[step=0.5cm,gray,thin,dashed] (0,0) grid (3,3.0);
\draw [black, thin] (0, 0) -- (3, 0) -- (3, 3) -- (0, 3) -- cycle;
\draw[step=0.5cm,gray,thin,dashed] (3.5,0) grid (6.5,3);
\draw [black, thin] (3.5, 0) -- (6.5, 0) -- (6.5, 3) -- (3.5, 3) -- cycle;
\draw[step=0.5cm,gray,thin,dashed] (7,0) grid (10,3);
\draw [black, thin] (7, 0) -- (10, 0) -- (10, 3) -- (7, 3) -- cycle;

\draw [red, thick] (0.25, 0.25) -- (0.75, 0.25);
\draw [red, thick] (0.75, 0.25) -- (1.25, 0.25);
\draw [red, thick] (1.75, 0.25) -- (2.25, 0.25);
\draw [red, thick] (2.25, 0.25) -- (2.75, 0.25);
\draw [red, thick] (0.75, 0.75) -- (1.25, 0.75);
\draw [red, thick] (2.25, 0.75) -- (2.75, 0.75);
\draw [red, thick] (0.75, 1.25) -- (1.25, 1.25);
\draw [red, thick] (1.25, 1.25) -- (1.75, 1.25);
\draw [red, thick] (2.25, 1.25) -- (2.75, 1.25);
\draw [red, thick] (0.25, 1.75) -- (0.75, 1.75);
\draw [red, thick] (0.75, 1.75) -- (1.25, 1.75);
\draw [red, thick] (1.75, 1.75) -- (2.25, 1.75);
\draw [red, thick] (2.25, 1.75) -- (2.75, 1.75);
\draw [red, thick] (0.25, 2.25) -- (0.75, 2.25);
\draw [red, thick] (0.75, 2.25) -- (1.25, 2.25);
\draw [red, thick] (1.75, 2.25) -- (2.25, 2.25);
\draw [red, thick] (2.25, 2.25) -- (2.75, 2.25);
\draw [red, thick] (0.25, 2.75) -- (0.75, 2.75);
\draw [red, thick] (0.75, 2.75) -- (1.25, 2.75);
\draw [red, thick] (1.25, 2.75) -- (1.75, 2.75);
\draw [red, thick] (1.75, 2.75) -- (2.25, 2.75);
\draw [red, thick] (2.25, 2.75) -- (2.75, 2.75);
\draw [red, thick] (0.25, 0.25) -- (0.25, 0.75);
\draw [red, thick] (1.25, 0.25) -- (1.25, 0.75);
\draw [red, thick] (1.75, 0.25) -- (1.75, 0.75);
\draw [red, thick] (2.75, 0.25) -- (2.75, 0.75);
\draw [red, thick] (0.25, 0.75) -- (0.25, 1.25);
\draw [red, thick] (0.75, 0.75) -- (0.75, 1.25);
\draw [red, thick] (1.75, 0.75) -- (1.75, 1.25);
\draw [red, thick] (2.25, 0.75) -- (2.25, 1.25);
\draw [red, thick] (0.25, 1.25) -- (0.25, 1.75);
\draw [red, thick] (2.75, 1.25) -- (2.75, 1.75);
\draw [red, thick] (1.25, 1.75) -- (1.25, 2.25);
\draw [red, thick] (1.75, 1.75) -- (1.75, 2.25);
\draw [red, thick] (0.25, 2.25) -- (0.25, 2.75);
\draw [red, thick] (2.75, 2.25) -- (2.75, 2.75);

\draw [blue, thick] (3.75, 0.25) -- (4.25, 0.25);
\draw [blue, thick] (4.75, 0.25) -- (5.25, 0.25);
\draw [blue, thick] (5.75, 0.25) -- (6.25, 0.25);
\draw [blue, thick] (3.75, 0.75) -- (4.25, 0.75);
\draw [blue, thick] (4.75, 0.75) -- (5.25, 0.75);
\draw [blue, thick] (5.75, 0.75) -- (6.25, 0.75);
\draw [blue, thick] (3.75, 1.25) -- (4.25, 1.25);
\draw [blue, thick] (4.75, 1.25) -- (5.25, 1.25);
\draw [blue, thick] (5.75, 1.25) -- (6.25, 1.25);
\draw [blue, thick] (3.75, 1.75) -- (4.25, 1.75);
\draw [blue, thick] (4.75, 1.75) -- (5.25, 1.75);
\draw [blue, thick] (5.75, 1.75) -- (6.25, 1.75);
\draw [blue, thick] (3.75, 2.25) -- (4.25, 2.25);
\draw [blue, thick] (4.75, 2.25) -- (5.25, 2.25);
\draw [blue, thick] (5.75, 2.25) -- (6.25, 2.25);
\draw [blue, thick] (3.75, 2.75) -- (4.25, 2.75);
\draw [blue, thick] (4.75, 2.75) -- (5.25, 2.75);
\draw [blue, thick] (5.75, 2.75) -- (6.25, 2.75);
\draw [blue, thick] (3.75, 0.25) -- (3.75, 0.75);
\draw [blue, thick] (4.25, 0.25) -- (4.25, 0.75);
\draw [blue, thick] (4.75, 0.25) -- (4.75, 0.75);
\draw [blue, thick] (5.25, 0.25) -- (5.25, 0.75);
\draw [blue, thick] (5.75, 0.25) -- (5.75, 0.75);
\draw [blue, thick] (6.25, 0.25) -- (6.25, 0.75);
\draw [blue, thick] (3.75, 1.25) -- (3.75, 1.75);
\draw [blue, thick] (4.25, 1.25) -- (4.25, 1.75);
\draw [blue, thick] (4.75, 1.25) -- (4.75, 1.75);
\draw [blue, thick] (5.25, 1.25) -- (5.25, 1.75);
\draw [blue, thick] (5.75, 1.25) -- (5.75, 1.75);
\draw [blue, thick] (6.25, 1.25) -- (6.25, 1.75);
\draw [blue, thick] (3.75, 2.25) -- (3.75, 2.75);
\draw [blue, thick] (4.25, 2.25) -- (4.25, 2.75);
\draw [blue, thick] (4.75, 2.25) -- (4.75, 2.75);
\draw [blue, thick] (5.25, 2.25) -- (5.25, 2.75);
\draw [blue, thick] (5.75, 2.25) -- (5.75, 2.75);
\draw [blue, thick] (6.25, 2.25) -- (6.25, 2.75);

\draw [red, thick] (7.75, 0.25) -- (8.25, 0.25);
\draw [red, thick] (8.75, 0.25) -- (9.25, 0.25);
\draw [red, thick] (7.75, 0.75) -- (8.25, 0.75);
\draw [red, thick] (7.75, 1.25) -- (8.25, 1.25);
\draw [red, thick] (7.75, 1.75) -- (8.25, 1.75);
\draw [red, thick] (8.75, 1.75) -- (9.25, 1.75);
\draw [red, thick] (7.75, 2.25) -- (8.25, 2.25);
\draw [red, thick] (8.75, 2.25) -- (9.25, 2.25);
\draw [red, thick] (7.75, 2.75) -- (8.25, 2.75);
\draw [red, thick] (8.75, 2.75) -- (9.25, 2.75);
\draw [red, thick] (7.25, 0.75) -- (7.25, 1.25);
\draw [red, thick] (7.75, 0.75) -- (7.75, 1.25);
\draw [red, thick] (8.75, 0.75) -- (8.75, 1.25);
\draw [red, thick] (9.25, 0.75) -- (9.25, 1.25);
\draw [red, thick] (8.25, 1.75) -- (8.25, 2.25);
\draw [red, thick] (8.75, 1.75) -- (8.75, 2.25);

\draw [blue, thick] (8.25, 0.25) -- (8.75, 0.25);
\draw [blue, thick] (7.25, 0.75) -- (7.75, 0.75);
\draw [blue, thick] (8.25, 0.75) -- (8.75, 0.75);
\draw [blue, thick] (7.25, 1.25) -- (7.75, 1.25);
\draw [blue, thick] (8.25, 1.75) -- (8.75, 1.75);
\draw [blue, thick] (8.25, 2.25) -- (8.75, 2.25);
\draw [blue, thick] (7.75, 0.25) -- (7.75, 0.75);
\draw [blue, thick] (9.25, 0.25) -- (9.25, 0.75);
\draw [blue, thick] (7.75, 1.25) -- (7.75, 1.75);
\draw [blue, thick] (8.25, 1.25) -- (8.25, 1.75);
\draw [blue, thick] (8.75, 1.25) -- (8.75, 1.75);
\draw [blue, thick] (9.25, 1.25) -- (9.25, 1.75);
\draw [blue, thick] (7.75, 2.25) -- (7.75, 2.75);
\draw [blue, thick] (8.25, 2.25) -- (8.25, 2.75);
\draw [blue, thick] (8.75, 2.25) -- (8.75, 2.75);
\draw [blue, thick] (9.25, 2.25) -- (9.25, 2.75);

\draw [OliveGreen, thin, dashed] (7.25, 0.25) -- (7.75, 0.25);
\draw [OliveGreen, thin, dashed] (9.25, 0.25) -- (9.75, 0.25);
\draw [OliveGreen, thin, dashed] (9.25, 0.75) -- (9.75, 0.75);
\draw [OliveGreen, thin, dashed] (8.25, 1.25) -- (8.75, 1.25);
\draw [OliveGreen, thin, dashed] (9.25, 1.25) -- (9.75, 1.25);
\draw [OliveGreen, thin, dashed] (7.25, 1.75) -- (7.75, 1.75);
\draw [OliveGreen, thin, dashed] (9.25, 1.75) -- (9.75, 1.75);
\draw [OliveGreen, thin, dashed] (7.25, 2.25) -- (7.75, 2.25);
\draw [OliveGreen, thin, dashed] (9.25, 2.25) -- (9.75, 2.25);
\draw [OliveGreen, thin, dashed] (7.25, 2.75) -- (7.75, 2.75);
\draw [OliveGreen, thin, dashed] (8.25, 2.75) -- (8.75, 2.75);
\draw [OliveGreen, thin, dashed] (9.25, 2.75) -- (9.75, 2.75);
\draw [OliveGreen, thin, dashed] (7.25, 0.25) -- (7.25, 0.75);
\draw [OliveGreen, thin, dashed] (8.25, 0.25) -- (8.25, 0.75);
\draw [OliveGreen, thin, dashed] (8.75, 0.25) -- (8.75, 0.75);
\draw [OliveGreen, thin, dashed] (9.75, 0.25) -- (9.75, 0.75);
\draw [OliveGreen, thin, dashed] (7.25, 1.25) -- (7.25, 1.75);
\draw [OliveGreen, thin, dashed] (9.75, 1.25) -- (9.75, 1.75);
\draw [OliveGreen, thin, dashed] (7.25, 2.25) -- (7.25, 2.75);
\draw [OliveGreen, thin, dashed] (9.75, 2.25) -- (9.75, 2.75);

\node at (3.25,1.5) {$+$};

\node at (6.75,1.5) {$=$};

        \end{tikzpicture}
        \caption{Obtaining the reduced form of a $6 \times 6$ cycle. Segments where red and blue segments were cancelled are shown in dashed green on the right.}
        \label{reducedform}
    \end{figure}
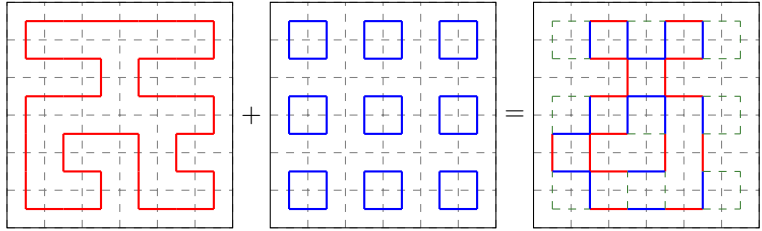
If the reduced form intersects itself, treat it as two straight double-segments over the same cell. For example, the reduced form shown in Figure \ref{reducedform} has $14$ turns; we do not count cells which act as an endpoint of four segments as turns.
\begin{lemma}\label{unionofcycles}
The reduced form is a union of cycles.
\end{lemma}
\begin{proof}
The number of segments going into or out from any cell $c$ is invariant modulo $2$ under the symmetric difference with any set of cycles, and hence all cells have even degree.
\end{proof}
\begin{lemma}\label{turnstostraights}
A turn in the reduced form creates a straight in the original cycle.
\end{lemma}
\begin{proof}
Consider the position of the turn, relative to a single cycle in the square overlay. If it does not create a straight, then it can be easily be checked that the original cycle is not Hamiltonian; it will either cross over some cell twice, or not contain some cell.
\end{proof}

Let us view $G(n,n)$ as an $\frac n2 \times \frac n2$ grid of $2 \times 2$ squares. For all cycles in the reduced form, round their vertices to the centre of the nearest $2 \times 2$ square. This gives a union of cycles (some of which may turn twice in the same cell or run over the same segment twice) on $G\left(\frac n2, \frac n2 \right)$; call this the {\it half-form}.

    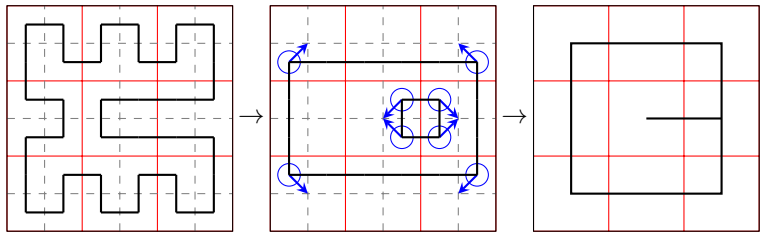
\begin{figure}[H]
        \centering
        \begin{tikzpicture}
\draw[step=1cm,red,thin] (0,0) grid (3,3.0);
\draw[step=1cm,gray,thin,dashed,shift={(-0.5,-0.5)}] (0.5,0.5) grid (3.5,3.5);
\draw [black, thin] (0, 0) -- (3, 0) -- (3, 3) -- (0, 3) -- cycle;
\draw[step=1cm,red,thin,shift={(0.5,0)}] (3,0) grid (6,3);
\draw[step=1cm,gray,thin,dashed,shift={(0,-0.5)}] (3.5,0.5) grid (6.5,3.5);
\draw [black, thin] (3.5, 0) -- (6.5, 0) -- (6.5, 3) -- (3.5, 3) -- cycle;
\draw[step=1cm,red,thin] (7,0) grid (10,3);

\draw [blue,thin] (3.75, 0.75) circle (0.15);
\draw [blue,thick,->,>=stealth] (3.75,0.75) -- (4,0.5);
\draw [blue,thin] (6.25, 0.75) circle (0.15);
\draw [blue,thick,->,>=stealth] (6.25,0.75) -- (6,0.5);
\draw [blue,thin] (5.25, 1.25) circle (0.15);
\draw [blue,thick,->,>=stealth] (5.25,1.25) -- (5,1.5);
\draw [blue,thin] (5.75, 1.25) circle (0.15);
\draw [blue,thick,->,>=stealth] (5.75,1.25) -- (6,1.5);
\draw [blue,thin] (5.25, 1.75) circle (0.15);
\draw [blue,thick,->,>=stealth] (5.25,1.75) -- (5,1.5);
\draw [blue,thin] (5.75, 1.75) circle (0.15);
\draw [blue,thick,->,>=stealth] (5.75,1.75) -- (6,1.5);
\draw [blue,thin] (3.75, 2.25) circle (0.15);
\draw [blue,thick,->,>=stealth] (3.75,2.25) -- (4,2.5);
\draw [blue,thin] (6.25, 2.25) circle (0.15);
\draw [blue,thick,->,>=stealth] (6.25,2.25) -- (6,2.5);

\draw [black, thin] (7, 0) -- (10, 0) -- (10, 3) -- (7, 3) -- cycle;
\draw [black, thick] (0.25, 0.25) -- (0.75, 0.25);
\draw [black, thick] (1.25, 0.25) -- (1.75, 0.25);
\draw [black, thick] (2.25, 0.25) -- (2.75, 0.25);
\draw [black, thick] (0.75, 0.75) -- (1.25, 0.75);
\draw [black, thick] (1.75, 0.75) -- (2.25, 0.75);
\draw [black, thick] (3.75, 0.75) -- (4.25, 0.75);
\draw [black, thick] (4.25, 0.75) -- (4.75, 0.75);
\draw [black, thick] (4.75, 0.75) -- (5.25, 0.75);
\draw [black, thick] (5.25, 0.75) -- (5.75, 0.75);
\draw [black, thick] (5.75, 0.75) -- (6.25, 0.75);
\draw [black, thick] (0.25, 1.25) -- (0.75, 1.25);
\draw [black, thick] (1.25, 1.25) -- (1.75, 1.25);
\draw [black, thick] (1.75, 1.25) -- (2.25, 1.25);
\draw [black, thick] (2.25, 1.25) -- (2.75, 1.25);
\draw [black, thick] (5.25, 1.25) -- (5.75, 1.25);
\draw [black, thick] (0.25, 1.75) -- (0.75, 1.75);
\draw [black, thick] (1.25, 1.75) -- (1.75, 1.75);
\draw [black, thick] (1.75, 1.75) -- (2.25, 1.75);
\draw [black, thick] (2.25, 1.75) -- (2.75, 1.75);
\draw [black, thick] (5.25, 1.75) -- (5.75, 1.75);
\draw [black, thick] (0.75, 2.25) -- (1.25, 2.25);
\draw [black, thick] (1.75, 2.25) -- (2.25, 2.25);
\draw [black, thick] (3.75, 2.25) -- (4.25, 2.25);
\draw [black, thick] (4.25, 2.25) -- (4.75, 2.25);
\draw [black, thick] (4.75, 2.25) -- (5.25, 2.25);
\draw [black, thick] (5.25, 2.25) -- (5.75, 2.25);
\draw [black, thick] (5.75, 2.25) -- (6.25, 2.25);
\draw [black, thick] (0.25, 2.75) -- (0.75, 2.75);
\draw [black, thick] (1.25, 2.75) -- (1.75, 2.75);
\draw [black, thick] (2.25, 2.75) -- (2.75, 2.75);
\draw [black, thick] (0.25, 0.25) -- (0.25, 0.75);
\draw [black, thick] (0.75, 0.25) -- (0.75, 0.75);
\draw [black, thick] (1.25, 0.25) -- (1.25, 0.75);
\draw [black, thick] (1.75, 0.25) -- (1.75, 0.75);
\draw [black, thick] (2.25, 0.25) -- (2.25, 0.75);
\draw [black, thick] (2.75, 0.25) -- (2.75, 0.75);
\draw [black, thick] (0.25, 0.75) -- (0.25, 1.25);
\draw [black, thick] (2.75, 0.75) -- (2.75, 1.25);
\draw [black, thick] (3.75, 0.75) -- (3.75, 1.25);
\draw [black, thick] (6.25, 0.75) -- (6.25, 1.25);
\draw [black, thick] (0.75, 1.25) -- (0.75, 1.75);
\draw [black, thick] (1.25, 1.25) -- (1.25, 1.75);
\draw [black, thick] (3.75, 1.25) -- (3.75, 1.75);
\draw [black, thick] (5.25, 1.25) -- (5.25, 1.75);
\draw [black, thick] (5.75, 1.25) -- (5.75, 1.75);
\draw [black, thick] (6.25, 1.25) -- (6.25, 1.75);
\draw [black, thick] (0.25, 1.75) -- (0.25, 2.25);
\draw [black, thick] (2.75, 1.75) -- (2.75, 2.25);
\draw [black, thick] (3.75, 1.75) -- (3.75, 2.25);
\draw [black, thick] (6.25, 1.75) -- (6.25, 2.25);
\draw [black, thick] (0.25, 2.25) -- (0.25, 2.75);
\draw [black, thick] (0.75, 2.25) -- (0.75, 2.75);
\draw [black, thick] (1.25, 2.25) -- (1.25, 2.75);
\draw [black, thick] (1.75, 2.25) -- (1.75, 2.75);
\draw [black, thick] (2.25, 2.25) -- (2.25, 2.75);
\draw [black, thick] (2.75, 2.25) -- (2.75, 2.75);

\draw [black, thick] (7.5, 0.5) -- (9.5, 0.5) -- (9.5, 2.5) -- (7.5, 2.5) -- cycle;
\draw [black, thick] (8.5,1.5) -- (9.5,1.5);

\node at (3.25,1.5) {$\to$};

\node at (6.75,1.5) {$\to$};

        \end{tikzpicture}
        \caption{Converting a Hamilton cycle of $G(6,6)$ into its reduced form, and then into its half-form. In the half-form, the centre cell and the cell to its right have two turns each, the segment between them is traversed twice.}
        \label{half-form_example}
    \end{figure}

\begin{lemma}\label{hamiltonianhalfform}
The half form goes through every cell.
\end{lemma}
\begin{proof}
If it does not pass through a cell, then the reduced form does not pass through a $2 \times 2$ area, and the original cycle has two components (one of which is a $2 \times 2$ cycle).
\end{proof}

By Corollary \ref{2turns}, there are at least $2 \times \frac n2 = n$ turns in the half-form, and hence at least $n$ straights in the original cycle.

Motivated by the constructions given in Claim \ref{minconstruct}, it is in fact possible construct many cycles of an $n \times n$ grid with exactly $n$ straights in the case that $4 | n$, by converting a cycle with minimal turns from the half-form into a reduced form. We demonstrate this below.

\begin{claim}\label{square4|n}
If $4|n$, the maximal number of turns of any Hamilton cycle of $G(n,n)$ is exactly $n^2 - n$.
\end{claim}
\begin{proof}
From the previous discussion, there are at least $n$ straights. It remains to show that this is attainable.

Consider any Hamilton cycle $C$ of $G\left(\frac n2, \frac n2 \right)$, which has exactly $n$ turns (existence of such a cycle was proven in Claim \ref{minconstruct}). This cycle will be the half-form of the final cycle.

    \begin{figure}[H]
        \centering
        \begin{tikzpicture}
\draw[step=1cm,gray,thin,dashed] (0,0) grid (6,6);
\draw [black, thin] (0, 0) -- (6, 0) -- (6, 6) -- (0, 6) -- cycle;
\draw [black, thick] (0.5, 0.5) -- (1.5, 0.5);
\draw [black, thick] (1.5, 0.5) -- (2.5, 0.5);
\draw [black, thick] (2.5, 0.5) -- (3.5, 0.5);
\draw [black, thick] (3.5, 0.5) -- (4.5, 0.5);
\draw [black, thick] (4.5, 0.5) -- (5.5, 0.5);
\draw [black, thick] (0.5, 1.5) -- (1.5, 1.5);
\draw [black, thick] (1.5, 1.5) -- (2.5, 1.5);
\draw [black, thick] (2.5, 1.5) -- (3.5, 1.5);
\draw [black, thick] (3.5, 1.5) -- (4.5, 1.5);
\draw [black, thick] (0.5, 2.5) -- (1.5, 2.5);
\draw [black, thick] (1.5, 2.5) -- (2.5, 2.5);
\draw [black, thick] (2.5, 2.5) -- (3.5, 2.5);
\draw [black, thick] (3.5, 2.5) -- (4.5, 2.5);
\draw [black, thick] (1.5, 3.5) -- (2.5, 3.5);
\draw [black, thick] (2.5, 3.5) -- (3.5, 3.5);
\draw [black, thick] (3.5, 3.5) -- (4.5, 3.5);
\draw [black, thick] (4.5, 3.5) -- (5.5, 3.5);
\draw [black, thick] (1.5, 4.5) -- (2.5, 4.5);
\draw [black, thick] (2.5, 4.5) -- (3.5, 4.5);
\draw [black, thick] (3.5, 4.5) -- (4.5, 4.5);
\draw [black, thick] (4.5, 4.5) -- (5.5, 4.5);
\draw [black, thick] (0.5, 5.5) -- (1.5, 5.5);
\draw [black, thick] (1.5, 5.5) -- (2.5, 5.5);
\draw [black, thick] (2.5, 5.5) -- (3.5, 5.5);
\draw [black, thick] (3.5, 5.5) -- (4.5, 5.5);
\draw [black, thick] (4.5, 5.5) -- (5.5, 5.5);
\draw [black, thick] (0.5, 0.5) -- (0.5, 1.5);
\draw [black, thick] (5.5, 0.5) -- (5.5, 1.5);
\draw [black, thick] (4.5, 1.5) -- (4.5, 2.5);
\draw [black, thick] (5.5, 1.5) -- (5.5, 2.5);
\draw [black, thick] (0.5, 2.5) -- (0.5, 3.5);
\draw [black, thick] (5.5, 2.5) -- (5.5, 3.5);
\draw [black, thick] (0.5, 3.5) -- (0.5, 4.5);
\draw [black, thick] (1.5, 3.5) -- (1.5, 4.5);
\draw [black, thick] (0.5, 4.5) -- (0.5, 5.5);
\draw [black, thick] (5.5, 4.5) -- (5.5, 5.5);

        \end{tikzpicture}
        \caption{A cycle on $G(6,6)$ with $12$ turns, the minimum possible.}
        \label{6x6min}
    \end{figure}
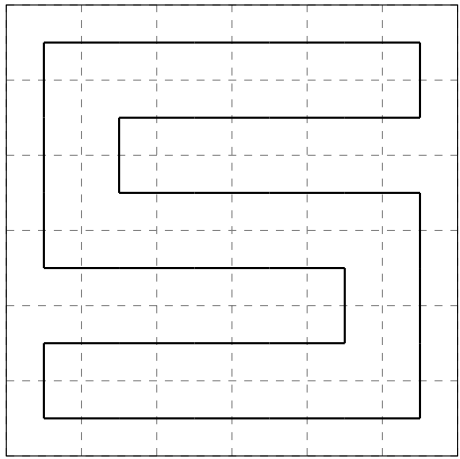

We now dilate $C$ from $(0,0)$ by a factor of $2$. This yields a cycle $C'$ in $G(n,n)$. Orient $C'$ clockwise. Every turn has a vertical and a horizontal segment. The orientation assigns each vertex one of four types, depending on whether the vertical segment is oriented up or down, and whether the horizontal segment is oriented left or right.

For each turn, reduce its $x$-coordinate by $1$ if its vertical segment is oriented up, and reduce its $y$-coordinate by $1$ if its horizontal segment is oriented right. Adjacent turns in still are in the same row or column, since either their horizontal or vertical orientation are the same. So this process gives a new cycle, $C''$, whose turns each share exactly one segment with the square overlay of the $n \times n$ grid.
    \begin{figure}[H]
        \centering
        \begin{tikzpicture}
\draw[step=0.5cm,gray,thin,dashed] (0,0) grid (6.0,6.0);
\draw [black, thin] (0, 0) -- (6.0, 0) -- (6.0, 6.0) -- (0, 6.0) -- cycle;
\draw [black, thick] (0.25, 0.75) -- (0.75, 0.75);
\draw [black, thick] (0.75, 0.75) -- (1.25, 0.75);
\draw [black, thick] (1.25, 0.75) -- (1.75, 0.75);
\draw [black, thick] (1.75, 0.75) -- (2.25, 0.75);
\draw [black, thick] (2.25, 0.75) -- (2.75, 0.75);
\draw [black, thick] (2.75, 0.75) -- (3.25, 0.75);
\draw [black, thick] (3.25, 0.75) -- (3.75, 0.75);
\draw [black, thick] (3.75, 0.75) -- (4.25, 0.75);
\draw [black, thick] (4.25, 0.75) -- (4.75, 0.75);
\draw [black, thick] (4.75, 0.75) -- (5.25, 0.75);
\draw [black, thick] (5.25, 0.75) -- (5.75, 0.75);
\draw [black, thick] (0.25, 1.25) -- (0.75, 1.25);
\draw [black, thick] (0.75, 1.25) -- (1.25, 1.25);
\draw [black, thick] (1.25, 1.25) -- (1.75, 1.25);
\draw [black, thick] (1.75, 1.25) -- (2.25, 1.25);
\draw [black, thick] (2.25, 1.25) -- (2.75, 1.25);
\draw [black, thick] (2.75, 1.25) -- (3.25, 1.25);
\draw [black, thick] (3.25, 1.25) -- (3.75, 1.25);
\draw [black, thick] (3.75, 1.25) -- (4.25, 1.25);
\draw [black, thick] (0.25, 2.75) -- (0.75, 2.75);
\draw [black, thick] (0.75, 2.75) -- (1.25, 2.75);
\draw [black, thick] (1.25, 2.75) -- (1.75, 2.75);
\draw [black, thick] (1.75, 2.75) -- (2.25, 2.75);
\draw [black, thick] (2.25, 2.75) -- (2.75, 2.75);
\draw [black, thick] (2.75, 2.75) -- (3.25, 2.75);
\draw [black, thick] (3.25, 2.75) -- (3.75, 2.75);
\draw [black, thick] (3.75, 2.75) -- (4.25, 2.75);
\draw [black, thick] (1.75, 3.25) -- (2.25, 3.25);
\draw [black, thick] (2.25, 3.25) -- (2.75, 3.25);
\draw [black, thick] (2.75, 3.25) -- (3.25, 3.25);
\draw [black, thick] (3.25, 3.25) -- (3.75, 3.25);
\draw [black, thick] (3.75, 3.25) -- (4.25, 3.25);
\draw [black, thick] (4.25, 3.25) -- (4.75, 3.25);
\draw [black, thick] (4.75, 3.25) -- (5.25, 3.25);
\draw [black, thick] (5.25, 3.25) -- (5.75, 3.25);
\draw [black, thick] (1.75, 4.75) -- (2.25, 4.75);
\draw [black, thick] (2.25, 4.75) -- (2.75, 4.75);
\draw [black, thick] (2.75, 4.75) -- (3.25, 4.75);
\draw [black, thick] (3.25, 4.75) -- (3.75, 4.75);
\draw [black, thick] (3.75, 4.75) -- (4.25, 4.75);
\draw [black, thick] (4.25, 4.75) -- (4.75, 4.75);
\draw [black, thick] (4.75, 4.75) -- (5.25, 4.75);
\draw [black, thick] (5.25, 4.75) -- (5.75, 4.75);
\draw [black, thick] (0.25, 5.25) -- (0.75, 5.25);
\draw [black, thick] (0.75, 5.25) -- (1.25, 5.25);
\draw [black, thick] (1.25, 5.25) -- (1.75, 5.25);
\draw [black, thick] (1.75, 5.25) -- (2.25, 5.25);
\draw [black, thick] (2.25, 5.25) -- (2.75, 5.25);
\draw [black, thick] (2.75, 5.25) -- (3.25, 5.25);
\draw [black, thick] (3.25, 5.25) -- (3.75, 5.25);
\draw [black, thick] (3.75, 5.25) -- (4.25, 5.25);
\draw [black, thick] (4.25, 5.25) -- (4.75, 5.25);
\draw [black, thick] (4.75, 5.25) -- (5.25, 5.25);
\draw [black, thick] (5.25, 5.25) -- (5.75, 5.25);
\draw [black, thick] (0.25, 0.75) -- (0.25, 1.25);
\draw [black, thick] (5.75, 0.75) -- (5.75, 1.25);
\draw [black, thick] (4.25, 1.25) -- (4.25, 1.75);
\draw [black, thick] (5.75, 1.25) -- (5.75, 1.75);
\draw [black, thick] (4.25, 1.75) -- (4.25, 2.25);
\draw [black, thick] (5.75, 1.75) -- (5.75, 2.25);
\draw [black, thick] (4.25, 2.25) -- (4.25, 2.75);
\draw [black, thick] (5.75, 2.25) -- (5.75, 2.75);
\draw [black, thick] (0.25, 2.75) -- (0.25, 3.25);
\draw [black, thick] (5.75, 2.75) -- (5.75, 3.25);
\draw [black, thick] (0.25, 3.25) -- (0.25, 3.75);
\draw [black, thick] (1.75, 3.25) -- (1.75, 3.75);
\draw [black, thick] (0.25, 3.75) -- (0.25, 4.25);
\draw [black, thick] (1.75, 3.75) -- (1.75, 4.25);
\draw [black, thick] (0.25, 4.25) -- (0.25, 4.75);
\draw [black, thick] (1.75, 4.25) -- (1.75, 4.75);
\draw [black, thick] (0.25, 4.75) -- (0.25, 5.25);
\draw [black, thick] (5.75, 4.75) -- (5.75, 5.25);

        \end{tikzpicture}
        \caption{The cycle $C''$, taking the cycle from Figure \ref{6x6min} as $C$.}
        \label{6x6C''}
    \end{figure}
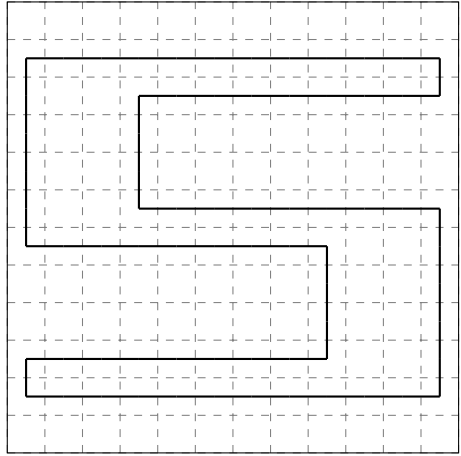

We now take the symmetric difference of $C''$ and a square overlay of $G(n,n)$, to obtain a final cycle $H$. Every turn in $C''$ yields a straight in $H$, and $H$ is Hamiltonian --- it joins the $2 \times 2$ squares that form the $n \times n$ grid in the same order that $C''$ does.
    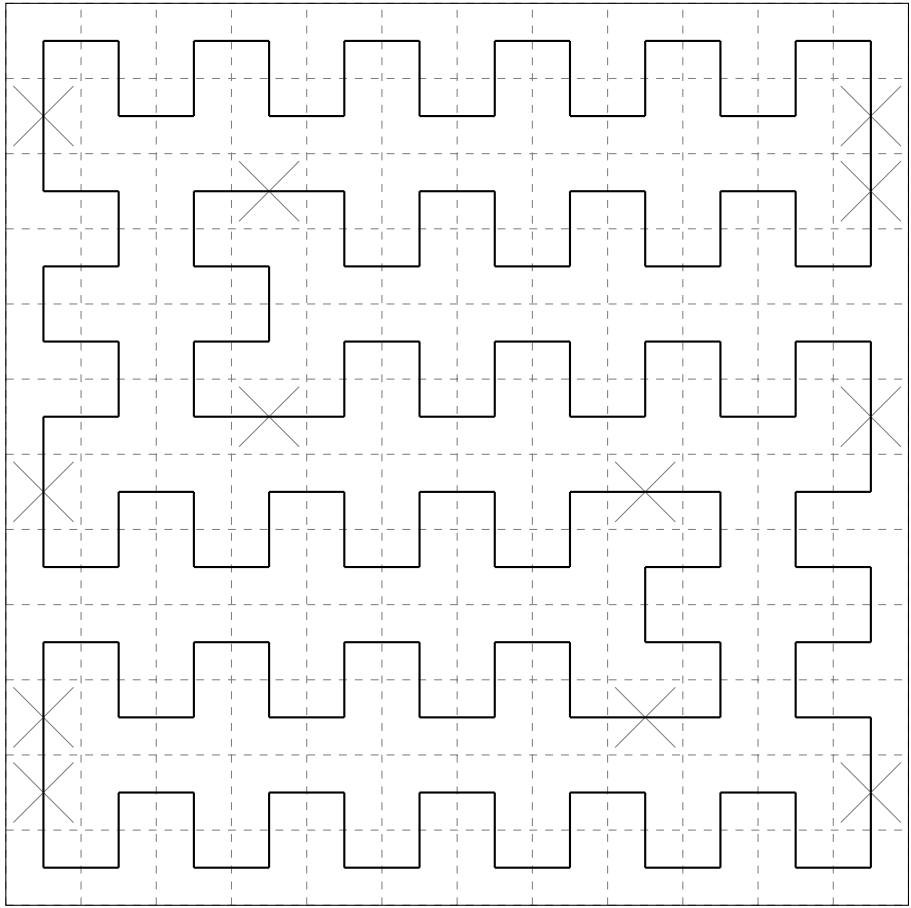
\begin{figure}[H]
        \centering
        \begin{tikzpicture}
\draw[step=1cm,gray,thin,dashed] (0,0) grid (12,12);
\draw [black, thin] (0, 0) -- (12, 0) -- (12, 12) -- (0, 12) -- cycle;
\draw [gray, thin] (0.1, 1.1) -- (0.9, 1.9);
\draw [gray, thin] (0.9, 1.1) -- (0.1, 1.9);
\draw [gray, thin] (11.1, 1.1) -- (11.9, 1.9);
\draw [gray, thin] (11.9, 1.1) -- (11.1, 1.9);
\draw [gray, thin] (0.1, 2.1) -- (0.9, 2.9);
\draw [gray, thin] (0.9, 2.1) -- (0.1, 2.9);
\draw [gray, thin] (8.1, 2.1) -- (8.9, 2.9);
\draw [gray, thin] (8.9, 2.1) -- (8.1, 2.9);
\draw [gray, thin] (0.1, 5.1) -- (0.9, 5.9);
\draw [gray, thin] (0.9, 5.1) -- (0.1, 5.9);
\draw [gray, thin] (8.1, 5.1) -- (8.9, 5.9);
\draw [gray, thin] (8.9, 5.1) -- (8.1, 5.9);
\draw [gray, thin] (3.1, 6.1) -- (3.9, 6.9);
\draw [gray, thin] (3.9, 6.1) -- (3.1, 6.9);
\draw [gray, thin] (11.1, 6.1) -- (11.9, 6.9);
\draw [gray, thin] (11.9, 6.1) -- (11.1, 6.9);
\draw [gray, thin] (3.1, 9.1) -- (3.9, 9.9);
\draw [gray, thin] (3.9, 9.1) -- (3.1, 9.9);
\draw [gray, thin] (11.1, 9.1) -- (11.9, 9.9);
\draw [gray, thin] (11.9, 9.1) -- (11.1, 9.9);
\draw [gray, thin] (0.1, 10.1) -- (0.9, 10.9);
\draw [gray, thin] (0.9, 10.1) -- (0.1, 10.9);
\draw [gray, thin] (11.1, 10.1) -- (11.9, 10.9);
\draw [gray, thin] (11.9, 10.1) -- (11.1, 10.9);
\draw [black, thick] (0.5, 0.5) -- (1.5, 0.5);
\draw [black, thick] (2.5, 0.5) -- (3.5, 0.5);
\draw [black, thick] (4.5, 0.5) -- (5.5, 0.5);
\draw [black, thick] (6.5, 0.5) -- (7.5, 0.5);
\draw [black, thick] (8.5, 0.5) -- (9.5, 0.5);
\draw [black, thick] (10.5, 0.5) -- (11.5, 0.5);
\draw [black, thick] (1.5, 1.5) -- (2.5, 1.5);
\draw [black, thick] (3.5, 1.5) -- (4.5, 1.5);
\draw [black, thick] (5.5, 1.5) -- (6.5, 1.5);
\draw [black, thick] (7.5, 1.5) -- (8.5, 1.5);
\draw [black, thick] (9.5, 1.5) -- (10.5, 1.5);
\draw [black, thick] (1.5, 2.5) -- (2.5, 2.5);
\draw [black, thick] (3.5, 2.5) -- (4.5, 2.5);
\draw [black, thick] (5.5, 2.5) -- (6.5, 2.5);
\draw [black, thick] (7.5, 2.5) -- (8.5, 2.5);
\draw [black, thick] (8.5, 2.5) -- (9.5, 2.5);
\draw [black, thick] (10.5, 2.5) -- (11.5, 2.5);
\draw [black, thick] (0.5, 3.5) -- (1.5, 3.5);
\draw [black, thick] (2.5, 3.5) -- (3.5, 3.5);
\draw [black, thick] (4.5, 3.5) -- (5.5, 3.5);
\draw [black, thick] (6.5, 3.5) -- (7.5, 3.5);
\draw [black, thick] (8.5, 3.5) -- (9.5, 3.5);
\draw [black, thick] (10.5, 3.5) -- (11.5, 3.5);
\draw [black, thick] (0.5, 4.5) -- (1.5, 4.5);
\draw [black, thick] (2.5, 4.5) -- (3.5, 4.5);
\draw [black, thick] (4.5, 4.5) -- (5.5, 4.5);
\draw [black, thick] (6.5, 4.5) -- (7.5, 4.5);
\draw [black, thick] (8.5, 4.5) -- (9.5, 4.5);
\draw [black, thick] (10.5, 4.5) -- (11.5, 4.5);
\draw [black, thick] (1.5, 5.5) -- (2.5, 5.5);
\draw [black, thick] (3.5, 5.5) -- (4.5, 5.5);
\draw [black, thick] (5.5, 5.5) -- (6.5, 5.5);
\draw [black, thick] (7.5, 5.5) -- (8.5, 5.5);
\draw [black, thick] (8.5, 5.5) -- (9.5, 5.5);
\draw [black, thick] (10.5, 5.5) -- (11.5, 5.5);
\draw [black, thick] (0.5, 6.5) -- (1.5, 6.5);
\draw [black, thick] (2.5, 6.5) -- (3.5, 6.5);
\draw [black, thick] (3.5, 6.5) -- (4.5, 6.5);
\draw [black, thick] (5.5, 6.5) -- (6.5, 6.5);
\draw [black, thick] (7.5, 6.5) -- (8.5, 6.5);
\draw [black, thick] (9.5, 6.5) -- (10.5, 6.5);
\draw [black, thick] (0.5, 7.5) -- (1.5, 7.5);
\draw [black, thick] (2.5, 7.5) -- (3.5, 7.5);
\draw [black, thick] (4.5, 7.5) -- (5.5, 7.5);
\draw [black, thick] (6.5, 7.5) -- (7.5, 7.5);
\draw [black, thick] (8.5, 7.5) -- (9.5, 7.5);
\draw [black, thick] (10.5, 7.5) -- (11.5, 7.5);
\draw [black, thick] (0.5, 8.5) -- (1.5, 8.5);
\draw [black, thick] (2.5, 8.5) -- (3.5, 8.5);
\draw [black, thick] (4.5, 8.5) -- (5.5, 8.5);
\draw [black, thick] (6.5, 8.5) -- (7.5, 8.5);
\draw [black, thick] (8.5, 8.5) -- (9.5, 8.5);
\draw [black, thick] (10.5, 8.5) -- (11.5, 8.5);
\draw [black, thick] (0.5, 9.5) -- (1.5, 9.5);
\draw [black, thick] (2.5, 9.5) -- (3.5, 9.5);
\draw [black, thick] (3.5, 9.5) -- (4.5, 9.5);
\draw [black, thick] (5.5, 9.5) -- (6.5, 9.5);
\draw [black, thick] (7.5, 9.5) -- (8.5, 9.5);
\draw [black, thick] (9.5, 9.5) -- (10.5, 9.5);
\draw [black, thick] (1.5, 10.5) -- (2.5, 10.5);
\draw [black, thick] (3.5, 10.5) -- (4.5, 10.5);
\draw [black, thick] (5.5, 10.5) -- (6.5, 10.5);
\draw [black, thick] (7.5, 10.5) -- (8.5, 10.5);
\draw [black, thick] (9.5, 10.5) -- (10.5, 10.5);
\draw [black, thick] (0.5, 11.5) -- (1.5, 11.5);
\draw [black, thick] (2.5, 11.5) -- (3.5, 11.5);
\draw [black, thick] (4.5, 11.5) -- (5.5, 11.5);
\draw [black, thick] (6.5, 11.5) -- (7.5, 11.5);
\draw [black, thick] (8.5, 11.5) -- (9.5, 11.5);
\draw [black, thick] (10.5, 11.5) -- (11.5, 11.5);
\draw [black, thick] (0.5, 0.5) -- (0.5, 1.5);
\draw [black, thick] (1.5, 0.5) -- (1.5, 1.5);
\draw [black, thick] (2.5, 0.5) -- (2.5, 1.5);
\draw [black, thick] (3.5, 0.5) -- (3.5, 1.5);
\draw [black, thick] (4.5, 0.5) -- (4.5, 1.5);
\draw [black, thick] (5.5, 0.5) -- (5.5, 1.5);
\draw [black, thick] (6.5, 0.5) -- (6.5, 1.5);
\draw [black, thick] (7.5, 0.5) -- (7.5, 1.5);
\draw [black, thick] (8.5, 0.5) -- (8.5, 1.5);
\draw [black, thick] (9.5, 0.5) -- (9.5, 1.5);
\draw [black, thick] (10.5, 0.5) -- (10.5, 1.5);
\draw [black, thick] (11.5, 0.5) -- (11.5, 1.5);
\draw [black, thick] (0.5, 1.5) -- (0.5, 2.5);
\draw [black, thick] (11.5, 1.5) -- (11.5, 2.5);
\draw [black, thick] (0.5, 2.5) -- (0.5, 3.5);
\draw [black, thick] (1.5, 2.5) -- (1.5, 3.5);
\draw [black, thick] (2.5, 2.5) -- (2.5, 3.5);
\draw [black, thick] (3.5, 2.5) -- (3.5, 3.5);
\draw [black, thick] (4.5, 2.5) -- (4.5, 3.5);
\draw [black, thick] (5.5, 2.5) -- (5.5, 3.5);
\draw [black, thick] (6.5, 2.5) -- (6.5, 3.5);
\draw [black, thick] (7.5, 2.5) -- (7.5, 3.5);
\draw [black, thick] (9.5, 2.5) -- (9.5, 3.5);
\draw [black, thick] (10.5, 2.5) -- (10.5, 3.5);
\draw [black, thick] (8.5, 3.5) -- (8.5, 4.5);
\draw [black, thick] (11.5, 3.5) -- (11.5, 4.5);
\draw [black, thick] (0.5, 4.5) -- (0.5, 5.5);
\draw [black, thick] (1.5, 4.5) -- (1.5, 5.5);
\draw [black, thick] (2.5, 4.5) -- (2.5, 5.5);
\draw [black, thick] (3.5, 4.5) -- (3.5, 5.5);
\draw [black, thick] (4.5, 4.5) -- (4.5, 5.5);
\draw [black, thick] (5.5, 4.5) -- (5.5, 5.5);
\draw [black, thick] (6.5, 4.5) -- (6.5, 5.5);
\draw [black, thick] (7.5, 4.5) -- (7.5, 5.5);
\draw [black, thick] (9.5, 4.5) -- (9.5, 5.5);
\draw [black, thick] (10.5, 4.5) -- (10.5, 5.5);
\draw [black, thick] (0.5, 5.5) -- (0.5, 6.5);
\draw [black, thick] (11.5, 5.5) -- (11.5, 6.5);
\draw [black, thick] (1.5, 6.5) -- (1.5, 7.5);
\draw [black, thick] (2.5, 6.5) -- (2.5, 7.5);
\draw [black, thick] (4.5, 6.5) -- (4.5, 7.5);
\draw [black, thick] (5.5, 6.5) -- (5.5, 7.5);
\draw [black, thick] (6.5, 6.5) -- (6.5, 7.5);
\draw [black, thick] (7.5, 6.5) -- (7.5, 7.5);
\draw [black, thick] (8.5, 6.5) -- (8.5, 7.5);
\draw [black, thick] (9.5, 6.5) -- (9.5, 7.5);
\draw [black, thick] (10.5, 6.5) -- (10.5, 7.5);
\draw [black, thick] (11.5, 6.5) -- (11.5, 7.5);
\draw [black, thick] (0.5, 7.5) -- (0.5, 8.5);
\draw [black, thick] (3.5, 7.5) -- (3.5, 8.5);
\draw [black, thick] (1.5, 8.5) -- (1.5, 9.5);
\draw [black, thick] (2.5, 8.5) -- (2.5, 9.5);
\draw [black, thick] (4.5, 8.5) -- (4.5, 9.5);
\draw [black, thick] (5.5, 8.5) -- (5.5, 9.5);
\draw [black, thick] (6.5, 8.5) -- (6.5, 9.5);
\draw [black, thick] (7.5, 8.5) -- (7.5, 9.5);
\draw [black, thick] (8.5, 8.5) -- (8.5, 9.5);
\draw [black, thick] (9.5, 8.5) -- (9.5, 9.5);
\draw [black, thick] (10.5, 8.5) -- (10.5, 9.5);
\draw [black, thick] (11.5, 8.5) -- (11.5, 9.5);
\draw [black, thick] (0.5, 9.5) -- (0.5, 10.5);
\draw [black, thick] (11.5, 9.5) -- (11.5, 10.5);
\draw [black, thick] (0.5, 10.5) -- (0.5, 11.5);
\draw [black, thick] (1.5, 10.5) -- (1.5, 11.5);
\draw [black, thick] (2.5, 10.5) -- (2.5, 11.5);
\draw [black, thick] (3.5, 10.5) -- (3.5, 11.5);
\draw [black, thick] (4.5, 10.5) -- (4.5, 11.5);
\draw [black, thick] (5.5, 10.5) -- (5.5, 11.5);
\draw [black, thick] (6.5, 10.5) -- (6.5, 11.5);
\draw [black, thick] (7.5, 10.5) -- (7.5, 11.5);
\draw [black, thick] (8.5, 10.5) -- (8.5, 11.5);
\draw [black, thick] (9.5, 10.5) -- (9.5, 11.5);
\draw [black, thick] (10.5, 10.5) -- (10.5, 11.5);
\draw [black, thick] (11.5, 10.5) -- (11.5, 11.5);

        \end{tikzpicture}
        \caption{A Hamilton cycle in $G(12,12)$ with exactly $12^2 - 12 = 132$ turns, formed from taking the symmetric difference of the cycle in Figure \ref{6x6C''} with a square overlay. Straights are marked with crosses.}
        \label{12x12max}
    \end{figure}

$H$ has exactly $n$ straights, corresponding to the $n$ turns of $C''$.

So we have formed a Hamilton cycle on $G(n,n)$ with exactly $n$ straights, i.e. exactly $n^2-n$ turns.
\end{proof}

In particular, every minimal turning cycle of $G\left(\frac n2, \frac n2 \right)$ can be turned into a maximal turning cycle of $G(n,n)$ in two ways; we could changed construction of $C''$ so that upward oriented segments move the vertex left, and rightward oriented segments move the vertex down. So the number of ways to draw a Hamilton cycle with $n^2 - n$ turns on an $n \times n$ grid where $4|n$ is at least twice the number of ways to draw a Hamilton cycle with $n$ turns on an $\frac n2 \times \frac n2$ grid.

The $n \equiv 2\pmod 4$ case has some differences to the $4|n$ case. The construction method above will not work, as the half-form is in a square of odd side length, and there are no Hamilton cycles of such a grid. In fact, the $n$ straights bound, given from Lemma \ref{hamiltonianhalfform} and Corollary \ref{2turns}, is not tight here, as we prove below.

\begin{lemma}\label{oddbound}
Any Hamilton cycle of $G(n,n)$, where $4|n+2$ and $n>2$, has at least $n+2$ straights.
\end{lemma}

\begin{proof}
Assume for contradiction that $n$ straights are possible. We then take the half-form, remembering that every turn in the half-form corresponds to a straight in the original cycle (including ``invisible" turns where the cycle can turn twice in the same square or go over segments twice).

Following the proof of Lemma \ref{minab}, we note that there are either turns in every row or turns in every column; if this were not true, then we can choose a cycle where there are no turns in its row or column, and derive a contradiction from Lemma \ref{turnexistence}. Without loss of generality, assume that there are turns in every column.

Since the number of turns in every row is even, there are exactly $2$ turns in every row, to attain exactly $n$ turns.

We claim that there are turns in every row. Suppose not, and take a row $R$ with no turns. Every cycle must pass through $R$ an even number of times, since cycles must alternate passing through upward and downward. But $R$ contains an odd number of cells, so some cell $c$ in $R$ is passed through an even number of times.

But the half-form passes through every cell, so in particular $c$ is passed through twice vertically, which must yield four turns in the same column as $C$, a contradiction.

Hence there are two turns in every row. All four corners of the $\frac n2 \times \frac n2$ grid must contain a turn, and since there are two turns in every row and column, there are no other turns adjacent to the boundary of the grid. So the half-form contains a square cycle along the outermost cells of the grid.

Converting this back into the original cycle, we see that the $8(n-2)$ squares forming the four rows and four columns closest to the boundary become a connected component, disconnected from the $(n-4) \times (n-4)$ grid inside, a contradiction as $n \geq 6$.
\end{proof}

The $n=2$ case trivially gives $4$ turns. For $n \geq 6$, $n+2$ straights is achievable. We will show this using a similar method to the proof of Claim \ref{square4|n}.

\begin{claim}\label{maxsq 2mod4}
If $4|n+2$ and $n \geq 6$, the maximal number of turns of a Hamilton cycle in $G(n,n)$ is exactly $n^2 - n - 2$.
\end{claim}

\begin{proof}
We begin by constructing a half-form on $G\left(\frac n2, \frac n2 \right)$ which has exactly $n+2$ turns. This consists of two cycles. One is a spiral cycle, and the other is a cycle consisting of a single segment twice (with two turns on each end).

    \begin{figure}[H]
        \centering
        \begin{tikzpicture}
\draw[step=0.5cm,gray,thin,dashed] (0,1) grid (1.5,2.5);
\draw [black, thin] (0, 1) -- (1.5, 1) -- (1.5, 2.5) -- (0, 2.5) -- cycle;
\draw[step=0.5cm,gray,thin,dashed] (2,0.5) grid (4.5,3);
\draw [black, thin] (2, 0.5) -- (4.5, 0.5) -- (4.5, 3) -- (2, 3) -- cycle;
\draw[step=0.5cm,gray,thin,dashed] (5,0) grid (8.5,3.5);
\draw [black, thin] (5, 0) -- (8.5, 0) -- (8.5, 3.5) -- (5, 3.5) -- cycle;
\draw [black, thick] (5.25, 0.25) -- (5.75, 0.25);
\draw [black, thick] (5.75, 0.25) -- (6.25, 0.25);
\draw [black, thick] (6.25, 0.25) -- (6.75, 0.25);
\draw [black, thick] (6.75, 0.25) -- (7.25, 0.25);
\draw [black, thick] (7.75, 0.25) -- (8.25, 0.25);
\draw [black, thick] (2.25, 0.75) -- (2.75, 0.75);
\draw [black, thick] (2.75, 0.75) -- (3.25, 0.75);
\draw [black, thick] (3.25, 0.75) -- (3.75, 0.75);
\draw [black, thick] (3.75, 0.75) -- (4.25, 0.75);
\draw [black, thick] (5.75, 0.75) -- (6.25, 0.75);
\draw [black, thick] (6.25, 0.75) -- (6.75, 0.75);
\draw [black, thick] (0.25, 1.25) -- (0.75, 1.25);
\draw [black, thick] (0.75, 1.25) -- (1.25, 1.25);
\draw [black, thick] (2.75, 1.25) -- (3.25, 1.25);
\draw [black, thick] (3.25, 1.25) -- (3.75, 1.25);
\draw [black, thick] (6.25, 1.25) -- (6.75, 1.25);
\draw [black, thick] (0.75, 1.75) -- (1.25, 1.75);
\draw [black, thick] (3.25, 1.75) -- (3.75, 1.75);
\draw [black, thick] (6.75, 1.75) -- (7.25, 1.75);
\draw [black, thick] (0.25, 2.25) -- (0.75, 2.25);
\draw [black, thick] (0.75, 2.25) -- (1.25, 2.25);
\draw [black, thick] (3.75, 2.25) -- (4.25, 2.25);
\draw [black, thick] (6.25, 2.25) -- (6.75, 2.25);
\draw [black, thick] (6.75, 2.25) -- (7.25, 2.25);
\draw [black, thick] (2.25, 2.75) -- (2.75, 2.75);
\draw [black, thick] (3.25, 2.75) -- (3.75, 2.75);
\draw [black, thick] (3.75, 2.75) -- (4.25, 2.75);
\draw [black, thick] (5.75, 2.75) -- (6.25, 2.75);
\draw [black, thick] (6.25, 2.75) -- (6.75, 2.75);
\draw [black, thick] (6.75, 2.75) -- (7.25, 2.75);
\draw [black, thick] (7.25, 2.75) -- (7.75, 2.75);
\draw [black, thick] (5.25, 3.25) -- (5.75, 3.25);
\draw [black, thick] (5.75, 3.25) -- (6.25, 3.25);
\draw [black, thick] (6.25, 3.25) -- (6.75, 3.25);
\draw [black, thick] (6.75, 3.25) -- (7.25, 3.25);
\draw [black, thick] (7.25, 3.25) -- (7.75, 3.25);
\draw [black, thick] (7.75, 3.25) -- (8.25, 3.25);
\draw [black, thick] (5.25, 0.25) -- (5.25, 0.75);
\draw [black, thick] (7.25, 0.25) -- (7.25, 0.75);
\draw [black, thick] (7.75, 0.25) -- (7.75, 0.75);
\draw [black, thick] (8.25, 0.25) -- (8.25, 0.75);
\draw [black, thick] (2.25, 0.75) -- (2.25, 1.25);
\draw [black, thick] (4.25, 0.75) -- (4.25, 1.25);
\draw [black, thick] (5.25, 0.75) -- (5.25, 1.25);
\draw [black, thick] (5.75, 0.75) -- (5.75, 1.25);
\draw [black, thick] (6.75, 0.75) -- (6.75, 1.25);
\draw [black, thick] (7.25, 0.75) -- (7.25, 1.25);
\draw [black, thick] (7.75, 0.75) -- (7.75, 1.25);
\draw [black, thick] (8.25, 0.75) -- (8.25, 1.25);
\draw [black, thick] (0.25, 1.25) -- (0.25, 1.75);
\draw [black, thick] (1.25, 1.25) -- (1.25, 1.75);
\draw [black, thick] (2.25, 1.25) -- (2.25, 1.75);
\draw [black, thick] (2.75, 1.25) -- (2.75, 1.75);
\draw [black, thick] (3.75, 1.25) -- (3.75, 1.75);
\draw [black, thick] (4.25, 1.25) -- (4.25, 1.75);
\draw [black, thick] (5.25, 1.25) -- (5.25, 1.75);
\draw [black, thick] (5.75, 1.25) -- (5.75, 1.75);
\draw [black, thick] (6.25, 1.25) -- (6.25, 1.75);
\draw [black, thick] (7.25, 1.25) -- (7.25, 1.75);
\draw [black, thick] (7.75, 1.25) -- (7.75, 1.75);
\draw [black, thick] (8.25, 1.25) -- (8.25, 1.75);
\draw [black, thick] (0.25, 1.75) -- (0.25, 2.25);
\draw [black, thick] (1.25, 1.75) -- (1.25, 2.25);
\draw [black, thick] (2.25, 1.75) -- (2.25, 2.25);
\draw [black, thick] (2.75, 1.75) -- (2.75, 2.25);
\draw [black, thick] (3.25, 1.75) -- (3.25, 2.25);
\draw [black, thick] (4.25, 1.75) -- (4.25, 2.25);
\draw [black, thick] (5.25, 1.75) -- (5.25, 2.25);
\draw [black, thick] (5.75, 1.75) -- (5.75, 2.25);
\draw [black, thick] (6.25, 1.75) -- (6.25, 2.25);
\draw [black, thick] (7.25, 1.75) -- (7.25, 2.25);
\draw [black, thick] (7.75, 1.75) -- (7.75, 2.25);
\draw [black, thick] (8.25, 1.75) -- (8.25, 2.25);
\draw [black, thick] (2.25, 2.25) -- (2.25, 2.75);
\draw [black, thick] (2.75, 2.25) -- (2.75, 2.75);
\draw [black, thick] (3.25, 2.25) -- (3.25, 2.75);
\draw [black, thick] (4.25, 2.25) -- (4.25, 2.75);
\draw [black, thick] (5.25, 2.25) -- (5.25, 2.75);
\draw [black, thick] (5.75, 2.25) -- (5.75, 2.75);
\draw [black, thick] (7.75, 2.25) -- (7.75, 2.75);
\draw [black, thick] (8.25, 2.25) -- (8.25, 2.75);
\draw [black, thick] (5.25, 2.75) -- (5.25, 3.25);
\draw [black, thick] (8.25, 2.75) -- (8.25, 3.25);

        \end{tikzpicture}
        \caption{Half-form constructions for $n=6, 10, 14$.}
        \label{oddhalfform}
    \end{figure}

We then orient both cycles clockwise; this includes the single-segment cycle; the left endpoint has a left-up turn and an up-right turn, and the right endpoint has a right-down and down-left turn.

We then apply the same transformation as detailed in the proof of Claim \ref{square4|n} to obtain a reduced form.

    \begin{figure}[H]
        \centering
        \begin{tikzpicture}
\draw[step=0.5cm,gray,thin,dashed] (0,0) grid (5.0,5.0);
\draw [black, thin] (0, 0) -- (5.0, 0) -- (5.0, 5.0) -- (0, 5.0) -- cycle;
\draw [black, thick] (0.25, 0.75) -- (0.75, 0.75);
\draw [black, thick] (0.75, 0.75) -- (1.25, 0.75);
\draw [black, thick] (1.25, 0.75) -- (1.75, 0.75);
\draw [black, thick] (1.75, 0.75) -- (2.25, 0.75);
\draw [black, thick] (2.25, 0.75) -- (2.75, 0.75);
\draw [black, thick] (2.75, 0.75) -- (3.25, 0.75);
\draw [black, thick] (3.25, 0.75) -- (3.75, 0.75);
\draw [black, thick] (3.75, 0.75) -- (4.25, 0.75);
\draw [black, thick] (4.25, 0.75) -- (4.75, 0.75);
\draw [black, thick] (1.75, 1.25) -- (2.25, 1.25);
\draw [black, thick] (2.25, 1.25) -- (2.75, 1.25);
\draw [black, thick] (2.75, 1.25) -- (3.25, 1.25);
\draw [black, thick] (2.25, 2.75) -- (2.75, 2.75);
\draw [black, thick] (2.75, 2.75) -- (3.25, 2.75);
\draw [black, thick] (3.75, 3.25) -- (4.25, 3.25);
\draw [black, thick] (3.75, 3.75) -- (4.25, 3.75);
\draw [black, thick] (0.25, 4.25) -- (0.75, 4.25);
\draw [black, thick] (0.75, 4.25) -- (1.25, 4.25);
\draw [black, thick] (1.25, 4.25) -- (1.75, 4.25);
\draw [black, thick] (2.25, 4.25) -- (2.75, 4.25);
\draw [black, thick] (2.75, 4.25) -- (3.25, 4.25);
\draw [black, thick] (3.25, 4.25) -- (3.75, 4.25);
\draw [black, thick] (3.75, 4.25) -- (4.25, 4.25);
\draw [black, thick] (4.25, 4.25) -- (4.75, 4.25);
\draw [black, thick] (0.25, 0.75) -- (0.25, 1.25);
\draw [black, thick] (4.75, 0.75) -- (4.75, 1.25);
\draw [black, thick] (0.25, 1.25) -- (0.25, 1.75);
\draw [black, thick] (1.75, 1.25) -- (1.75, 1.75);
\draw [black, thick] (3.25, 1.25) -- (3.25, 1.75);
\draw [black, thick] (4.75, 1.25) -- (4.75, 1.75);
\draw [black, thick] (0.25, 1.75) -- (0.25, 2.25);
\draw [black, thick] (1.75, 1.75) -- (1.75, 2.25);
\draw [black, thick] (3.25, 1.75) -- (3.25, 2.25);
\draw [black, thick] (4.75, 1.75) -- (4.75, 2.25);
\draw [black, thick] (0.25, 2.25) -- (0.25, 2.75);
\draw [black, thick] (1.75, 2.25) -- (1.75, 2.75);
\draw [black, thick] (3.25, 2.25) -- (3.25, 2.75);
\draw [black, thick] (4.75, 2.25) -- (4.75, 2.75);
\draw [black, thick] (0.25, 2.75) -- (0.25, 3.25);
\draw [black, thick] (1.75, 2.75) -- (1.75, 3.25);
\draw [black, thick] (2.25, 2.75) -- (2.25, 3.25);
\draw [black, thick] (4.75, 2.75) -- (4.75, 3.25);
\draw [black, thick] (0.25, 3.25) -- (0.25, 3.75);
\draw [black, thick] (1.75, 3.25) -- (1.75, 3.75);
\draw [black, thick] (2.25, 3.25) -- (2.25, 3.75);
\draw [black, thick] (3.75, 3.25) -- (3.75, 3.75);
\draw [black, thick] (4.25, 3.25) -- (4.25, 3.75);
\draw [black, thick] (4.75, 3.25) -- (4.75, 3.75);
\draw [black, thick] (0.25, 3.75) -- (0.25, 4.25);
\draw [black, thick] (1.75, 3.75) -- (1.75, 4.25);
\draw [black, thick] (2.25, 3.75) -- (2.25, 4.25);
\draw [black, thick] (4.75, 3.75) -- (4.75, 4.25);

        \end{tikzpicture}
        \caption{The reduced form for $n=10$.}
        \label{10reducedform}
    \end{figure}

The single segment-cycle becomes a $2 \times 2$ cycle. Taking the symmetric difference of this cycle with a square overlay joins the two cycles and yields a Hamilton cycle with exactly $n+2$ straights. An example of the $n=10$ case is shown below.

    \begin{figure}[H]
        \centering
        \begin{tikzpicture}
\draw[step=1cm,gray,thin,dashed] (0,0) grid (10,10);
\draw [black, thin] (0, 0) -- (10, 0) -- (10, 10) -- (0, 10) -- cycle;
\draw [gray, thin] (0.10, 1.1) -- (0.9, 1.9);
\draw [gray, thin] (0.9, 1.1) -- (0.10, 1.9);
\draw [gray, thin] (9.1, 1.1) -- (9.9, 1.9);
\draw [gray, thin] (9.9, 1.1) -- (9.1, 1.9);
\draw [gray, thin] (3.1, 2.1) -- (3.9, 2.9);
\draw [gray, thin] (3.9, 2.1) -- (3.1, 2.9);
\draw [gray, thin] (6.1, 2.1) -- (6.9, 2.9);
\draw [gray, thin] (6.9, 2.1) -- (6.1, 2.9);
\draw [gray, thin] (4.1, 5.1) -- (4.9, 5.9);
\draw [gray, thin] (4.9, 5.1) -- (4.1, 5.9);
\draw [gray, thin] (6.1, 5.1) -- (6.9, 5.9);
\draw [gray, thin] (6.9, 5.1) -- (6.1, 5.9);
\draw [gray, thin] (7.1, 6.1) -- (7.9, 6.9);
\draw [gray, thin] (7.9, 6.1) -- (7.1, 6.9);
\draw [gray, thin] (8.1, 6.1) -- (8.9, 6.9);
\draw [gray, thin] (8.9, 6.1) -- (8.1, 6.9);
\draw [gray, thin] (7.1, 7.1) -- (7.9, 7.9);
\draw [gray, thin] (7.9, 7.1) -- (7.1, 7.9);
\draw [gray, thin] (8.1, 7.1) -- (8.9, 7.9);
\draw [gray, thin] (8.9, 7.1) -- (8.1, 7.9);
\draw [gray, thin] (0.10, 8.1) -- (0.9, 8.9);
\draw [gray, thin] (0.9, 8.1) -- (0.10, 8.9);
\draw [gray, thin] (3.1, 8.1) -- (3.9, 8.9);
\draw [gray, thin] (3.9, 8.1) -- (3.1, 8.9);
\draw [gray, thin] (4.1, 8.1) -- (4.9, 8.9);
\draw [gray, thin] (4.9, 8.1) -- (4.1, 8.9);
\draw [gray, thin] (9.1, 8.1) -- (9.9, 8.9);
\draw [gray, thin] (9.9, 8.1) -- (9.1, 8.9);
\draw [black, thick] (0.5, 0.5) -- (1.5, 0.5);
\draw [black, thick] (2.5, 0.5) -- (3.5, 0.5);
\draw [black, thick] (4.5, 0.5) -- (5.5, 0.5);
\draw [black, thick] (6.5, 0.5) -- (7.5, 0.5);
\draw [black, thick] (8.5, 0.5) -- (9.5, 0.5);
\draw [black, thick] (1.5, 1.5) -- (2.5, 1.5);
\draw [black, thick] (3.5, 1.5) -- (4.5, 1.5);
\draw [black, thick] (5.5, 1.5) -- (6.5, 1.5);
\draw [black, thick] (7.5, 1.5) -- (8.5, 1.5);
\draw [black, thick] (0.5, 2.5) -- (1.5, 2.5);
\draw [black, thick] (2.5, 2.5) -- (3.5, 2.5);
\draw [black, thick] (3.5, 2.5) -- (4.5, 2.5);
\draw [black, thick] (5.5, 2.5) -- (6.5, 2.5);
\draw [black, thick] (6.5, 2.5) -- (7.5, 2.5);
\draw [black, thick] (8.5, 2.5) -- (9.5, 2.5);
\draw [black, thick] (0.5, 3.5) -- (1.5, 3.5);
\draw [black, thick] (2.5, 3.5) -- (3.5, 3.5);
\draw [black, thick] (4.5, 3.5) -- (5.5, 3.5);
\draw [black, thick] (6.5, 3.5) -- (7.5, 3.5);
\draw [black, thick] (8.5, 3.5) -- (9.5, 3.5);
\draw [black, thick] (0.5, 4.5) -- (1.5, 4.5);
\draw [black, thick] (2.5, 4.5) -- (3.5, 4.5);
\draw [black, thick] (4.5, 4.5) -- (5.5, 4.5);
\draw [black, thick] (6.5, 4.5) -- (7.5, 4.5);
\draw [black, thick] (8.5, 4.5) -- (9.5, 4.5);
\draw [black, thick] (0.5, 5.5) -- (1.5, 5.5);
\draw [black, thick] (2.5, 5.5) -- (3.5, 5.5);
\draw [black, thick] (5.5, 5.5) -- (6.5, 5.5);
\draw [black, thick] (6.5, 5.5) -- (7.5, 5.5);
\draw [black, thick] (8.5, 5.5) -- (9.5, 5.5);
\draw [black, thick] (0.5, 6.5) -- (1.5, 6.5);
\draw [black, thick] (2.5, 6.5) -- (3.5, 6.5);
\draw [black, thick] (4.5, 6.5) -- (5.5, 6.5);
\draw [black, thick] (6.5, 6.5) -- (7.5, 6.5);
\draw [black, thick] (7.5, 6.5) -- (8.5, 6.5);
\draw [black, thick] (8.5, 6.5) -- (9.5, 6.5);
\draw [black, thick] (0.5, 7.5) -- (1.5, 7.5);
\draw [black, thick] (2.5, 7.5) -- (3.5, 7.5);
\draw [black, thick] (4.5, 7.5) -- (5.5, 7.5);
\draw [black, thick] (6.5, 7.5) -- (7.5, 7.5);
\draw [black, thick] (7.5, 7.5) -- (8.5, 7.5);
\draw [black, thick] (8.5, 7.5) -- (9.5, 7.5);
\draw [black, thick] (1.5, 8.5) -- (2.5, 8.5);
\draw [black, thick] (5.5, 8.5) -- (6.5, 8.5);
\draw [black, thick] (7.5, 8.5) -- (8.5, 8.5);
\draw [black, thick] (0.5, 9.5) -- (1.5, 9.5);
\draw [black, thick] (2.5, 9.5) -- (3.5, 9.5);
\draw [black, thick] (4.5, 9.5) -- (5.5, 9.5);
\draw [black, thick] (6.5, 9.5) -- (7.5, 9.5);
\draw [black, thick] (8.5, 9.5) -- (9.5, 9.5);
\draw [black, thick] (0.5, 0.5) -- (0.5, 1.5);
\draw [black, thick] (1.5, 0.5) -- (1.5, 1.5);
\draw [black, thick] (2.5, 0.5) -- (2.5, 1.5);
\draw [black, thick] (3.5, 0.5) -- (3.5, 1.5);
\draw [black, thick] (4.5, 0.5) -- (4.5, 1.5);
\draw [black, thick] (5.5, 0.5) -- (5.5, 1.5);
\draw [black, thick] (6.5, 0.5) -- (6.5, 1.5);
\draw [black, thick] (7.5, 0.5) -- (7.5, 1.5);
\draw [black, thick] (8.5, 0.5) -- (8.5, 1.5);
\draw [black, thick] (9.5, 0.5) -- (9.5, 1.5);
\draw [black, thick] (0.5, 1.5) -- (0.5, 2.5);
\draw [black, thick] (9.5, 1.5) -- (9.5, 2.5);
\draw [black, thick] (1.5, 2.5) -- (1.5, 3.5);
\draw [black, thick] (2.5, 2.5) -- (2.5, 3.5);
\draw [black, thick] (4.5, 2.5) -- (4.5, 3.5);
\draw [black, thick] (5.5, 2.5) -- (5.5, 3.5);
\draw [black, thick] (7.5, 2.5) -- (7.5, 3.5);
\draw [black, thick] (8.5, 2.5) -- (8.5, 3.5);
\draw [black, thick] (0.5, 3.5) -- (0.5, 4.5);
\draw [black, thick] (3.5, 3.5) -- (3.5, 4.5);
\draw [black, thick] (6.5, 3.5) -- (6.5, 4.5);
\draw [black, thick] (9.5, 3.5) -- (9.5, 4.5);
\draw [black, thick] (1.5, 4.5) -- (1.5, 5.5);
\draw [black, thick] (2.5, 4.5) -- (2.5, 5.5);
\draw [black, thick] (4.5, 4.5) -- (4.5, 5.5);
\draw [black, thick] (5.5, 4.5) -- (5.5, 5.5);
\draw [black, thick] (7.5, 4.5) -- (7.5, 5.5);
\draw [black, thick] (8.5, 4.5) -- (8.5, 5.5);
\draw [black, thick] (0.5, 5.5) -- (0.5, 6.5);
\draw [black, thick] (3.5, 5.5) -- (3.5, 6.5);
\draw [black, thick] (4.5, 5.5) -- (4.5, 6.5);
\draw [black, thick] (9.5, 5.5) -- (9.5, 6.5);
\draw [black, thick] (1.5, 6.5) -- (1.5, 7.5);
\draw [black, thick] (2.5, 6.5) -- (2.5, 7.5);
\draw [black, thick] (5.5, 6.5) -- (5.5, 7.5);
\draw [black, thick] (6.5, 6.5) -- (6.5, 7.5);
\draw [black, thick] (0.5, 7.5) -- (0.5, 8.5);
\draw [black, thick] (3.5, 7.5) -- (3.5, 8.5);
\draw [black, thick] (4.5, 7.5) -- (4.5, 8.5);
\draw [black, thick] (9.5, 7.5) -- (9.5, 8.5);
\draw [black, thick] (0.5, 8.5) -- (0.5, 9.5);
\draw [black, thick] (1.5, 8.5) -- (1.5, 9.5);
\draw [black, thick] (2.5, 8.5) -- (2.5, 9.5);
\draw [black, thick] (3.5, 8.5) -- (3.5, 9.5);
\draw [black, thick] (4.5, 8.5) -- (4.5, 9.5);
\draw [black, thick] (5.5, 8.5) -- (5.5, 9.5);
\draw [black, thick] (6.5, 8.5) -- (6.5, 9.5);
\draw [black, thick] (7.5, 8.5) -- (7.5, 9.5);
\draw [black, thick] (8.5, 8.5) -- (8.5, 9.5);
\draw [black, thick] (9.5, 8.5) -- (9.5, 9.5);

        \end{tikzpicture}
        \caption{A Hamilton cycle of $G(10, 10)$ with exactly $10^2 - 10 - 2 = 88$ turns, formed from taking the symmetric difference of the cycle in Figure \ref{10reducedform} with a square overlay. Straights are marked with crosses.}
        \label{10x10max}
    \end{figure}

Hence we have constructed a Hamilton cycle with $n^2-n-2$ turns. From Lemma \ref{oddbound}, this is optimal.
\end{proof}

The half-turn used is not unique; the construction can be rotated/flipped, and the segment can be rotated. As mentioned previously, the above method to converts any minimal cycle on an $2n \times 2n$ grid into a maximal cycle on an $4n \times 4n$ grid (which is noted in \cite{beluhov} in the form of a problem), although examples not of this form exist; one can find a minimal half-form which passes through all cells, passing through some twice (and hence is not a cycle of the $2n \times 2n$ grid).

Claims \ref{square4|n} and \ref{maxsq 2mod4} combined prove Theorem \ref{square beluhov}, concluding the square case.


\section{Maximum turns in rectangles}\label{rect}
We now turn to the general rectangle case of an $m \times n$ grid, where $m > n > 2$ (the $n=2$ case is trivial).

The main result we will prove is that if $n\equiv2\pmod4$, the maximal number of turns is $$mn-2\left\lfloor \frac{2m+n+2}{6}\right\rfloor\text{   or   }mn-2\left\lfloor \frac{2m+n+2}{6}\right\rfloor-2,$$
and otherwise the maximal number of turns is $mn-n$.

\subsection{\texorpdfstring{$n \not\equiv 2 \pmod 4$}{}}

For $G(m,n)$, if both $m$ and $n$ are even, we can define the square overlay, reduced form, and half-form in the same way as the square case. Lemma \ref{hamiltonianhalfform} still holds in the non-square case, and from Corollary \ref{2turns}, there are at least $2 \min\left(\frac m2, \frac n2\right)=n$ turns in the half-form, proven in the same way as \ref{minab} (which holds for any union of cycles that cover all cells of a grid). So there are at least $n$ turns.

One may think that when the half-grid has a side which has odd length, Lemma \ref{oddturn} will still hold, but this is in fact false; Corollary \ref{2turnsodd} does not hold for a half-form because we are now permitted to go over a segment twice, and hence a column/row with odd length may be passed through with no turns by simply going through some cells twice. A modification will be presented later, which gives a different bound.

For now, we want to see when the bound of $n$ is attainable. For the cases where it is, the construction method is identical to that of Claim \ref{square4|n}.

\begin{claim}\label{max4|n}
In $G(m,n)$ with $m>n$, if either
\begin{itemize}
    \item $n$ is divisible by four, or
    \item $n$ is odd,
\end{itemize}
then the maximal number of turns of any Hamilton cycle is exactly $mn - n$ in the former case, and $mn-m$ in the latter.
\end{claim}
In the case where $4|n$ and $2|m$, follow the proof of Claim \ref{square4|n}, except use a Hamilton cycle of $G\left(\frac m2,\frac n2\right)$ grid with $n$ turns as the half-form. This is guaranteed to exist, from Theorem \ref{minturns}.

    \begin{figure}[H]
        \centering
        \begin{tikzpicture}
\draw[step=1cm,gray,thin,dashed] (0,0) grid (5,4);
\draw [black, thin] (0, 0) -- (5, 0) -- (5, 4) -- (0, 4) -- cycle;
\draw [black, thick] (0.5, 0.5) -- (1.5, 0.5);
\draw [black, thick] (1.5, 0.5) -- (2.5, 0.5);
\draw [black, thick] (2.5, 0.5) -- (3.5, 0.5);
\draw [black, thick] (3.5, 0.5) -- (4.5, 0.5);
\draw [black, thick] (1.5, 1.5) -- (2.5, 1.5);
\draw [black, thick] (2.5, 1.5) -- (3.5, 1.5);
\draw [black, thick] (3.5, 1.5) -- (4.5, 1.5);
\draw [black, thick] (1.5, 2.5) -- (2.5, 2.5);
\draw [black, thick] (2.5, 2.5) -- (3.5, 2.5);
\draw [black, thick] (3.5, 2.5) -- (4.5, 2.5);
\draw [black, thick] (0.5, 3.5) -- (1.5, 3.5);
\draw [black, thick] (1.5, 3.5) -- (2.5, 3.5);
\draw [black, thick] (2.5, 3.5) -- (3.5, 3.5);
\draw [black, thick] (3.5, 3.5) -- (4.5, 3.5);
\draw [black, thick] (0.5, 0.5) -- (0.5, 1.5);
\draw [black, thick] (4.5, 0.5) -- (4.5, 1.5);
\draw [black, thick] (0.5, 1.5) -- (0.5, 2.5);
\draw [black, thick] (1.5, 1.5) -- (1.5, 2.5);
\draw [black, thick] (0.5, 2.5) -- (0.5, 3.5);
\draw [black, thick] (4.5, 2.5) -- (4.5, 3.5);

\draw[step=0.5cm,gray,thin,dashed] (6,0) grid (11.0,4.0);
\draw [black, thin] (6, 0) -- (11.0, 0) -- (11.0, 4.0) -- (6, 4.0) -- cycle;
\draw [black, thick] (6.25, 0.75) -- (6.75, 0.75);
\draw [black, thick] (6.75, 0.75) -- (7.25, 0.75);
\draw [black, thick] (7.25, 0.75) -- (7.75, 0.75);
\draw [black, thick] (10.25, 0.75) -- (10.75, 0.75);
\draw [black, thick] (8.25, 0.75) -- (7.75, 0.75);
\draw [black, thick] (8.75, 0.75) -- (8.25, 0.75);
\draw [black, thick] (9.25, 0.75) -- (8.75, 0.75);
\draw [black, thick] (9.75, 0.75) -- (9.25, 0.75);
\draw [black, thick] (10.25, 0.75) -- (9.75, 0.75);
\draw [black, thick] (8.25, 1.25) -- (8.75, 1.25);
\draw [black, thick] (8.25, 1.25) -- (7.75, 1.25);
\draw [black, thick] (8.75, 1.25) -- (9.25, 1.25);
\draw [black, thick] (9.25, 1.25) -- (9.75, 1.25);
\draw [black, thick] (9.75, 1.25) -- (10.25, 1.25);
\draw [black, thick] (10.25, 1.25) -- (10.75, 1.25);
\draw [black, thick] (7.75, 2.75) -- (8.25, 2.75);
\draw [black, thick] (8.25, 2.75) -- (8.75, 2.75);
\draw [black, thick] (8.75, 2.75) -- (9.25, 2.75);
\draw [black, thick] (9.25, 2.75) -- (9.75, 2.75);
\draw [black, thick] (9.75, 2.75) -- (10.25, 2.75);
\draw [black, thick] (10.25, 2.75) -- (10.75, 2.75);
\draw [black, thick] (6.25, 3.25) -- (6.75, 3.25);
\draw [black, thick] (6.75, 3.25) -- (7.25, 3.25);
\draw [black, thick] (7.25, 3.25) -- (7.75, 3.25);
\draw [black, thick] (7.75, 3.25) -- (8.25, 3.25);
\draw [black, thick] (8.25, 3.25) -- (8.75, 3.25);
\draw [black, thick] (8.75, 3.25) -- (9.25, 3.25);
\draw [black, thick] (9.25, 3.25) -- (9.75, 3.25);
\draw [black, thick] (9.75, 3.25) -- (10.25, 3.25);
\draw [black, thick] (10.25, 3.25) -- (10.75, 3.25);
\draw [black, thick] (6.25, 0.75) -- (6.25, 1.25);
\draw [black, thick] (10.75, 0.75) -- (10.75, 1.25);
\draw [black, thick] (6.25, 1.25) -- (6.25, 1.75);
\draw [black, thick] (7.75, 1.25) -- (7.75, 1.75);
\draw [black, thick] (6.25, 1.75) -- (6.25, 2.25);
\draw [black, thick] (7.75, 1.75) -- (7.75, 2.25);
\draw [black, thick] (6.25, 2.25) -- (6.25, 2.75);
\draw [black, thick] (7.75, 2.25) -- (7.75, 2.75);
\draw [black, thick] (6.25, 2.75) -- (6.25, 3.25);
\draw [black, thick] (10.75, 2.75) -- (10.75, 3.25);
        \end{tikzpicture}
        \caption{Converting a $G(5,4)$ half-form into a $G(10,8)$ reduced form.}
        \label{10x8reduced}
    \end{figure}
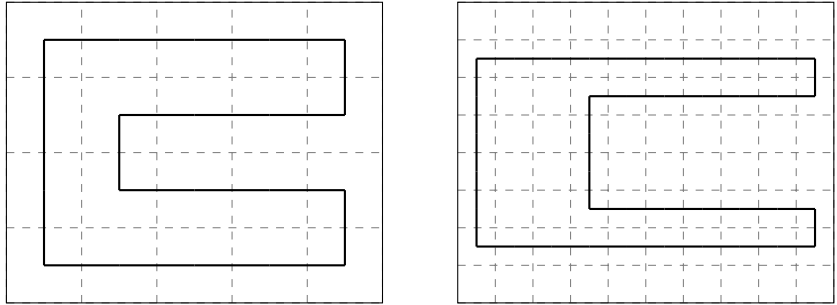

    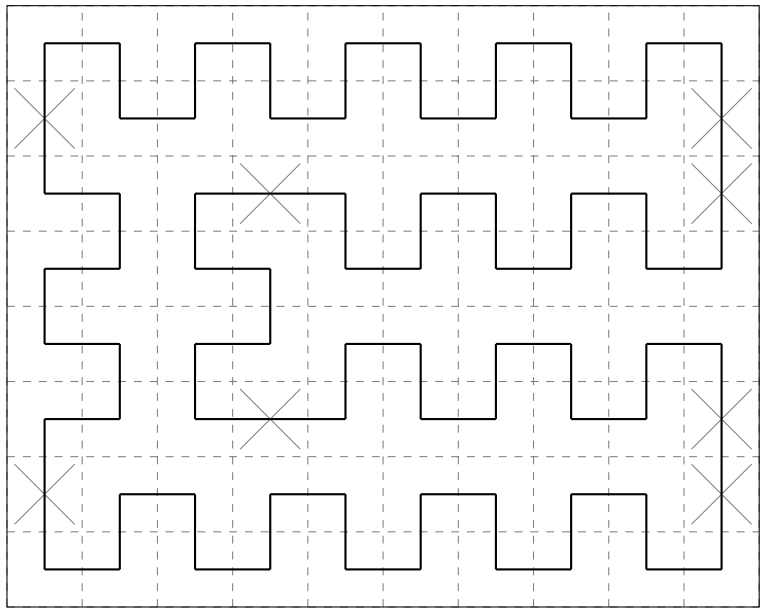
\begin{figure}[H]
        \centering
        \begin{tikzpicture}
\draw[step=1cm,gray,thin,dashed] (0,0) grid (10,8);
\draw [black, thin] (0, 0) -- (10, 0) -- (10, 8) -- (0, 8) -- cycle;
\draw [gray, thin] (0.1, 1.1) -- (0.9, 1.9);
\draw [gray, thin] (0.9, 1.1) -- (0.1, 1.9);
\draw [gray, thin] (9.1, 1.1) -- (9.9, 1.9);
\draw [gray, thin] (9.9, 1.1) -- (9.1, 1.9);
\draw [gray, thin] (3.1, 2.1) -- (3.9, 2.9);
\draw [gray, thin] (3.9, 2.1) -- (3.1, 2.9);
\draw [gray, thin] (9.1, 2.1) -- (9.9, 2.9);
\draw [gray, thin] (9.9, 2.1) -- (9.1, 2.9);
\draw [gray, thin] (3.1, 5.1) -- (3.9, 5.9);
\draw [gray, thin] (3.9, 5.1) -- (3.1, 5.9);
\draw [gray, thin] (9.1, 5.1) -- (9.9, 5.9);
\draw [gray, thin] (9.9, 5.1) -- (9.1, 5.9);
\draw [gray, thin] (0.1, 6.1) -- (0.9, 6.9);
\draw [gray, thin] (0.9, 6.1) -- (0.1, 6.9);
\draw [gray, thin] (9.1, 6.1) -- (9.9, 6.9);
\draw [gray, thin] (9.9, 6.1) -- (9.1, 6.9);
\draw [black, thick] (0.5, 0.5) -- (1.5, 0.5);
\draw [black, thick] (2.5, 0.5) -- (3.5, 0.5);
\draw [black, thick] (4.5, 0.5) -- (5.5, 0.5);
\draw [black, thick] (6.5, 0.5) -- (7.5, 0.5);
\draw [black, thick] (8.5, 0.5) -- (9.5, 0.5);
\draw [black, thick] (1.5, 1.5) -- (2.5, 1.5);
\draw [black, thick] (3.5, 1.5) -- (4.5, 1.5);
\draw [black, thick] (5.5, 1.5) -- (6.5, 1.5);
\draw [black, thick] (7.5, 1.5) -- (8.5, 1.5);
\draw [black, thick] (0.5, 2.5) -- (1.5, 2.5);
\draw [black, thick] (2.5, 2.5) -- (3.5, 2.5);
\draw [black, thick] (3.5, 2.5) -- (4.5, 2.5);
\draw [black, thick] (5.5, 2.5) -- (6.5, 2.5);
\draw [black, thick] (7.5, 2.5) -- (8.5, 2.5);
\draw [black, thick] (0.5, 3.5) -- (1.5, 3.5);
\draw [black, thick] (2.5, 3.5) -- (3.5, 3.5);
\draw [black, thick] (4.5, 3.5) -- (5.5, 3.5);
\draw [black, thick] (6.5, 3.5) -- (7.5, 3.5);
\draw [black, thick] (8.5, 3.5) -- (9.5, 3.5);
\draw [black, thick] (0.5, 4.5) -- (1.5, 4.5);
\draw [black, thick] (2.5, 4.5) -- (3.5, 4.5);
\draw [black, thick] (4.5, 4.5) -- (5.5, 4.5);
\draw [black, thick] (6.5, 4.5) -- (7.5, 4.5);
\draw [black, thick] (8.5, 4.5) -- (9.5, 4.5);
\draw [black, thick] (0.5, 5.5) -- (1.5, 5.5);
\draw [black, thick] (2.5, 5.5) -- (3.5, 5.5);
\draw [black, thick] (3.5, 5.5) -- (4.5, 5.5);
\draw [black, thick] (5.5, 5.5) -- (6.5, 5.5);
\draw [black, thick] (7.5, 5.5) -- (8.5, 5.5);
\draw [black, thick] (1.5, 6.5) -- (2.5, 6.5);
\draw [black, thick] (3.5, 6.5) -- (4.5, 6.5);
\draw [black, thick] (5.5, 6.5) -- (6.5, 6.5);
\draw [black, thick] (7.5, 6.5) -- (8.5, 6.5);
\draw [black, thick] (0.5, 7.5) -- (1.5, 7.5);
\draw [black, thick] (2.5, 7.5) -- (3.5, 7.5);
\draw [black, thick] (4.5, 7.5) -- (5.5, 7.5);
\draw [black, thick] (6.5, 7.5) -- (7.5, 7.5);
\draw [black, thick] (8.5, 7.5) -- (9.5, 7.5);
\draw [black, thick] (0.5, 0.5) -- (0.5, 1.5);
\draw [black, thick] (1.5, 0.5) -- (1.5, 1.5);
\draw [black, thick] (2.5, 0.5) -- (2.5, 1.5);
\draw [black, thick] (3.5, 0.5) -- (3.5, 1.5);
\draw [black, thick] (4.5, 0.5) -- (4.5, 1.5);
\draw [black, thick] (5.5, 0.5) -- (5.5, 1.5);
\draw [black, thick] (6.5, 0.5) -- (6.5, 1.5);
\draw [black, thick] (7.5, 0.5) -- (7.5, 1.5);
\draw [black, thick] (8.5, 0.5) -- (8.5, 1.5);
\draw [black, thick] (9.5, 0.5) -- (9.5, 1.5);
\draw [black, thick] (0.5, 1.5) -- (0.5, 2.5);
\draw [black, thick] (9.5, 1.5) -- (9.5, 2.5);
\draw [black, thick] (1.5, 2.5) -- (1.5, 3.5);
\draw [black, thick] (2.5, 2.5) -- (2.5, 3.5);
\draw [black, thick] (4.5, 2.5) -- (4.5, 3.5);
\draw [black, thick] (5.5, 2.5) -- (5.5, 3.5);
\draw [black, thick] (6.5, 2.5) -- (6.5, 3.5);
\draw [black, thick] (7.5, 2.5) -- (7.5, 3.5);
\draw [black, thick] (8.5, 2.5) -- (8.5, 3.5);
\draw [black, thick] (9.5, 2.5) -- (9.5, 3.5);
\draw [black, thick] (0.5, 3.5) -- (0.5, 4.5);
\draw [black, thick] (3.5, 3.5) -- (3.5, 4.5);
\draw [black, thick] (1.5, 4.5) -- (1.5, 5.5);
\draw [black, thick] (2.5, 4.5) -- (2.5, 5.5);
\draw [black, thick] (4.5, 4.5) -- (4.5, 5.5);
\draw [black, thick] (5.5, 4.5) -- (5.5, 5.5);
\draw [black, thick] (6.5, 4.5) -- (6.5, 5.5);
\draw [black, thick] (7.5, 4.5) -- (7.5, 5.5);
\draw [black, thick] (8.5, 4.5) -- (8.5, 5.5);
\draw [black, thick] (9.5, 4.5) -- (9.5, 5.5);
\draw [black, thick] (0.5, 5.5) -- (0.5, 6.5);
\draw [black, thick] (9.5, 5.5) -- (9.5, 6.5);
\draw [black, thick] (0.5, 6.5) -- (0.5, 7.5);
\draw [black, thick] (1.5, 6.5) -- (1.5, 7.5);
\draw [black, thick] (2.5, 6.5) -- (2.5, 7.5);
\draw [black, thick] (3.5, 6.5) -- (3.5, 7.5);
\draw [black, thick] (4.5, 6.5) -- (4.5, 7.5);
\draw [black, thick] (5.5, 6.5) -- (5.5, 7.5);
\draw [black, thick] (6.5, 6.5) -- (6.5, 7.5);
\draw [black, thick] (7.5, 6.5) -- (7.5, 7.5);
\draw [black, thick] (8.5, 6.5) -- (8.5, 7.5);
\draw [black, thick] (9.5, 6.5) -- (9.5, 7.5);

        \end{tikzpicture}
        \caption{A Hamilton cycle of $G(10,8)$ with $10 \times 8 - 8 = 72$ turns, formed from taking the symmetric difference of the reduced form in Figure \ref{10x8reduced} with a square overlay. Straights are marked with crosses.}
        \label{10x8max}
    \end{figure}

Now consider the case where $n$ is odd. By Lemma \ref{evenrowcol}, there are an even number of turns in each column, and hence every column has an odd number of straights, giving a bound of at least $m$ straights. When $m$ is divisible by $4$, this is attainable, using a tower-like construction. An example is shown below.

    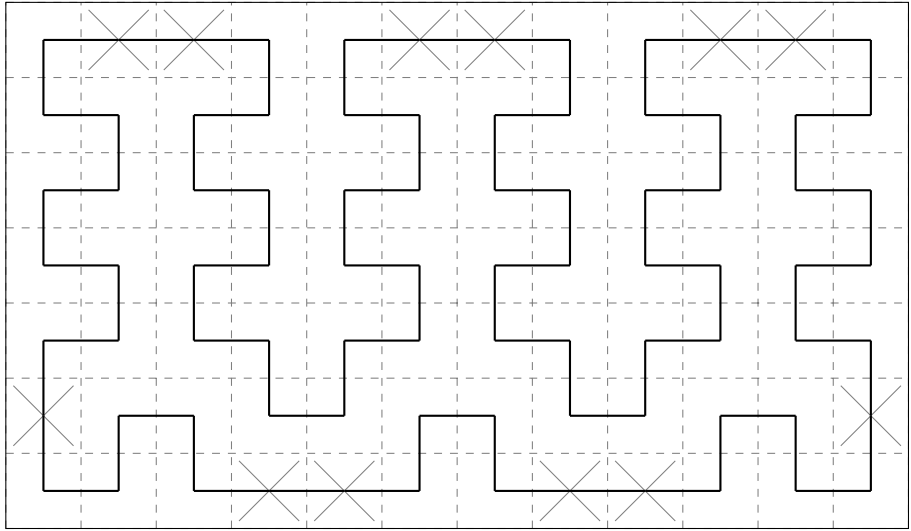
\begin{figure}[H]
        \centering
        \begin{tikzpicture}
\draw[step=1cm,gray,thin,dashed] (0,0) grid (12,7);
\draw [black, thin] (0, 0) -- (12, 0) -- (12, 7) -- (0, 7) -- cycle;
\draw [gray, thin] (3.1, 0.10) -- (3.9, 0.9);
\draw [gray, thin] (3.9, 0.10) -- (3.1, 0.9);
\draw [gray, thin] (4.1, 0.10) -- (4.9, 0.9);
\draw [gray, thin] (4.9, 0.10) -- (4.1, 0.9);
\draw [gray, thin] (7.1, 0.10) -- (7.9, 0.9);
\draw [gray, thin] (7.9, 0.10) -- (7.1, 0.9);
\draw [gray, thin] (8.1, 0.10) -- (8.9, 0.9);
\draw [gray, thin] (8.9, 0.10) -- (8.1, 0.9);
\draw [gray, thin] (0.10, 1.1) -- (0.9, 1.9);
\draw [gray, thin] (0.9, 1.1) -- (0.10, 1.9);
\draw [gray, thin] (11.1, 1.1) -- (11.9, 1.9);
\draw [gray, thin] (11.9, 1.1) -- (11.1, 1.9);
\draw [gray, thin] (1.1, 6.1) -- (1.9, 6.9);
\draw [gray, thin] (1.9, 6.1) -- (1.1, 6.9);
\draw [gray, thin] (2.1, 6.1) -- (2.9, 6.9);
\draw [gray, thin] (2.9, 6.1) -- (2.1, 6.9);
\draw [gray, thin] (5.1, 6.1) -- (5.9, 6.9);
\draw [gray, thin] (5.9, 6.1) -- (5.1, 6.9);
\draw [gray, thin] (6.1, 6.1) -- (6.9, 6.9);
\draw [gray, thin] (6.9, 6.1) -- (6.1, 6.9);
\draw [gray, thin] (9.1, 6.1) -- (9.9, 6.9);
\draw [gray, thin] (9.9, 6.1) -- (9.1, 6.9);
\draw [gray, thin] (10.1, 6.1) -- (10.9, 6.9);
\draw [gray, thin] (10.9, 6.1) -- (10.1, 6.9);
\draw [black, thick] (0.5, 0.5) -- (1.5, 0.5);
\draw [black, thick] (2.5, 0.5) -- (3.5, 0.5);
\draw [black, thick] (3.5, 0.5) -- (4.5, 0.5);
\draw [black, thick] (4.5, 0.5) -- (5.5, 0.5);
\draw [black, thick] (6.5, 0.5) -- (7.5, 0.5);
\draw [black, thick] (7.5, 0.5) -- (8.5, 0.5);
\draw [black, thick] (8.5, 0.5) -- (9.5, 0.5);
\draw [black, thick] (10.5, 0.5) -- (11.5, 0.5);
\draw [black, thick] (1.5, 1.5) -- (2.5, 1.5);
\draw [black, thick] (3.5, 1.5) -- (4.5, 1.5);
\draw [black, thick] (5.5, 1.5) -- (6.5, 1.5);
\draw [black, thick] (7.5, 1.5) -- (8.5, 1.5);
\draw [black, thick] (9.5, 1.5) -- (10.5, 1.5);
\draw [black, thick] (0.5, 2.5) -- (1.5, 2.5);
\draw [black, thick] (2.5, 2.5) -- (3.5, 2.5);
\draw [black, thick] (4.5, 2.5) -- (5.5, 2.5);
\draw [black, thick] (6.5, 2.5) -- (7.5, 2.5);
\draw [black, thick] (8.5, 2.5) -- (9.5, 2.5);
\draw [black, thick] (10.5, 2.5) -- (11.5, 2.5);
\draw [black, thick] (0.5, 3.5) -- (1.5, 3.5);
\draw [black, thick] (2.5, 3.5) -- (3.5, 3.5);
\draw [black, thick] (4.5, 3.5) -- (5.5, 3.5);
\draw [black, thick] (6.5, 3.5) -- (7.5, 3.5);
\draw [black, thick] (8.5, 3.5) -- (9.5, 3.5);
\draw [black, thick] (10.5, 3.5) -- (11.5, 3.5);
\draw [black, thick] (0.5, 4.5) -- (1.5, 4.5);
\draw [black, thick] (2.5, 4.5) -- (3.5, 4.5);
\draw [black, thick] (4.5, 4.5) -- (5.5, 4.5);
\draw [black, thick] (6.5, 4.5) -- (7.5, 4.5);
\draw [black, thick] (8.5, 4.5) -- (9.5, 4.5);
\draw [black, thick] (10.5, 4.5) -- (11.5, 4.5);
\draw [black, thick] (0.5, 5.5) -- (1.5, 5.5);
\draw [black, thick] (2.5, 5.5) -- (3.5, 5.5);
\draw [black, thick] (4.5, 5.5) -- (5.5, 5.5);
\draw [black, thick] (6.5, 5.5) -- (7.5, 5.5);
\draw [black, thick] (8.5, 5.5) -- (9.5, 5.5);
\draw [black, thick] (10.5, 5.5) -- (11.5, 5.5);
\draw [black, thick] (0.5, 6.5) -- (1.5, 6.5);
\draw [black, thick] (1.5, 6.5) -- (2.5, 6.5);
\draw [black, thick] (2.5, 6.5) -- (3.5, 6.5);
\draw [black, thick] (4.5, 6.5) -- (5.5, 6.5);
\draw [black, thick] (5.5, 6.5) -- (6.5, 6.5);
\draw [black, thick] (6.5, 6.5) -- (7.5, 6.5);
\draw [black, thick] (8.5, 6.5) -- (9.5, 6.5);
\draw [black, thick] (9.5, 6.5) -- (10.5, 6.5);
\draw [black, thick] (10.5, 6.5) -- (11.5, 6.5);
\draw [black, thick] (0.5, 0.5) -- (0.5, 1.5);
\draw [black, thick] (1.5, 0.5) -- (1.5, 1.5);
\draw [black, thick] (2.5, 0.5) -- (2.5, 1.5);
\draw [black, thick] (5.5, 0.5) -- (5.5, 1.5);
\draw [black, thick] (6.5, 0.5) -- (6.5, 1.5);
\draw [black, thick] (9.5, 0.5) -- (9.5, 1.5);
\draw [black, thick] (10.5, 0.5) -- (10.5, 1.5);
\draw [black, thick] (11.5, 0.5) -- (11.5, 1.5);
\draw [black, thick] (0.5, 1.5) -- (0.5, 2.5);
\draw [black, thick] (3.5, 1.5) -- (3.5, 2.5);
\draw [black, thick] (4.5, 1.5) -- (4.5, 2.5);
\draw [black, thick] (7.5, 1.5) -- (7.5, 2.5);
\draw [black, thick] (8.5, 1.5) -- (8.5, 2.5);
\draw [black, thick] (11.5, 1.5) -- (11.5, 2.5);
\draw [black, thick] (1.5, 2.5) -- (1.5, 3.5);
\draw [black, thick] (2.5, 2.5) -- (2.5, 3.5);
\draw [black, thick] (5.5, 2.5) -- (5.5, 3.5);
\draw [black, thick] (6.5, 2.5) -- (6.5, 3.5);
\draw [black, thick] (9.5, 2.5) -- (9.5, 3.5);
\draw [black, thick] (10.5, 2.5) -- (10.5, 3.5);
\draw [black, thick] (0.5, 3.5) -- (0.5, 4.5);
\draw [black, thick] (3.5, 3.5) -- (3.5, 4.5);
\draw [black, thick] (4.5, 3.5) -- (4.5, 4.5);
\draw [black, thick] (7.5, 3.5) -- (7.5, 4.5);
\draw [black, thick] (8.5, 3.5) -- (8.5, 4.5);
\draw [black, thick] (11.5, 3.5) -- (11.5, 4.5);
\draw [black, thick] (1.5, 4.5) -- (1.5, 5.5);
\draw [black, thick] (2.5, 4.5) -- (2.5, 5.5);
\draw [black, thick] (5.5, 4.5) -- (5.5, 5.5);
\draw [black, thick] (6.5, 4.5) -- (6.5, 5.5);
\draw [black, thick] (9.5, 4.5) -- (9.5, 5.5);
\draw [black, thick] (10.5, 4.5) -- (10.5, 5.5);
\draw [black, thick] (0.5, 5.5) -- (0.5, 6.5);
\draw [black, thick] (3.5, 5.5) -- (3.5, 6.5);
\draw [black, thick] (4.5, 5.5) -- (4.5, 6.5);
\draw [black, thick] (7.5, 5.5) -- (7.5, 6.5);
\draw [black, thick] (8.5, 5.5) -- (8.5, 6.5);
\draw [black, thick] (11.5, 5.5) -- (11.5, 6.5);

        \end{tikzpicture}
        \caption{A Hamilton cycle on $G(12,7)$, with exactly one straight in each column.}
        \label{12x7max}
    \end{figure}

This can easily be generalised to any odd $n$ and $4|m$, by simply extending the heights of the ``towers" in the above construction, and duplicating the towers horizontally.

For the case where $4\nmid m$, we provide some constructions first:

\begin{figure}[H]
        \centering
        \begin{tikzpicture}
\draw[step=1cm,gray,thin,dashed] (0,0) grid (6,3);
\draw [black, thin] (0, 0) -- (6, 0) -- (6, 3) -- (0, 3) -- cycle;
\draw [gray, thin] (1.1, 0.10) -- (1.9, 0.9);
\draw [gray, thin] (1.9, 0.10) -- (1.1, 0.9);
\draw [gray, thin] (2.1, 0.10) -- (2.9, 0.9);
\draw [gray, thin] (2.9, 0.10) -- (2.1, 0.9);
\draw [gray, thin] (3.1, 0.10) -- (3.9, 0.9);
\draw [gray, thin] (3.9, 0.10) -- (3.1, 0.9);
\draw [gray, thin] (4.1, 0.10) -- (4.9, 0.9);
\draw [gray, thin] (4.9, 0.10) -- (4.1, 0.9);
\draw [gray, thin] (0.10, 1.1) -- (0.9, 1.9);
\draw [gray, thin] (0.9, 1.1) -- (0.10, 1.9);
\draw [gray, thin] (5.1, 1.1) -- (5.9, 1.9);
\draw [gray, thin] (5.9, 1.1) -- (5.1, 1.9);
\draw [black, thick] (0.5, 0.5) -- (1.5, 0.5);
\draw [black, thick] (1.5, 0.5) -- (2.5, 0.5);
\draw [black, thick] (2.5, 0.5) -- (3.5, 0.5);
\draw [black, thick] (3.5, 0.5) -- (4.5, 0.5);
\draw [black, thick] (4.5, 0.5) -- (5.5, 0.5);
\draw [black, thick] (1.5, 1.5) -- (2.5, 1.5);
\draw [black, thick] (3.5, 1.5) -- (4.5, 1.5);
\draw [black, thick] (0.5, 2.5) -- (1.5, 2.5);
\draw [black, thick] (2.5, 2.5) -- (3.5, 2.5);
\draw [black, thick] (4.5, 2.5) -- (5.5, 2.5);
\draw [black, thick] (0.5, 0.5) -- (0.5, 1.5);
\draw [black, thick] (5.5, 0.5) -- (5.5, 1.5);
\draw [black, thick] (0.5, 1.5) -- (0.5, 2.5);
\draw [black, thick] (1.5, 1.5) -- (1.5, 2.5);
\draw [black, thick] (2.5, 1.5) -- (2.5, 2.5);
\draw [black, thick] (3.5, 1.5) -- (3.5, 2.5);
\draw [black, thick] (4.5, 1.5) -- (4.5, 2.5);
\draw [black, thick] (5.5, 1.5) -- (5.5, 2.5);
        \end{tikzpicture}
        \caption{A Hamilton cycle of $G(6,3)$ with $6 \times 3 - 6 = 12$ turns.}
        \label{6x3max}
\end{figure}
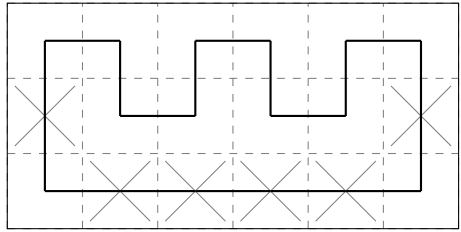

\begin{figure}[H]
        \centering
        \begin{tikzpicture}
\draw[step=0.8cm,gray,thin,dashed] (0,0) grid (11.200,4.0);
\draw [black, thin] (0, 0) -- (11.200, 0) -- (11.200, 4.0) -- (0, 4.0) -- cycle;
\draw [gray, thin] (0.8800, 0.080) -- (1.52, 0.72);
\draw [gray, thin] (1.52, 0.080) -- (0.8800, 0.72);
\draw [gray, thin] (3.28, 0.080) -- (3.92, 0.72);
\draw [gray, thin] (3.92, 0.080) -- (3.28, 0.72);
\draw [gray, thin] (5.68, 0.080) -- (6.32, 0.72);
\draw [gray, thin] (6.32, 0.080) -- (5.68, 0.72);
\draw [gray, thin] (6.48, 0.080) -- (7.120, 0.72);
\draw [gray, thin] (7.120, 0.080) -- (6.48, 0.72);
\draw [gray, thin] (8.88, 0.080) -- (9.520, 0.72);
\draw [gray, thin] (9.520, 0.080) -- (8.88, 0.72);
\draw [gray, thin] (9.68, 0.080) -- (10.32, 0.72);
\draw [gray, thin] (10.32, 0.080) -- (9.68, 0.72);
\draw [gray, thin] (1.68, 0.8800) -- (2.32, 1.52);
\draw [gray, thin] (2.32, 0.8800) -- (1.68, 1.52);
\draw [gray, thin] (2.480, 0.8800) -- (3.12, 1.52);
\draw [gray, thin] (3.12, 0.8800) -- (2.480, 1.52);
\draw [gray, thin] (4.08, 1.68) -- (4.720, 2.32);
\draw [gray, thin] (4.720, 1.68) -- (4.08, 2.32);
\draw [gray, thin] (4.88, 1.68) -- (5.52, 2.32);
\draw [gray, thin] (5.52, 1.68) -- (4.88, 2.32);
\draw [gray, thin] (7.28, 1.68) -- (7.920, 2.32);
\draw [gray, thin] (7.920, 1.68) -- (7.28, 2.32);
\draw [gray, thin] (8.08, 1.68) -- (8.72, 2.32);
\draw [gray, thin] (8.72, 1.68) -- (8.08, 2.32);
\draw [gray, thin] (0.080, 2.480) -- (0.72, 3.12);
\draw [gray, thin] (0.72, 2.480) -- (0.080, 3.12);
\draw [gray, thin] (10.48, 2.480) -- (11.120, 3.12);
\draw [gray, thin] (11.120, 2.480) -- (10.48, 3.12);
\draw [black, thick] (0.4, 0.4) -- (1.20, 0.4);
\draw [black, thick] (1.20, 0.4) -- (2.0, 0.4);
\draw [black, thick] (2.80, 0.4) -- (3.6, 0.4);
\draw [black, thick] (3.6, 0.4) -- (4.4, 0.4);
\draw [black, thick] (5.2, 0.4) -- (6.0, 0.4);
\draw [black, thick] (6.0, 0.4) -- (6.800, 0.4);
\draw [black, thick] (6.800, 0.4) -- (7.60, 0.4);
\draw [black, thick] (8.4, 0.4) -- (9.200, 0.4);
\draw [black, thick] (9.200, 0.4) -- (10.0, 0.4);
\draw [black, thick] (10.0, 0.4) -- (10.8, 0.4);
\draw [black, thick] (0.4, 1.20) -- (1.20, 1.20);
\draw [black, thick] (3.6, 1.20) -- (4.4, 1.20);
\draw [black, thick] (5.2, 1.20) -- (6.0, 1.20);
\draw [black, thick] (6.800, 1.20) -- (7.60, 1.20);
\draw [black, thick] (8.4, 1.20) -- (9.200, 1.20);
\draw [black, thick] (10.0, 1.20) -- (10.8, 1.20);
\draw [black, thick] (0.4, 2.0) -- (1.20, 2.0);
\draw [black, thick] (2.0, 2.0) -- (2.80, 2.0);
\draw [black, thick] (3.6, 2.0) -- (4.4, 2.0);
\draw [black, thick] (4.4, 2.0) -- (5.2, 2.0);
\draw [black, thick] (5.2, 2.0) -- (6.0, 2.0);
\draw [black, thick] (6.800, 2.0) -- (7.60, 2.0);
\draw [black, thick] (7.60, 2.0) -- (8.4, 2.0);
\draw [black, thick] (8.4, 2.0) -- (9.200, 2.0);
\draw [black, thick] (10.0, 2.0) -- (10.8, 2.0);
\draw [black, thick] (1.20, 2.80) -- (2.0, 2.80);
\draw [black, thick] (2.80, 2.80) -- (3.6, 2.80);
\draw [black, thick] (4.4, 2.80) -- (5.2, 2.80);
\draw [black, thick] (6.0, 2.80) -- (6.800, 2.80);
\draw [black, thick] (7.60, 2.80) -- (8.4, 2.80);
\draw [black, thick] (9.200, 2.80) -- (10.0, 2.80);
\draw [black, thick] (0.4, 3.6) -- (1.20, 3.6);
\draw [black, thick] (2.0, 3.6) -- (2.80, 3.6);
\draw [black, thick] (3.6, 3.6) -- (4.4, 3.6);
\draw [black, thick] (5.2, 3.6) -- (6.0, 3.6);
\draw [black, thick] (6.800, 3.6) -- (7.60, 3.6);
\draw [black, thick] (8.4, 3.6) -- (9.20, 3.6);
\draw [black, thick] (10.0, 3.6) -- (10.8, 3.6);
\draw [black, thick] (0.4, 0.4) -- (0.4, 1.20);
\draw [black, thick] (2.0, 0.4) -- (2.0, 1.20);
\draw [black, thick] (2.80, 0.4) -- (2.80, 1.20);
\draw [black, thick] (4.4, 0.4) -- (4.4, 1.20);
\draw [black, thick] (5.2, 0.4) -- (5.2, 1.20);
\draw [black, thick] (7.60, 0.4) -- (7.60, 1.20);
\draw [black, thick] (8.4, 0.4) -- (8.4, 1.20);
\draw [black, thick] (10.8, 0.4) -- (10.8, 1.20);
\draw [black, thick] (1.20, 1.20) -- (1.20, 2.0);
\draw [black, thick] (2.0, 1.20) -- (2.0, 2.0);
\draw [black, thick] (2.80, 1.20) -- (2.80, 2.0);
\draw [black, thick] (3.6, 1.20) -- (3.6, 2.0);
\draw [black, thick] (6.0, 1.20) -- (6.0, 2.0);
\draw [black, thick] (6.800, 1.20) -- (6.800, 2.0);
\draw [black, thick] (9.200, 1.20) -- (9.200, 2.0);
\draw [black, thick] (10.0, 1.20) -- (10.0, 2.0);
\draw [black, thick] (0.4, 2.0) -- (0.4, 2.80);
\draw [black, thick] (10.8, 2.0) -- (10.8, 2.80);
\draw [black, thick] (0.4, 2.80) -- (0.4, 3.6);
\draw [black, thick] (1.20, 2.80) -- (1.20, 3.6);
\draw [black, thick] (2.0, 2.80) -- (2.0, 3.6);
\draw [black, thick] (2.80, 2.80) -- (2.80, 3.6);
\draw [black, thick] (3.6, 2.80) -- (3.6, 3.6);
\draw [black, thick] (4.4, 2.80) -- (4.4, 3.6);
\draw [black, thick] (5.2, 2.80) -- (5.2, 3.6);
\draw [black, thick] (6.0, 2.80) -- (6.0, 3.6);
\draw [black, thick] (6.800, 2.80) -- (6.800, 3.6);
\draw [black, thick] (7.60, 2.80) -- (7.60, 3.6);
\draw [black, thick] (8.4, 2.80) -- (8.4, 3.6);
\draw [black, thick] (9.200, 2.80) -- (9.200, 3.6);
\draw [black, thick] (10.0, 2.80) -- (10.0, 3.6);
\draw [black, thick] (10.8, 2.80) -- (10.8, 3.6);

        \end{tikzpicture}
        \caption{A Hamilton cycle on $G(14,5)$ with $14 \times 5 - 14 = 56$ turns.}
        \label{14x5max}
\end{figure}
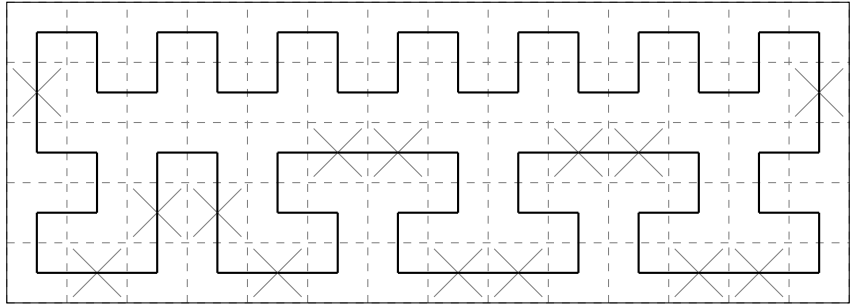

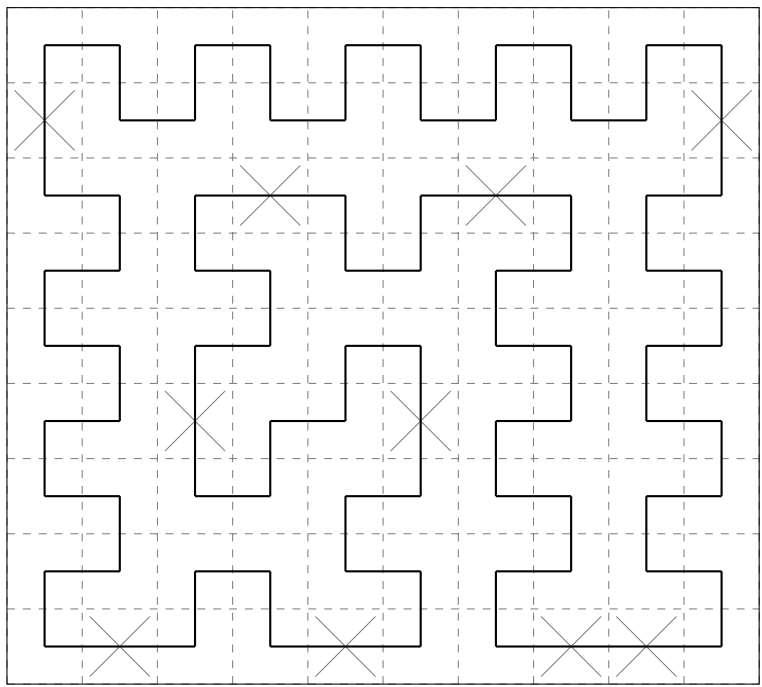
\begin{figure}[H]
        \centering
        \begin{tikzpicture}
\draw[step=1cm,gray,thin,dashed] (0,0) grid (10,9);
\draw [black, thin] (0, 0) -- (10, 0) -- (10, 9) -- (0, 9) -- cycle;
\draw [gray, thin] (1.1, 0.1) -- (1.9, 0.9);
\draw [gray, thin] (1.9, 0.1) -- (1.1, 0.9);
\draw [gray, thin] (4.1, 0.1) -- (4.9, 0.9);
\draw [gray, thin] (4.9, 0.1) -- (4.1, 0.9);
\draw [gray, thin] (7.1, 0.1) -- (7.9, 0.9);
\draw [gray, thin] (7.9, 0.1) -- (7.1, 0.9);
\draw [gray, thin] (8.1, 0.1) -- (8.9, 0.9);
\draw [gray, thin] (8.9, 0.1) -- (8.1, 0.9);
\draw [gray, thin] (2.1, 3.1) -- (2.9, 3.9);
\draw [gray, thin] (2.9, 3.1) -- (2.1, 3.9);
\draw [gray, thin] (5.1, 3.1) -- (5.9, 3.9);
\draw [gray, thin] (5.9, 3.1) -- (5.1, 3.9);
\draw [gray, thin] (3.1, 6.1) -- (3.9, 6.9);
\draw [gray, thin] (3.9, 6.1) -- (3.1, 6.9);
\draw [gray, thin] (6.1, 6.1) -- (6.9, 6.9);
\draw [gray, thin] (6.9, 6.1) -- (6.1, 6.9);
\draw [gray, thin] (0.1, 7.1) -- (0.9, 7.9);
\draw [gray, thin] (0.9, 7.1) -- (0.1, 7.9);
\draw [gray, thin] (9.1, 7.1) -- (9.9, 7.9);
\draw [gray, thin] (9.9, 7.1) -- (9.1, 7.9);
\draw [black, thick] (0.5, 0.5) -- (1.5, 0.5);
\draw [black, thick] (1.5, 0.5) -- (2.5, 0.5);
\draw [black, thick] (3.5, 0.5) -- (4.5, 0.5);
\draw [black, thick] (4.5, 0.5) -- (5.5, 0.5);
\draw [black, thick] (6.5, 0.5) -- (7.5, 0.5);
\draw [black, thick] (7.5, 0.5) -- (8.5, 0.5);
\draw [black, thick] (8.5, 0.5) -- (9.5, 0.5);
\draw [black, thick] (0.5, 1.5) -- (1.5, 1.5);
\draw [black, thick] (2.5, 1.5) -- (3.5, 1.5);
\draw [black, thick] (4.5, 1.5) -- (5.5, 1.5);
\draw [black, thick] (6.5, 1.5) -- (7.5, 1.5);
\draw [black, thick] (8.5, 1.5) -- (9.5, 1.5);
\draw [black, thick] (0.5, 2.5) -- (1.5, 2.5);
\draw [black, thick] (2.5, 2.5) -- (3.5, 2.5);
\draw [black, thick] (4.5, 2.5) -- (5.5, 2.5);
\draw [black, thick] (6.5, 2.5) -- (7.5, 2.5);
\draw [black, thick] (8.5, 2.5) -- (9.5, 2.5);
\draw [black, thick] (0.5, 3.5) -- (1.5, 3.5);
\draw [black, thick] (3.5, 3.5) -- (4.5, 3.5);
\draw [black, thick] (6.5, 3.5) -- (7.5, 3.5);
\draw [black, thick] (8.5, 3.5) -- (9.5, 3.5);
\draw [black, thick] (0.5, 4.5) -- (1.5, 4.5);
\draw [black, thick] (2.5, 4.5) -- (3.5, 4.5);
\draw [black, thick] (4.5, 4.5) -- (5.5, 4.5);
\draw [black, thick] (6.5, 4.5) -- (7.5, 4.5);
\draw [black, thick] (8.5, 4.5) -- (9.5, 4.5);
\draw [black, thick] (0.5, 5.5) -- (1.5, 5.5);
\draw [black, thick] (2.5, 5.5) -- (3.5, 5.5);
\draw [black, thick] (4.5, 5.5) -- (5.5, 5.5);
\draw [black, thick] (6.5, 5.5) -- (7.5, 5.5);
\draw [black, thick] (8.5, 5.5) -- (9.5, 5.5);
\draw [black, thick] (0.5, 6.5) -- (1.5, 6.5);
\draw [black, thick] (2.5, 6.5) -- (3.5, 6.5);
\draw [black, thick] (3.5, 6.5) -- (4.5, 6.5);
\draw [black, thick] (5.5, 6.5) -- (6.5, 6.5);
\draw [black, thick] (6.5, 6.5) -- (7.5, 6.5);
\draw [black, thick] (8.5, 6.5) -- (9.5, 6.5);
\draw [black, thick] (1.5, 7.5) -- (2.5, 7.5);
\draw [black, thick] (3.5, 7.5) -- (4.5, 7.5);
\draw [black, thick] (5.5, 7.5) -- (6.5, 7.5);
\draw [black, thick] (7.5, 7.5) -- (8.5, 7.5);
\draw [black, thick] (0.5, 8.5) -- (1.5, 8.5);
\draw [black, thick] (2.5, 8.5) -- (3.5, 8.5);
\draw [black, thick] (4.5, 8.5) -- (5.5, 8.5);
\draw [black, thick] (6.5, 8.5) -- (7.5, 8.5);
\draw [black, thick] (8.5, 8.5) -- (9.5, 8.5);
\draw [black, thick] (0.5, 0.5) -- (0.5, 1.5);
\draw [black, thick] (2.5, 0.5) -- (2.5, 1.5);
\draw [black, thick] (3.5, 0.5) -- (3.5, 1.5);
\draw [black, thick] (5.5, 0.5) -- (5.5, 1.5);
\draw [black, thick] (6.5, 0.5) -- (6.5, 1.5);
\draw [black, thick] (9.5, 0.5) -- (9.5, 1.5);
\draw [black, thick] (1.5, 1.5) -- (1.5, 2.5);
\draw [black, thick] (4.5, 1.5) -- (4.5, 2.5);
\draw [black, thick] (7.5, 1.5) -- (7.5, 2.5);
\draw [black, thick] (8.5, 1.5) -- (8.5, 2.5);
\draw [black, thick] (0.5, 2.5) -- (0.5, 3.5);
\draw [black, thick] (2.5, 2.5) -- (2.5, 3.5);
\draw [black, thick] (3.5, 2.5) -- (3.5, 3.5);
\draw [black, thick] (5.5, 2.5) -- (5.5, 3.5);
\draw [black, thick] (6.5, 2.5) -- (6.5, 3.5);
\draw [black, thick] (9.5, 2.5) -- (9.5, 3.5);
\draw [black, thick] (1.5, 3.5) -- (1.5, 4.5);
\draw [black, thick] (2.5, 3.5) -- (2.5, 4.5);
\draw [black, thick] (4.5, 3.5) -- (4.5, 4.5);
\draw [black, thick] (5.5, 3.5) -- (5.5, 4.5);
\draw [black, thick] (7.5, 3.5) -- (7.5, 4.5);
\draw [black, thick] (8.5, 3.5) -- (8.5, 4.5);
\draw [black, thick] (0.5, 4.5) -- (0.5, 5.5);
\draw [black, thick] (3.5, 4.5) -- (3.5, 5.5);
\draw [black, thick] (6.5, 4.5) -- (6.5, 5.5);
\draw [black, thick] (9.5, 4.5) -- (9.5, 5.5);
\draw [black, thick] (1.5, 5.5) -- (1.5, 6.5);
\draw [black, thick] (2.5, 5.5) -- (2.5, 6.5);
\draw [black, thick] (4.5, 5.5) -- (4.5, 6.5);
\draw [black, thick] (5.5, 5.5) -- (5.5, 6.5);
\draw [black, thick] (7.5, 5.5) -- (7.5, 6.5);
\draw [black, thick] (8.5, 5.5) -- (8.5, 6.5);
\draw [black, thick] (0.5, 6.5) -- (0.5, 7.5);
\draw [black, thick] (9.5, 6.5) -- (9.5, 7.5);
\draw [black, thick] (0.5, 7.5) -- (0.5, 8.5);
\draw [black, thick] (1.5, 7.5) -- (1.5, 8.5);
\draw [black, thick] (2.5, 7.5) -- (2.5, 8.5);
\draw [black, thick] (3.5, 7.5) -- (3.5, 8.5);
\draw [black, thick] (4.5, 7.5) -- (4.5, 8.5);
\draw [black, thick] (5.5, 7.5) -- (5.5, 8.5);
\draw [black, thick] (6.5, 7.5) -- (6.5, 8.5);
\draw [black, thick] (7.5, 7.5) -- (7.5, 8.5);
\draw [black, thick] (8.5, 7.5) -- (8.5, 8.5);
\draw [black, thick] (9.5, 7.5) -- (9.5, 8.5);

        \end{tikzpicture}
        \caption{A Hamilton cycle on $G(10,9)$ with $10 \times 9 - 10 = 80$ turns.}
\end{figure}

The constructions in Figures \ref{6x3max} and \ref{14x5max} easily generalise to $G(m,3)$ and $G(m,5)$ for any $m\equiv 2\pmod 4,m \geq 6$.

We now use these constructions to prove a more general result.

Call a cycle in $G(m,n)$ {\it extensible} if it contains any of the following two paths (called $T$ and $U$) as subpaths in a corner, up to rotation/reflection of the grid.

    \begin{figure}[H]
        \centering
        \begin{tikzpicture}
\draw[step=1cm,gray,thin,dashed] (0.1,0) grid (4,3.9);
\draw [black, thin] (0, 0) -- (4, 0) -- (4, 4);
\draw [black, thick] (0.5, 0.5) -- (1.5, 0.5);
\draw [black, thick] (1.5, 0.5) -- (2.5, 0.5);
\draw [black, thick] (2.5, 0.5) -- (3.5, 0.5);
\node at (-0.5, 2) {$\cdots$};
\node at (2, 4.5) {$\vdots$};
        \end{tikzpicture}
        \caption{A ``$T$" shape, consisting of two consecutive straights on the edge next to the corner.}
        \label{T-extend}
    \end{figure}
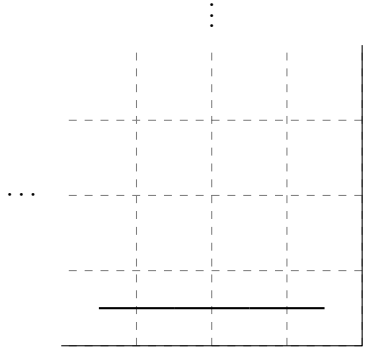
    
    \begin{figure}[H]
        \centering
        \begin{tikzpicture}
\draw[step=1cm,gray,thin,dashed] (0.1,0) grid (4,3.9);
\draw [black, thin] (0, 0) -- (4, 0) -- (4, 4);
\draw [black, thick] (0.5, 0.5) -- (1.5, 0.5);
\draw [black, thick] (2.5, 0.5) -- (3.5, 0.5);
\draw [black, thick] (1.5, 1.5) -- (2.5, 1.5);
\draw [black, thick] (0.5, 0.5) -- (0.5, 1.5);
\draw [black, thick] (1.5, 0.5) -- (1.5, 1.5);
\draw [black, thick] (2.5, 0.5) -- (2.5, 1.5);
\draw [black, thick] (3.5, 0.5) -- (3.5, 1.5);
\draw [black, thick] (0.5, 1.5) -- (0.5, 2.5);
\draw [black, thick] (3.5, 1.5) -- (3.5, 2.5);
\node at (-0.5, 2) {$\cdots$};
\node at (2, 4.5) {$\vdots$};
        \end{tikzpicture}
        \caption{A ``U" shape, pointing downward.}
        \label{U-extend}
    \end{figure}
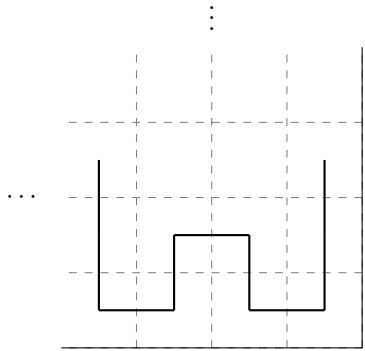

\begin{lemma}
If there exists an extensible Hamilton cycle in $G(m,n)$ with exactly $k$ straights, then there is an extensible Hamilton cycle in $G(m+4,n+4)$ with exactly $k+4$ straights.
\label{extensionlemma}
\end{lemma}

\begin{proof}
Without loss of generality, assume that a $T$ or $U$ is in the bottom-right corner, and in the case of $U$, assume that it points downwards. Extend $G(m,n)$ to $G(m+4,n+4)$ by adding four rows and four columns to the left and bottom.

We first consider the $T$ case.

    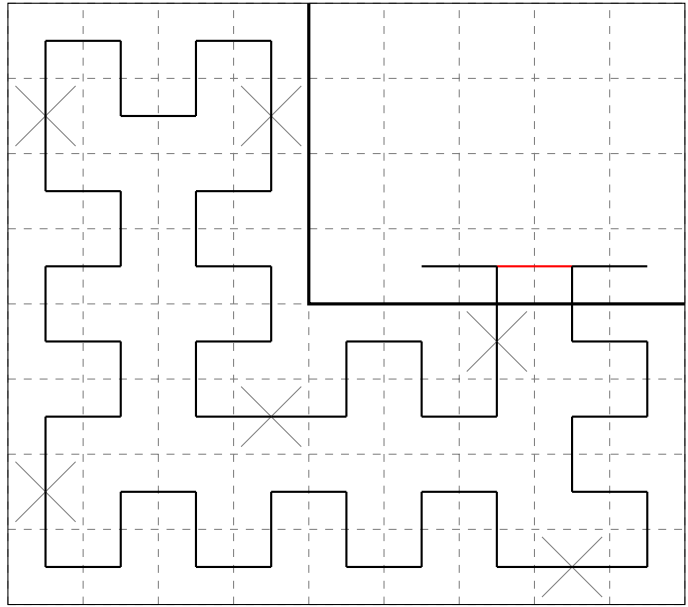
\begin{figure}[H]
        \centering
        \begin{tikzpicture}
\draw[step=1cm,gray,thin,dashed] (0,0) grid (9,8);
\draw [black, thin] (0, 0) -- (9, 0) -- (9, 8) -- (0, 8) -- cycle;
\draw [black, very thick] (4, 8) -- (4, 4) -- (9, 4);
\draw [gray, thin] (7.1, 0.1) -- (7.9, 0.9);
\draw [gray, thin] (7.9, 0.1) -- (7.1, 0.9);
\draw [gray, thin] (0.1, 1.1) -- (0.9, 1.9);
\draw [gray, thin] (0.9, 1.1) -- (0.1, 1.9);
\draw [gray, thin] (3.1, 2.1) -- (3.9, 2.9);
\draw [gray, thin] (3.9, 2.1) -- (3.1, 2.9);
\draw [gray, thin] (6.1, 3.1) -- (6.9, 3.9);
\draw [gray, thin] (6.9, 3.1) -- (6.1, 3.9);
\draw [gray, thin] (0.1, 6.1) -- (0.9, 6.9);
\draw [gray, thin] (0.9, 6.1) -- (0.1, 6.9);
\draw [gray, thin] (3.1, 6.1) -- (3.9, 6.9);
\draw [gray, thin] (3.9, 6.1) -- (3.1, 6.9);

\draw [red, thick] (6.5, 4.5) -- (7.5, 4.5);

\draw [black, thick] (0.5, 0.5) -- (1.5, 0.5);
\draw [black, thick] (2.5, 0.5) -- (3.5, 0.5);
\draw [black, thick] (4.5, 0.5) -- (5.5, 0.5);
\draw [black, thick] (6.5, 0.5) -- (7.5, 0.5);
\draw [black, thick] (7.5, 0.5) -- (8.5, 0.5);
\draw [black, thick] (1.5, 1.5) -- (2.5, 1.5);
\draw [black, thick] (3.5, 1.5) -- (4.5, 1.5);
\draw [black, thick] (5.5, 1.5) -- (6.5, 1.5);
\draw [black, thick] (7.5, 1.5) -- (8.5, 1.5);
\draw [black, thick] (0.5, 2.5) -- (1.5, 2.5);
\draw [black, thick] (2.5, 2.5) -- (3.5, 2.5);
\draw [black, thick] (3.5, 2.5) -- (4.5, 2.5);
\draw [black, thick] (5.5, 2.5) -- (6.5, 2.5);
\draw [black, thick] (7.5, 2.5) -- (8.5, 2.5);
\draw [black, thick] (0.5, 3.5) -- (1.5, 3.5);
\draw [black, thick] (2.5, 3.5) -- (3.5, 3.5);
\draw [black, thick] (4.5, 3.5) -- (5.5, 3.5);
\draw [black, thick] (7.5, 3.5) -- (8.5, 3.5);
\draw [black, thick] (0.5, 4.5) -- (1.5, 4.5);
\draw [black, thick] (2.5, 4.5) -- (3.5, 4.5);
\draw [black, thick] (5.5, 4.5) -- (6.5, 4.5);
\draw [black, thick] (7.5, 4.5) -- (8.5, 4.5);
\draw [black, thick] (0.5, 5.5) -- (1.5, 5.5);
\draw [black, thick] (2.5, 5.5) -- (3.5, 5.5);
\draw [black, thick] (1.5, 6.5) -- (2.5, 6.5);
\draw [black, thick] (0.5, 7.5) -- (1.5, 7.5);
\draw [black, thick] (2.5, 7.5) -- (3.5, 7.5);
\draw [black, thick] (0.5, 0.5) -- (0.5, 1.5);
\draw [black, thick] (1.5, 0.5) -- (1.5, 1.5);
\draw [black, thick] (2.5, 0.5) -- (2.5, 1.5);
\draw [black, thick] (3.5, 0.5) -- (3.5, 1.5);
\draw [black, thick] (4.5, 0.5) -- (4.5, 1.5);
\draw [black, thick] (5.5, 0.5) -- (5.5, 1.5);
\draw [black, thick] (6.5, 0.5) -- (6.5, 1.5);
\draw [black, thick] (8.5, 0.5) -- (8.5, 1.5);
\draw [black, thick] (0.5, 1.5) -- (0.5, 2.5);
\draw [black, thick] (7.5, 1.5) -- (7.5, 2.5);
\draw [black, thick] (1.5, 2.5) -- (1.5, 3.5);
\draw [black, thick] (2.5, 2.5) -- (2.5, 3.5);
\draw [black, thick] (4.5, 2.5) -- (4.5, 3.5);
\draw [black, thick] (5.5, 2.5) -- (5.5, 3.5);
\draw [black, thick] (6.5, 2.5) -- (6.5, 3.5);
\draw [black, thick] (8.5, 2.5) -- (8.5, 3.5);
\draw [black, thick] (0.5, 3.5) -- (0.5, 4.5);
\draw [black, thick] (3.5, 3.5) -- (3.5, 4.5);
\draw [black, thick] (6.5, 3.5) -- (6.5, 4.5);
\draw [black, thick] (7.5, 3.5) -- (7.5, 4.5);
\draw [black, thick] (1.5, 4.5) -- (1.5, 5.5);
\draw [black, thick] (2.5, 4.5) -- (2.5, 5.5);
\draw [black, thick] (0.5, 5.5) -- (0.5, 6.5);
\draw [black, thick] (3.5, 5.5) -- (3.5, 6.5);
\draw [black, thick] (0.5, 6.5) -- (0.5, 7.5);
\draw [black, thick] (1.5, 6.5) -- (1.5, 7.5);
\draw [black, thick] (2.5, 6.5) -- (2.5, 7.5);
\draw [black, thick] (3.5, 6.5) -- (3.5, 7.5);

        \end{tikzpicture}
        \caption{An extensible $G(5,4)$ grid being extended into an extensible $G(9,8)$ grid.}
        \label{T-extend-example}
    \end{figure}

We turn the two straights outwards, and replace them with a straight and a zigzag, as shown in Figure \ref{T-extend-example}. We then draw a zig-zag path right around the outside, adding straights at the corners when needed. By a checkerboard argument, it can be easily seen that the endpoints join up and the new cycle covers all new squares, and that the top-left now contains a $T$ or $U$. Furthermore, there are exactly four more straights in the new grid (since two were removed and six were added).

We now consider the $U$ case, which proceeds in a similar manner.

    \begin{figure}[H]
        \centering
        \begin{tikzpicture}
\draw[step=0.8cm,gray,thin,dashed] (0, 3.2) grid (3.2, 6.4);
\draw [black, thin] (0, 3.2) -- (0, 6.4) -- (3.2, 6.4) -- (3.2, 3.2) -- cycle;

\draw[step=0.8cm,gray,thin,dashed] (4.0, 0) grid (10.4, 6.4);
\draw [black, thin] (4.0, 0) -- (10.4, 0) -- (10.4, 6.4) -- (4.0, 6.4) -- cycle;

\node at (3.6, 4.8) {$\to$};

\draw [black, very thick] (7.2, 6.4) -- (7.2, 3.2) -- (10.4, 3.2);

\draw [black, thick] (4.4, 0.4) -- (5.2, 0.4);
\draw [black, thick] (6.0, 0.4) -- (6.80, 0.4);
\draw [black, thick] (7.60, 0.4) -- (8.4, 0.4);
\draw [black, thick] (9.20, 0.4) -- (10.0, 0.4);
\draw [black, thick] (5.2, 1.20) -- (6.0, 1.20);
\draw [black, thick] (6.80, 1.20) -- (7.60, 1.20);
\draw [black, thick] (8.4, 1.20) -- (9.20, 1.20);
\draw [black, thick] (4.4, 2.0) -- (5.2, 2.0);
\draw [black, thick] (6.0, 2.0) -- (6.80, 2.0);
\draw [black, thick] (6.80, 2.0) -- (7.60, 2.0);
\draw [black, thick] (7.60, 2.0) -- (8.4, 2.0);
\draw [black, thick] (9.20, 2.0) -- (10.0, 2.0);
\draw [black, thick] (4.4, 2.80) -- (5.2, 2.80);
\draw [black, thick] (6.0, 2.80) -- (6.80, 2.80);
\draw [black, thick] (7.60, 2.80) -- (8.4, 2.80);
\draw [black, thick] (9.20, 2.80) -- (10.0, 2.80);
\draw [black, thick] (0.4, 3.6) -- (1.20, 3.6);
\draw [black, thick] (2.0, 3.6) -- (2.80, 3.6);
\draw [black, thick] (4.4, 3.6) -- (5.2, 3.6);
\draw [black, thick] (6.0, 3.6) -- (6.80, 3.6);
\draw [black, thick] (7.60, 3.6) -- (8.4, 3.6);
\draw [black, thick] (9.20, 3.6) -- (10.0, 3.6);
\draw [black, thick] (1.20, 4.4) -- (2.0, 4.4);
\draw [black, thick] (4.4, 4.4) -- (5.2, 4.4);
\draw [black, thick] (6.0, 4.4) -- (6.80, 4.4);
\draw [black, thick] (7.60, 4.4) -- (8.4, 4.4);
\draw [black, thick] (9.20, 4.4) -- (10.0, 4.4);
\draw [black, thick] (5.2, 5.2) -- (6.0, 5.2);
\draw [black, thick] (4.4, 6.0) -- (5.2, 6.0);
\draw [black, thick] (6.0, 6.0) -- (6.80, 6.0);
\draw [black, thick] (4.4, 0.4) -- (4.4, 1.20);
\draw [black, thick] (5.2, 0.4) -- (5.2, 1.20);
\draw [black, thick] (6.0, 0.4) -- (6.0, 1.20);
\draw [black, thick] (6.80, 0.4) -- (6.80, 1.20);
\draw [black, thick] (7.60, 0.4) -- (7.60, 1.20);
\draw [black, thick] (8.4, 0.4) -- (8.4, 1.20);
\draw [black, thick] (9.20, 0.4) -- (9.20, 1.20);
\draw [black, thick] (10.0, 0.4) -- (10.0, 1.20);
\draw [black, thick] (4.4, 1.20) -- (4.4, 2.0);
\draw [black, thick] (10.0, 1.20) -- (10.0, 2.0);
\draw [black, thick] (5.2, 2.0) -- (5.2, 2.80);
\draw [black, thick] (6.0, 2.0) -- (6.0, 2.80);
\draw [black, thick] (8.4, 2.0) -- (8.4, 2.80);
\draw [black, thick] (9.20, 2.0) -- (9.20, 2.80);
\draw [black, thick] (4.4, 2.80) -- (4.4, 3.6);
\draw [black, thick] (6.80, 2.80) -- (6.80, 3.6);
\draw [black, thick] (7.60, 2.80) -- (7.60, 3.6);
\draw [black, thick] (10.0, 2.80) -- (10.0, 3.6);
\draw [black, thick] (0.4, 3.6) -- (0.4, 4.4);
\draw [black, thick] (1.20, 3.6) -- (1.20, 4.4);
\draw [black, thick] (2.0, 3.6) -- (2.0, 4.4);
\draw [black, thick] (2.80, 3.6) -- (2.80, 4.4);
\draw [black, thick] (5.2, 3.6) -- (5.2, 4.4);
\draw [black, thick] (6.0, 3.6) -- (6.0, 4.4);
\draw [black, thick] (8.4, 3.6) -- (8.4, 4.4);
\draw [black, thick] (9.20, 3.6) -- (9.20, 4.4);
\draw [black, thick] (0.4, 4.4) -- (0.4, 5.2);
\draw [black, thick] (2.80, 4.4) -- (2.80, 5.2);
\draw [black, thick] (4.4, 4.4) -- (4.4, 5.2);
\draw [black, thick] (6.80, 4.4) -- (6.80, 5.2);
\draw [black, thick] (7.60, 4.4) -- (7.60, 5.2);
\draw [black, thick] (10.0, 4.4) -- (10.0, 5.2);
\draw [black, thick] (4.4, 5.2) -- (4.4, 6.0);
\draw [black, thick] (5.2, 5.2) -- (5.2, 6.0);
\draw [black, thick] (6.0, 5.2) -- (6.0, 6.0);
\draw [black, thick] (6.80, 5.2) -- (6.80, 6.0);

        \end{tikzpicture}
        \caption{An extensible $G(4,4)$ grid being extended into an extensible $G(8,8)$ grid.}
        \label{U-extend-example}
    \end{figure}

We change the shape of the U, as shown in figure \ref{U-extend-example}.  After exiting the grid, we draw a zig-zag path around the outside, adding straights near corners as needed. By a checkerboard argument, that the endpoints join up and the new cycle covers all new squares, and that the top-left now contains a $T$ or $U$.

So in either case, the lemma holds.
\end{proof}

Claim \ref{max4|n} follows from the provided constructions (which are all extensible, except for $6\times 5$) and Lemma \ref{extensionlemma}.

Claim \ref{max4|n} resolves all cases except for that in which the shorter edge is not $2 \pmod 4$. We now turn our attention to this case.

\subsection{\texorpdfstring{$n\equiv 2 \pmod 4$}{}}\label{2mod4}

Consider $G(m,n)$, where $m > n$ and $n \equiv 2 \pmod 4$.

\begin{lemma}\label{3col}
In a Hamilton cycle on such a grid, every column either has a turn, or is adjacent to a column with a straight. Equivalently, for any three consecutive columns, at least one contains a straight.
\end{lemma}
\begin{proof}
Suppose that three adjacent columns $C_1,C_2,C_3$, in that order, contain no straights. Then these columns contain only turns, and hence every square in these columns is part of exactly one vertical segment. It is clear that these vertical segments must join the cell on row $2a-1$ to the cell on row $2a$ for each positive integer $a$ with $2a\le n$; furthermore, they cannot join any cell on row $2a$ to one on row $2a+1$.

Consider columns $C_1$ and $C_2$, and consider rows $2a-1$ and $2a$. The number of horizontal segments between $C_1$ and $C_2$ in these rows is either $0$ or $1$; two segments would create a disconnected cycle. Since the loop must cross between $C_1$ and $C_2$ an even number of times, and there an an odd number of such pairs of rows, one such pair of rows must contain $0$ horizontal segments between $C_1$ and $C_2$. These two cells in $C_2$ must therefore both continue horizontally into $C_3$, forming a $2 \times 2$ loop, a contradiction.

\end{proof}

\begin{lemma}\label{leftandrightbounds}
In a Hamilton cycle on such a grid, suppose that columns $2b+1$ and $2b+2$ are both all turns, and suppose further that this is the minimal $b$ with this property. Then columns $1$ to $2b$ contain at least $\max(2b,\frac n2+1)$ straights.
\end{lemma}
\begin{proof}
Suppose that rows $2a-1$ and $2a$ both do not contain any straights in the first $2b$ columns. Notice that the turns in the first $2b+2$ columns of rows $2a-1$ and $2a$ must all be turns; columns $2b+1$ and $2b+2$ follow directly from the hypothesis. Then, similar to Lemma \ref{3col}, $(2k-1, 2a-1)$ and $(2k-1, 2a)$ are connected to $(2k, 2a-1)$ and $(2k, 2a)$ respectively for $2k\le 2b+2$. In particular, this means that the two cells in column $2b+1$ and the given rows must go horizontally to row $2b+2$, which forms a $2\times 2$ loop, a contradiction.

Thus each pair of rows must contain at least one straight, giving at least $\frac b2$ straights. Observe that if $(2b, a)$ is connected to $(2b+1, a)$ by a horizontal segment, then the number of turns in row $a$ with $x$-coordinate $\leq 2b$ is odd (this is true in a similar manner to Lemma \ref{evenrowcol}). Hence, there must be a straight in this row, to the left of column $2b+1$. From the proof of Lemma \ref{3col}, there is a pair of row where two such horizontal connections exist, and so at least one pair of rows have two straights. Thus there are at least $\frac b2+1$ straights.

By minimality of $b$, each pair of columns $2c-1,2c$ for $c<b$ must contain at least $1$ straight, and hence by Lemma \ref{evenrowcol} at least $2$ straights. Thus another lower bound for the number of straights is $2b$.
\end{proof}

Using Lemmas \ref{3col} and \ref{leftandrightbounds} we deduce a lower bound on the number of straights, and hence an upper bound on the number of turns.

If any such pair of columns as specified in Lemma \ref{leftandrightbounds} exists, find the relevant pair of columns (with even distance from the left side); in other words, a pair of columns $2a+1,2a+2$ for some $a$. We say that this pair has distance $2a$.

If no such pair exists, then every pair of columns starting from the left hand side has at least $2$ straights, and a lower bound of $2\lfloor\frac m2\rfloor$ straights is achieved.

Similarly, find the first pair of columns with no turns at an even distance from the right side (if it does not exist, we again get a $2\lfloor\frac m2\rfloor$ bound), and say it has distance $2b$. If $2a+2b\ge m$, then we have two straights in each pair of the last $2b$ columns, and two straights in each pair from the left excluding the last $2b$ columns, Again, this achieves the lower bound of $2\lfloor\frac m2\rfloor$ straights.

Finally, we have the case where $2a+2b<m$. Using Lemma \ref{leftandrightbounds} on the left and right sides, we have at least $\max(2a,\frac n2+1)+\max(2b,\frac n2+1)$ straights on these two sides. In the center, via Lemma \ref{3col}, we have at least $2\left\lfloor\frac{m-2a-2b}{3}\right\rfloor$ straights. Hence the two bounds combined give a lower bound for the number of straights of 
$$f(a,b) = \max\left(2a,\frac n2+1\right)+\max\left(2b,\frac n2+1\right)+2\left\lfloor\frac{m-2a-2b}{3}\right\rfloor.$$

Note that $f(a,b)$ is non-increasing in $a$ for $a\le\frac n2+1.$ For $a\ge\frac n2+1$, the first term increases by $f(a+1)-f(a)=2+2\lfloor\frac{m-2(a+1)-2b}{3}\rfloor-2\lfloor\frac{m-2a-2b}{3}\rfloor$. But $\lfloor x \rfloor - \lfloor y\rfloor \le 1$ if $|x-y|<1$, and thus the function is non-decreasing. Hence, over all $a$, a global minimum occurs at $a=\frac n2 +1$. A similar analysis shows that $b=\frac n2+1$ is also minimal, and hence the minimum of this function over all integer $a,b$ is
$$n+2+2\left\lfloor\frac{m-n-2}{3}\right\rfloor.$$

Hence, the minimum number of straights over any $m \times n$ grid of with $m>n$, $n \equiv 2 \pmod 4$ is at least $$\min\left(2\left\lfloor \frac m2\right\rfloor,n+2+2\left\lfloor\frac{m-n-2}{3}\right\rfloor\right).$$

To check which of these hold, we solve $2\lfloor \frac m2\rfloor<n+2+2\lfloor\frac{m-n-2}{3}\rfloor$, we write $m=n+6k+r$ where $0\le r<6$. The inequality becomes $k<1+\lfloor\frac{r-2}{3}\rfloor-\lfloor \frac r2\rfloor$. Checking all cases for $r$, we see that the inequality never holds for positive integers $k,r$.

Thus the minimum can be simplified to $$g(m,n)=n+2+2\left\lfloor\frac{m-n-2}{3}\right\rfloor = 2\left\lfloor \frac{2m+n+2}{6}\right\rfloor,$$
where $g$ denotes the lower bound for the number of straights (and $h$ will later analogously denote the upper bound on the number of straights).

Before we provide constructions, we provide a new way to extend constructions to larger ones.

We call a set of positive integers $S=\{a_1,a_2,\ldots,a_k\}$ {\it evenly extensible} with respect to $n$ if
\begin{itemize}
    \item $|S|=n-1$,
    \item $S$ contains at most one element from each set of the form $\{2a-1,2a\}$ for positive integral $a$,
    \item There is a unique integer $1\le a\le n$ such that $S$ does not contain $2a-1$ or $2a$, and for this $a$, S contains at least one of $2a-2$ and $2a+1$.
\end{itemize}

Given a Hamilton cycle in $G(m,n)$, where $m>n$ and $n\equiv 2\pmod 4$, a line dividing two columns is called an {\it unfolding line} if the set $R=\{r_1,r_2,\ldots,r_k\}$ of row numbers where the Hamilton cycle crosses this line is evenly extensible with respect to $n$.

\begin{lemma}\label{6extensible}
If a Hamilton cycle with $k$ straights in $G(m,n)$ has an unfolding line, then there exists a Hamilton cycle in an $G(m+6,n)$ with $k+4$ straights and an unfolding line. Furthermore, if the grid has a $U$ or $T$ which is not cut by some unfolding line, then the $G(m+6,n)$ grid also has a $U$ or $T$ not cut by some unfolding line.
\end{lemma}

\begin{proof}
By definition, we can split all the rows into contiguous pairs, where all pairs but one contain exactly one crossing over the unfolding line. We can merge the pair $\{2a-1, 2a\}$ with no crossings with the pair that contains a crossing at $2a-2$ or $2a+1$ (at least one of which exists), to get $4$ contiguous rows, with exactly one crossing in one of its middle two rows. Then we deal with the rows as follows, zigzagging all the pairs, and using a special pattern for the four contiguous rows:

    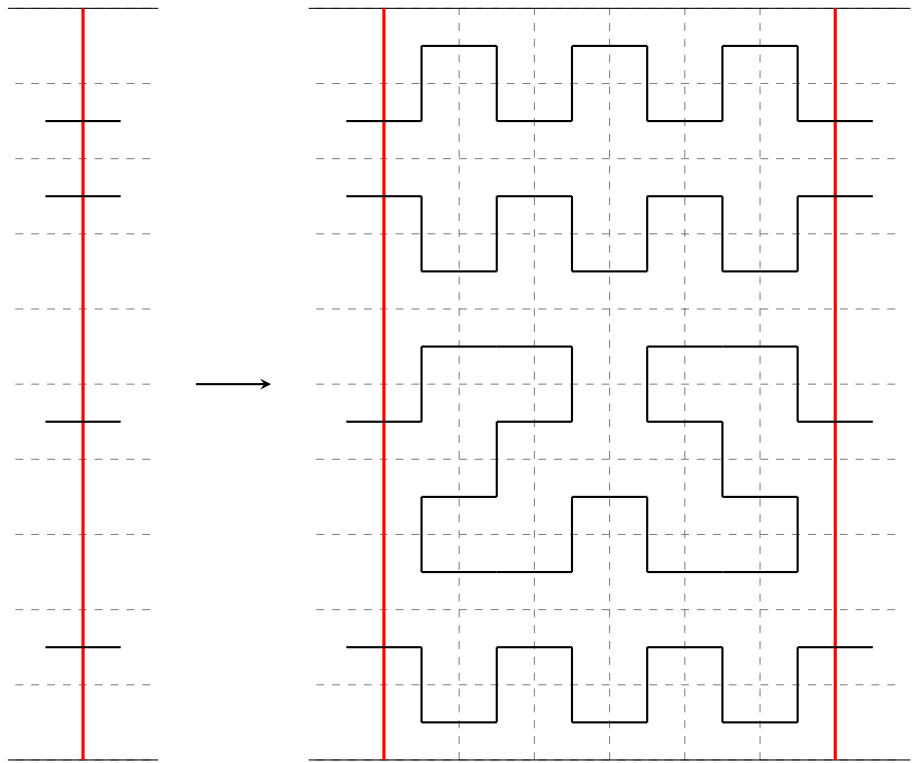
\begin{figure}[H]
        \centering
        \begin{tikzpicture}
\draw[step=1cm,gray,thin,dashed] (0.1,0) grid (1.9,10);
\draw[step=1cm,gray,thin,dashed] (4.1,0) grid (11.9,10);

\draw [black, thin] (0,0) -- (2,0);
\draw [black, thin] (0,10) -- (2,10);
\draw [black, thin] (4,0) -- (12,0);
\draw [black, thin] (4,10) -- (12,10);

\draw [red, very thick] (1,0) -- (1, 10);
\draw [red, very thick] (5,0) -- (5, 10);
\draw [red, very thick] (11,0) -- (11, 10);

\draw [->, black, thick, >=stealth] (2.5, 5) -- (3.5, 5);

\draw [black, thick] (5.5, 0.5) -- (6.5, 0.5);
\draw [black, thick] (7.5, 0.5) -- (8.5, 0.5);
\draw [black, thick] (9.5, 0.5) -- (10.5, 0.5);
\draw [black, thick] (0.5, 1.5) -- (1.5, 1.5);
\draw [black, thick] (4.5, 1.5) -- (5.5, 1.5);
\draw [black, thick] (6.5, 1.5) -- (7.5, 1.5);
\draw [black, thick] (8.5, 1.5) -- (9.5, 1.5);
\draw [black, thick] (10.5, 1.5) -- (11.5, 1.5);
\draw [black, thick] (5.5, 2.5) -- (6.5, 2.5);
\draw [black, thick] (6.5, 2.5) -- (7.5, 2.5);
\draw [black, thick] (8.5, 2.5) -- (9.5, 2.5);
\draw [black, thick] (9.5, 2.5) -- (10.5, 2.5);
\draw [black, thick] (5.5, 3.5) -- (6.5, 3.5);
\draw [black, thick] (7.5, 3.5) -- (8.5, 3.5);
\draw [black, thick] (9.5, 3.5) -- (10.5, 3.5);
\draw [black, thick] (0.5, 4.5) -- (1.5, 4.5);
\draw [black, thick] (4.5, 4.5) -- (5.5, 4.5);
\draw [black, thick] (6.5, 4.5) -- (7.5, 4.5);
\draw [black, thick] (8.5, 4.5) -- (9.5, 4.5);
\draw [black, thick] (10.5, 4.5) -- (11.5, 4.5);
\draw [black, thick] (5.5, 5.5) -- (6.5, 5.5);
\draw [black, thick] (6.5, 5.5) -- (7.5, 5.5);
\draw [black, thick] (8.5, 5.5) -- (9.5, 5.5);
\draw [black, thick] (9.5, 5.5) -- (10.5, 5.5);
\draw [black, thick] (5.5, 6.5) -- (6.5, 6.5);
\draw [black, thick] (7.5, 6.5) -- (8.5, 6.5);
\draw [black, thick] (9.5, 6.5) -- (10.5, 6.5);
\draw [black, thick] (0.5, 7.5) -- (1.5, 7.5);
\draw [black, thick] (4.5, 7.5) -- (5.5, 7.5);
\draw [black, thick] (6.5, 7.5) -- (7.5, 7.5);
\draw [black, thick] (8.5, 7.5) -- (9.5, 7.5);
\draw [black, thick] (10.5, 7.5) -- (11.5, 7.5);
\draw [black, thick] (0.5, 8.5) -- (1.5, 8.5);
\draw [black, thick] (4.5, 8.5) -- (5.5, 8.5);
\draw [black, thick] (6.5, 8.5) -- (7.5, 8.5);
\draw [black, thick] (8.5, 8.5) -- (9.5, 8.5);
\draw [black, thick] (10.5, 8.5) -- (11.5, 8.5);
\draw [black, thick] (5.5, 9.5) -- (6.5, 9.5);
\draw [black, thick] (7.5, 9.5) -- (8.5, 9.5);
\draw [black, thick] (9.5, 9.5) -- (10.5, 9.5);
\draw [black, thick] (5.5, 0.5) -- (5.5, 1.5);
\draw [black, thick] (6.5, 0.5) -- (6.5, 1.5);
\draw [black, thick] (7.5, 0.5) -- (7.5, 1.5);
\draw [black, thick] (8.5, 0.5) -- (8.5, 1.5);
\draw [black, thick] (9.5, 0.5) -- (9.5, 1.5);
\draw [black, thick] (10.5, 0.5) -- (10.5, 1.5);
\draw [black, thick] (5.5, 2.5) -- (5.5, 3.5);
\draw [black, thick] (7.5, 2.5) -- (7.5, 3.5);
\draw [black, thick] (8.5, 2.5) -- (8.5, 3.5);
\draw [black, thick] (10.5, 2.5) -- (10.5, 3.5);
\draw [black, thick] (6.5, 3.5) -- (6.5, 4.5);
\draw [black, thick] (9.5, 3.5) -- (9.5, 4.5);
\draw [black, thick] (5.5, 4.5) -- (5.5, 5.5);
\draw [black, thick] (7.5, 4.5) -- (7.5, 5.5);
\draw [black, thick] (8.5, 4.5) -- (8.5, 5.5);
\draw [black, thick] (10.5, 4.5) -- (10.5, 5.5);
\draw [black, thick] (5.5, 6.5) -- (5.5, 7.5);
\draw [black, thick] (6.5, 6.5) -- (6.5, 7.5);
\draw [black, thick] (7.5, 6.5) -- (7.5, 7.5);
\draw [black, thick] (8.5, 6.5) -- (8.5, 7.5);
\draw [black, thick] (9.5, 6.5) -- (9.5, 7.5);
\draw [black, thick] (10.5, 6.5) -- (10.5, 7.5);
\draw [black, thick] (5.5, 8.5) -- (5.5, 9.5);
\draw [black, thick] (6.5, 8.5) -- (6.5, 9.5);
\draw [black, thick] (7.5, 8.5) -- (7.5, 9.5);
\draw [black, thick] (8.5, 8.5) -- (8.5, 9.5);
\draw [black, thick] (9.5, 8.5) -- (9.5, 9.5);
\draw [black, thick] (10.5, 8.5) -- (10.5, 9.5);

        \end{tikzpicture}
        \caption{Unfolding along a line, where $n=10$. The four special rows here are rows $3,4,5,6$ (where the bottom row is row $1$).}
        \label{6expansiondiagram}
    \end{figure}

It is easy to see that any of the two new red lines can be used as a new unfolding line, and that all the conditions stated in the lemma hold.
\end{proof}

We now give constructions to attain a lower bound on the maximum number of terms, i.e. an upper bound on the minimum number of straights. To do this, we construct Hamilton cycles of $G(m,6)$ that are both unfoldable and extensible (i.e. the grid is extensible, and there exists an unfolding line that does not intersect the relevant $T$ or $U$). Note that increasing both $m$ and $n$ by $4$ increases the lower bound by $4$, as does increasing $m$ by $6$. Thus, the upper bound on the minimum number of straights obtained from these constructions minus the lower bound for the minimum number of straights is unchanged after unfolding or extending.


    \begin{figure}[H]
        \centering
        \begin{tikzpicture}
\draw [red, very thick] (4, 0) -- (4,6);

\draw[step=1cm,gray,thin,dashed] (0,0) grid (6,6);
\draw [black, thin] (0, 0) -- (6, 0) -- (6, 6) -- (0, 6) -- cycle;
\draw [gray, thin] (1.1, 0.1) -- (1.9, 0.9);
\draw [gray, thin] (1.9, 0.1) -- (1.1, 0.9);
\draw [gray, thin] (2.1, 0.1) -- (2.9, 0.9);
\draw [gray, thin] (2.9, 0.1) -- (2.1, 0.9);
\draw [gray, thin] (0.1, 1.1) -- (0.9, 1.9);
\draw [gray, thin] (0.9, 1.1) -- (0.1, 1.9);
\draw [gray, thin] (5.1, 1.1) -- (5.9, 1.9);
\draw [gray, thin] (5.9, 1.1) -- (5.1, 1.9);
\draw [gray, thin] (1.1, 3.1) -- (1.9, 3.9);
\draw [gray, thin] (1.9, 3.1) -- (1.1, 3.9);
\draw [gray, thin] (2.1, 3.1) -- (2.9, 3.9);
\draw [gray, thin] (2.9, 3.1) -- (2.1, 3.9);
\draw [gray, thin] (0.1, 4.1) -- (0.9, 4.9);
\draw [gray, thin] (0.9, 4.1) -- (0.1, 4.9);
\draw [gray, thin] (5.1, 4.1) -- (5.9, 4.9);
\draw [gray, thin] (5.9, 4.1) -- (5.1, 4.9);
\draw [black, thick] (0.5, 0.5) -- (1.5, 0.5);
\draw [black, thick] (1.5, 0.5) -- (2.5, 0.5);
\draw [black, thick] (2.5, 0.5) -- (3.5, 0.5);
\draw [black, thick] (4.5, 0.5) -- (5.5, 0.5);
\draw [black, thick] (1.5, 1.5) -- (2.5, 1.5);
\draw [black, thick] (3.5, 1.5) -- (4.5, 1.5);
\draw [black, thick] (0.5, 2.5) -- (1.5, 2.5);
\draw [black, thick] (2.5, 2.5) -- (3.5, 2.5);
\draw [black, thick] (4.5, 2.5) -- (5.5, 2.5);
\draw [black, thick] (0.5, 3.5) -- (1.5, 3.5);
\draw [black, thick] (1.5, 3.5) -- (2.5, 3.5);
\draw [black, thick] (2.5, 3.5) -- (3.5, 3.5);
\draw [black, thick] (4.5, 3.5) -- (5.5, 3.5);
\draw [black, thick] (1.5, 4.5) -- (2.5, 4.5);
\draw [black, thick] (3.5, 4.5) -- (4.5, 4.5);
\draw [black, thick] (0.5, 5.5) -- (1.5, 5.5);
\draw [black, thick] (2.5, 5.5) -- (3.5, 5.5);
\draw [black, thick] (4.5, 5.5) -- (5.5, 5.5);
\draw [black, thick] (0.5, 0.5) -- (0.5, 1.5);
\draw [black, thick] (3.5, 0.5) -- (3.5, 1.5);
\draw [black, thick] (4.5, 0.5) -- (4.5, 1.5);
\draw [black, thick] (5.5, 0.5) -- (5.5, 1.5);
\draw [black, thick] (0.5, 1.5) -- (0.5, 2.5);
\draw [black, thick] (1.5, 1.5) -- (1.5, 2.5);
\draw [black, thick] (2.5, 1.5) -- (2.5, 2.5);
\draw [black, thick] (5.5, 1.5) -- (5.5, 2.5);
\draw [black, thick] (3.5, 2.5) -- (3.5, 3.5);
\draw [black, thick] (4.5, 2.5) -- (4.5, 3.5);
\draw [black, thick] (0.5, 3.5) -- (0.5, 4.5);
\draw [black, thick] (5.5, 3.5) -- (5.5, 4.5);
\draw [black, thick] (0.5, 4.5) -- (0.5, 5.5);
\draw [black, thick] (1.5, 4.5) -- (1.5, 5.5);
\draw [black, thick] (2.5, 4.5) -- (2.5, 5.5);
\draw [black, thick] (3.5, 4.5) -- (3.5, 5.5);
\draw [black, thick] (4.5, 4.5) -- (4.5, 5.5);
\draw [black, thick] (5.5, 4.5) -- (5.5, 5.5);

        \end{tikzpicture}
        \caption{$8$ straights in $G(6,6)$. This is the case $m-n\equiv0\pmod6$, where the upper bound satisfies $h(m,n)=g(m,n)+2$. In this case, even though $m\not > n$, extending this provides valid constructions.}
        \label{6x6ub}
    \end{figure}
    
    \begin{figure}[H]
        \centering
        \begin{tikzpicture}
\draw [red, very thick] (5, 0) -- (5,6);

\draw[step=1cm,gray,thin,dashed] (0,0) grid (7,6);
\draw [black, thin] (0, 0) -- (7, 0) -- (7, 6) -- (0, 6) -- cycle;
\draw [gray, thin] (1.1, 0.1) -- (1.9, 0.9);
\draw [gray, thin] (1.9, 0.1) -- (1.1, 0.9);
\draw [gray, thin] (6.1, 1.1) -- (6.9, 1.9);
\draw [gray, thin] (6.9, 1.1) -- (6.1, 1.9);
\draw [gray, thin] (3.1, 2.1) -- (3.9, 2.9);
\draw [gray, thin] (3.9, 2.1) -- (3.1, 2.9);
\draw [gray, thin] (0.1, 3.1) -- (0.9, 3.9);
\draw [gray, thin] (0.9, 3.1) -- (0.1, 3.9);
\draw [gray, thin] (0.1, 4.1) -- (0.9, 4.9);
\draw [gray, thin] (0.9, 4.1) -- (0.1, 4.9);
\draw [gray, thin] (3.1, 4.1) -- (3.9, 4.9);
\draw [gray, thin] (3.9, 4.1) -- (3.1, 4.9);
\draw [gray, thin] (6.1, 4.1) -- (6.9, 4.9);
\draw [gray, thin] (6.9, 4.1) -- (6.1, 4.9);
\draw [gray, thin] (1.1, 5.1) -- (1.9, 5.9);
\draw [gray, thin] (1.9, 5.1) -- (1.1, 5.9);
\draw [black, thick] (0.5, 0.5) -- (1.5, 0.5);
\draw [black, thick] (1.5, 0.5) -- (2.5, 0.5);
\draw [black, thick] (3.5, 0.5) -- (4.5, 0.5);
\draw [black, thick] (5.5, 0.5) -- (6.5, 0.5);
\draw [black, thick] (0.5, 1.5) -- (1.5, 1.5);
\draw [black, thick] (2.5, 1.5) -- (3.5, 1.5);
\draw [black, thick] (4.5, 1.5) -- (5.5, 1.5);
\draw [black, thick] (0.5, 2.5) -- (1.5, 2.5);
\draw [black, thick] (2.5, 2.5) -- (3.5, 2.5);
\draw [black, thick] (3.5, 2.5) -- (4.5, 2.5);
\draw [black, thick] (5.5, 2.5) -- (6.5, 2.5);
\draw [black, thick] (1.5, 3.5) -- (2.5, 3.5);
\draw [black, thick] (3.5, 3.5) -- (4.5, 3.5);
\draw [black, thick] (5.5, 3.5) -- (6.5, 3.5);
\draw [black, thick] (1.5, 4.5) -- (2.5, 4.5);
\draw [black, thick] (4.5, 4.5) -- (5.5, 4.5);
\draw [black, thick] (0.5, 5.5) -- (1.5, 5.5);
\draw [black, thick] (1.5, 5.5) -- (2.5, 5.5);
\draw [black, thick] (3.5, 5.5) -- (4.5, 5.5);
\draw [black, thick] (5.5, 5.5) -- (6.5, 5.5);
\draw [black, thick] (0.5, 0.5) -- (0.5, 1.5);
\draw [black, thick] (2.5, 0.5) -- (2.5, 1.5);
\draw [black, thick] (3.5, 0.5) -- (3.5, 1.5);
\draw [black, thick] (4.5, 0.5) -- (4.5, 1.5);
\draw [black, thick] (5.5, 0.5) -- (5.5, 1.5);
\draw [black, thick] (6.5, 0.5) -- (6.5, 1.5);
\draw [black, thick] (1.5, 1.5) -- (1.5, 2.5);
\draw [black, thick] (6.5, 1.5) -- (6.5, 2.5);
\draw [black, thick] (0.5, 2.5) -- (0.5, 3.5);
\draw [black, thick] (2.5, 2.5) -- (2.5, 3.5);
\draw [black, thick] (4.5, 2.5) -- (4.5, 3.5);
\draw [black, thick] (5.5, 2.5) -- (5.5, 3.5);
\draw [black, thick] (0.5, 3.5) -- (0.5, 4.5);
\draw [black, thick] (1.5, 3.5) -- (1.5, 4.5);
\draw [black, thick] (3.5, 3.5) -- (3.5, 4.5);
\draw [black, thick] (6.5, 3.5) -- (6.5, 4.5);
\draw [black, thick] (0.5, 4.5) -- (0.5, 5.5);
\draw [black, thick] (2.5, 4.5) -- (2.5, 5.5);
\draw [black, thick] (3.5, 4.5) -- (3.5, 5.5);
\draw [black, thick] (4.5, 4.5) -- (4.5, 5.5);
\draw [black, thick] (5.5, 4.5) -- (5.5, 5.5);
\draw [black, thick] (6.5, 4.5) -- (6.5, 5.5);

        \end{tikzpicture}
        \caption{$8$ straights in $G(7,6)$. This is the case $m-n\equiv1\pmod6$, where the upper bound satisfies $h(m,n)=g(m,n)+2$.}
        \label{7x6ub}
    \end{figure}
    
    \begin{figure}[H]
        \centering
        \begin{tikzpicture}
\draw [red, very thick] (6, 0) -- (6,6);

\draw[step=1cm,gray,thin,dashed] (0,0) grid (8,6);
\draw [black, thin] (0, 0) -- (8, 0) -- (8, 6) -- (0, 6) -- cycle;
\draw [gray, thin] (0.1, 1.1) -- (0.9, 1.9);
\draw [gray, thin] (0.9, 1.1) -- (0.1, 1.9);
\draw [gray, thin] (7.1, 1.1) -- (7.9, 1.9);
\draw [gray, thin] (7.9, 1.1) -- (7.1, 1.9);
\draw [gray, thin] (3.1, 2.1) -- (3.9, 2.9);
\draw [gray, thin] (3.9, 2.1) -- (3.1, 2.9);
\draw [gray, thin] (4.1, 2.1) -- (4.9, 2.9);
\draw [gray, thin] (4.9, 2.1) -- (4.1, 2.9);
\draw [gray, thin] (0.1, 4.1) -- (0.9, 4.9);
\draw [gray, thin] (0.9, 4.1) -- (0.1, 4.9);
\draw [gray, thin] (3.1, 4.1) -- (3.9, 4.9);
\draw [gray, thin] (3.9, 4.1) -- (3.1, 4.9);
\draw [gray, thin] (4.1, 4.1) -- (4.9, 4.9);
\draw [gray, thin] (4.9, 4.1) -- (4.1, 4.9);
\draw [gray, thin] (7.1, 4.1) -- (7.9, 4.9);
\draw [gray, thin] (7.9, 4.1) -- (7.1, 4.9);
\draw [black, thick] (0.5, 0.5) -- (1.5, 0.5);
\draw [black, thick] (2.5, 0.5) -- (3.5, 0.5);
\draw [black, thick] (4.5, 0.5) -- (5.5, 0.5);
\draw [black, thick] (6.5, 0.5) -- (7.5, 0.5);
\draw [black, thick] (1.5, 1.5) -- (2.5, 1.5);
\draw [black, thick] (3.5, 1.5) -- (4.5, 1.5);
\draw [black, thick] (5.5, 1.5) -- (6.5, 1.5);
\draw [black, thick] (0.5, 2.5) -- (1.5, 2.5);
\draw [black, thick] (2.5, 2.5) -- (3.5, 2.5);
\draw [black, thick] (3.5, 2.5) -- (4.5, 2.5);
\draw [black, thick] (4.5, 2.5) -- (5.5, 2.5);
\draw [black, thick] (6.5, 2.5) -- (7.5, 2.5);
\draw [black, thick] (0.5, 3.5) -- (1.5, 3.5);
\draw [black, thick] (2.5, 3.5) -- (3.5, 3.5);
\draw [black, thick] (4.5, 3.5) -- (5.5, 3.5);
\draw [black, thick] (6.5, 3.5) -- (7.5, 3.5);
\draw [black, thick] (1.5, 4.5) -- (2.5, 4.5);
\draw [black, thick] (5.5, 4.5) -- (6.5, 4.5);
\draw [black, thick] (0.5, 5.5) -- (1.5, 5.5);
\draw [black, thick] (2.5, 5.5) -- (3.5, 5.5);
\draw [black, thick] (4.5, 5.5) -- (5.5, 5.5);
\draw [black, thick] (6.5, 5.5) -- (7.5, 5.5);
\draw [black, thick] (0.5, 0.5) -- (0.5, 1.5);
\draw [black, thick] (1.5, 0.5) -- (1.5, 1.5);
\draw [black, thick] (2.5, 0.5) -- (2.5, 1.5);
\draw [black, thick] (3.5, 0.5) -- (3.5, 1.5);
\draw [black, thick] (4.5, 0.5) -- (4.5, 1.5);
\draw [black, thick] (5.5, 0.5) -- (5.5, 1.5);
\draw [black, thick] (6.5, 0.5) -- (6.5, 1.5);
\draw [black, thick] (7.5, 0.5) -- (7.5, 1.5);
\draw [black, thick] (0.5, 1.5) -- (0.5, 2.5);
\draw [black, thick] (7.5, 1.5) -- (7.5, 2.5);
\draw [black, thick] (1.5, 2.5) -- (1.5, 3.5);
\draw [black, thick] (2.5, 2.5) -- (2.5, 3.5);
\draw [black, thick] (5.5, 2.5) -- (5.5, 3.5);
\draw [black, thick] (6.5, 2.5) -- (6.5, 3.5);
\draw [black, thick] (0.5, 3.5) -- (0.5, 4.5);
\draw [black, thick] (3.5, 3.5) -- (3.5, 4.5);
\draw [black, thick] (4.5, 3.5) -- (4.5, 4.5);
\draw [black, thick] (7.5, 3.5) -- (7.5, 4.5);
\draw [black, thick] (0.5, 4.5) -- (0.5, 5.5);
\draw [black, thick] (1.5, 4.5) -- (1.5, 5.5);
\draw [black, thick] (2.5, 4.5) -- (2.5, 5.5);
\draw [black, thick] (3.5, 4.5) -- (3.5, 5.5);
\draw [black, thick] (4.5, 4.5) -- (4.5, 5.5);
\draw [black, thick] (5.5, 4.5) -- (5.5, 5.5);
\draw [black, thick] (6.5, 4.5) -- (6.5, 5.5);
\draw [black, thick] (7.5, 4.5) -- (7.5, 5.5);

        \end{tikzpicture}
        \caption{$8$ straights in $G(8,6)$. This is the case $m-n\equiv2\pmod6$, where the upper bound satisfies $h(m,n)=g(m,n)$.}
        \label{8x6ub}
    \end{figure}
    
    \begin{figure}[H]
        \centering
        \begin{tikzpicture}

\draw [red, very thick] (7, 0) -- (7,6);

\draw[step=1cm,gray,thin,dashed] (0,0) grid (9,6);
\draw [black, thin] (0, 0) -- (9, 0) -- (9, 6) -- (0, 6) -- cycle;
\draw [gray, thin] (1.1, 0.1) -- (1.9, 0.9);
\draw [gray, thin] (1.9, 0.1) -- (1.1, 0.9);
\draw [gray, thin] (8.1, 1.1) -- (8.9, 1.9);
\draw [gray, thin] (8.9, 1.1) -- (8.1, 1.9);
\draw [gray, thin] (5.1, 2.1) -- (5.9, 2.9);
\draw [gray, thin] (5.9, 2.1) -- (5.1, 2.9);
\draw [gray, thin] (2.1, 3.1) -- (2.9, 3.9);
\draw [gray, thin] (2.9, 3.1) -- (2.1, 3.9);
\draw [gray, thin] (8.1, 4.1) -- (8.9, 4.9);
\draw [gray, thin] (8.9, 4.1) -- (8.1, 4.9);
\draw [gray, thin] (1.1, 5.1) -- (1.9, 5.9);
\draw [gray, thin] (1.9, 5.1) -- (1.1, 5.9);
\draw [gray, thin] (2.1, 5.1) -- (2.9, 5.9);
\draw [gray, thin] (2.9, 5.1) -- (2.1, 5.9);
\draw [gray, thin] (5.1, 5.1) -- (5.9, 5.9);
\draw [gray, thin] (5.9, 5.1) -- (5.1, 5.9);
\draw [black, thick] (0.5, 0.5) -- (1.5, 0.5);
\draw [black, thick] (1.5, 0.5) -- (2.5, 0.5);
\draw [black, thick] (3.5, 0.5) -- (4.5, 0.5);
\draw [black, thick] (5.5, 0.5) -- (6.5, 0.5);
\draw [black, thick] (7.5, 0.5) -- (8.5, 0.5);
\draw [black, thick] (0.5, 1.5) -- (1.5, 1.5);
\draw [black, thick] (2.5, 1.5) -- (3.5, 1.5);
\draw [black, thick] (4.5, 1.5) -- (5.5, 1.5);
\draw [black, thick] (6.5, 1.5) -- (7.5, 1.5);
\draw [black, thick] (0.5, 2.5) -- (1.5, 2.5);
\draw [black, thick] (2.5, 2.5) -- (3.5, 2.5);
\draw [black, thick] (4.5, 2.5) -- (5.5, 2.5);
\draw [black, thick] (5.5, 2.5) -- (6.5, 2.5);
\draw [black, thick] (7.5, 2.5) -- (8.5, 2.5);
\draw [black, thick] (0.5, 3.5) -- (1.5, 3.5);
\draw [black, thick] (3.5, 3.5) -- (4.5, 3.5);
\draw [black, thick] (5.5, 3.5) -- (6.5, 3.5);
\draw [black, thick] (7.5, 3.5) -- (8.5, 3.5);
\draw [black, thick] (0.5, 4.5) -- (1.5, 4.5);
\draw [black, thick] (2.5, 4.5) -- (3.5, 4.5);
\draw [black, thick] (4.5, 4.5) -- (5.5, 4.5);
\draw [black, thick] (6.5, 4.5) -- (7.5, 4.5);
\draw [black, thick] (0.5, 5.5) -- (1.5, 5.5);
\draw [black, thick] (1.5, 5.5) -- (2.5, 5.5);
\draw [black, thick] (2.5, 5.5) -- (3.5, 5.5);
\draw [black, thick] (4.5, 5.5) -- (5.5, 5.5);
\draw [black, thick] (5.5, 5.5) -- (6.5, 5.5);
\draw [black, thick] (7.5, 5.5) -- (8.5, 5.5);
\draw [black, thick] (0.5, 0.5) -- (0.5, 1.5);
\draw [black, thick] (2.5, 0.5) -- (2.5, 1.5);
\draw [black, thick] (3.5, 0.5) -- (3.5, 1.5);
\draw [black, thick] (4.5, 0.5) -- (4.5, 1.5);
\draw [black, thick] (5.5, 0.5) -- (5.5, 1.5);
\draw [black, thick] (6.5, 0.5) -- (6.5, 1.5);
\draw [black, thick] (7.5, 0.5) -- (7.5, 1.5);
\draw [black, thick] (8.5, 0.5) -- (8.5, 1.5);
\draw [black, thick] (1.5, 1.5) -- (1.5, 2.5);
\draw [black, thick] (8.5, 1.5) -- (8.5, 2.5);
\draw [black, thick] (0.5, 2.5) -- (0.5, 3.5);
\draw [black, thick] (2.5, 2.5) -- (2.5, 3.5);
\draw [black, thick] (3.5, 2.5) -- (3.5, 3.5);
\draw [black, thick] (4.5, 2.5) -- (4.5, 3.5);
\draw [black, thick] (6.5, 2.5) -- (6.5, 3.5);
\draw [black, thick] (7.5, 2.5) -- (7.5, 3.5);
\draw [black, thick] (1.5, 3.5) -- (1.5, 4.5);
\draw [black, thick] (2.5, 3.5) -- (2.5, 4.5);
\draw [black, thick] (5.5, 3.5) -- (5.5, 4.5);
\draw [black, thick] (8.5, 3.5) -- (8.5, 4.5);
\draw [black, thick] (0.5, 4.5) -- (0.5, 5.5);
\draw [black, thick] (3.5, 4.5) -- (3.5, 5.5);
\draw [black, thick] (4.5, 4.5) -- (4.5, 5.5);
\draw [black, thick] (6.5, 4.5) -- (6.5, 5.5);
\draw [black, thick] (7.5, 4.5) -- (7.5, 5.5);
\draw [black, thick] (8.5, 4.5) -- (8.5, 5.5);

        \end{tikzpicture}
        \caption{$10$ straights in $G(9,6)$. This is the case $m-n\equiv3\pmod6$, where the upper bound satisfies $h(m,n)=g(m,n)+2$.}
        \label{9x6ub}
    \end{figure}
    
    \begin{figure}[H]
        \centering
        \begin{tikzpicture}
\draw [red, very thick] (8, 0) -- (8, 6);

\draw[step=1cm,gray,thin,dashed] (0,0) grid (10,6);
\draw [black, thin] (0, 0) -- (10, 0) -- (10, 6) -- (0, 6) -- cycle;
\draw [gray, thin] (0.1, 1.1) -- (0.9, 1.9);
\draw [gray, thin] (0.9, 1.1) -- (0.1, 1.9);
\draw [gray, thin] (9.1, 1.1) -- (9.9, 1.9);
\draw [gray, thin] (9.9, 1.1) -- (9.1, 1.9);
\draw [gray, thin] (3.1, 2.1) -- (3.9, 2.9);
\draw [gray, thin] (3.9, 2.1) -- (3.1, 2.9);
\draw [gray, thin] (6.1, 2.1) -- (6.9, 2.9);
\draw [gray, thin] (6.9, 2.1) -- (6.1, 2.9);
\draw [gray, thin] (0.1, 4.1) -- (0.9, 4.9);
\draw [gray, thin] (0.9, 4.1) -- (0.1, 4.9);
\draw [gray, thin] (3.1, 4.1) -- (3.9, 4.9);
\draw [gray, thin] (3.9, 4.1) -- (3.1, 4.9);
\draw [gray, thin] (5.1, 4.1) -- (5.9, 4.9);
\draw [gray, thin] (5.9, 4.1) -- (5.1, 4.9);
\draw [gray, thin] (9.1, 4.1) -- (9.9, 4.9);
\draw [gray, thin] (9.9, 4.1) -- (9.1, 4.9);
\draw [gray, thin] (5.1, 5.1) -- (5.9, 5.9);
\draw [gray, thin] (5.9, 5.1) -- (5.1, 5.9);
\draw [gray, thin] (6.1, 5.1) -- (6.9, 5.9);
\draw [gray, thin] (6.9, 5.1) -- (6.1, 5.9);
\draw [black, thick] (0.5, 0.5) -- (1.5, 0.5);
\draw [black, thick] (2.5, 0.5) -- (3.5, 0.5);
\draw [black, thick] (4.5, 0.5) -- (5.5, 0.5);
\draw [black, thick] (6.5, 0.5) -- (7.5, 0.5);
\draw [black, thick] (8.5, 0.5) -- (9.5, 0.5);
\draw [black, thick] (1.5, 1.5) -- (2.5, 1.5);
\draw [black, thick] (3.5, 1.5) -- (4.5, 1.5);
\draw [black, thick] (5.5, 1.5) -- (6.5, 1.5);
\draw [black, thick] (7.5, 1.5) -- (8.5, 1.5);
\draw [black, thick] (0.5, 2.5) -- (1.5, 2.5);
\draw [black, thick] (2.5, 2.5) -- (3.5, 2.5);
\draw [black, thick] (3.5, 2.5) -- (4.5, 2.5);
\draw [black, thick] (5.5, 2.5) -- (6.5, 2.5);
\draw [black, thick] (6.5, 2.5) -- (7.5, 2.5);
\draw [black, thick] (8.5, 2.5) -- (9.5, 2.5);
\draw [black, thick] (0.5, 3.5) -- (1.5, 3.5);
\draw [black, thick] (2.5, 3.5) -- (3.5, 3.5);
\draw [black, thick] (4.5, 3.5) -- (5.5, 3.5);
\draw [black, thick] (6.5, 3.5) -- (7.5, 3.5);
\draw [black, thick] (8.5, 3.5) -- (9.5, 3.5);
\draw [black, thick] (1.5, 4.5) -- (2.5, 4.5);
\draw [black, thick] (4.5, 4.5) -- (5.5, 4.5);
\draw [black, thick] (5.5, 4.5) -- (6.5, 4.5);
\draw [black, thick] (7.5, 4.5) -- (8.5, 4.5);
\draw [black, thick] (0.5, 5.5) -- (1.5, 5.5);
\draw [black, thick] (2.5, 5.5) -- (3.5, 5.5);
\draw [black, thick] (4.5, 5.5) -- (5.5, 5.5);
\draw [black, thick] (5.5, 5.5) -- (6.5, 5.5);
\draw [black, thick] (6.5, 5.5) -- (7.5, 5.5);
\draw [black, thick] (8.5, 5.5) -- (9.5, 5.5);
\draw [black, thick] (0.5, 0.5) -- (0.5, 1.5);
\draw [black, thick] (1.5, 0.5) -- (1.5, 1.5);
\draw [black, thick] (2.5, 0.5) -- (2.5, 1.5);
\draw [black, thick] (3.5, 0.5) -- (3.5, 1.5);
\draw [black, thick] (4.5, 0.5) -- (4.5, 1.5);
\draw [black, thick] (5.5, 0.5) -- (5.5, 1.5);
\draw [black, thick] (6.5, 0.5) -- (6.5, 1.5);
\draw [black, thick] (7.5, 0.5) -- (7.5, 1.5);
\draw [black, thick] (8.5, 0.5) -- (8.5, 1.5);
\draw [black, thick] (9.5, 0.5) -- (9.5, 1.5);
\draw [black, thick] (0.5, 1.5) -- (0.5, 2.5);
\draw [black, thick] (9.5, 1.5) -- (9.5, 2.5);
\draw [black, thick] (1.5, 2.5) -- (1.5, 3.5);
\draw [black, thick] (2.5, 2.5) -- (2.5, 3.5);
\draw [black, thick] (4.5, 2.5) -- (4.5, 3.5);
\draw [black, thick] (5.5, 2.5) -- (5.5, 3.5);
\draw [black, thick] (7.5, 2.5) -- (7.5, 3.5);
\draw [black, thick] (8.5, 2.5) -- (8.5, 3.5);
\draw [black, thick] (0.5, 3.5) -- (0.5, 4.5);
\draw [black, thick] (3.5, 3.5) -- (3.5, 4.5);
\draw [black, thick] (6.5, 3.5) -- (6.5, 4.5);
\draw [black, thick] (9.5, 3.5) -- (9.5, 4.5);
\draw [black, thick] (0.5, 4.5) -- (0.5, 5.5);
\draw [black, thick] (1.5, 4.5) -- (1.5, 5.5);
\draw [black, thick] (2.5, 4.5) -- (2.5, 5.5);
\draw [black, thick] (3.5, 4.5) -- (3.5, 5.5);
\draw [black, thick] (4.5, 4.5) -- (4.5, 5.5);
\draw [black, thick] (7.5, 4.5) -- (7.5, 5.5);
\draw [black, thick] (8.5, 4.5) -- (8.5, 5.5);
\draw [black, thick] (9.5, 4.5) -- (9.5, 5.5);

        \end{tikzpicture}
        \caption{$10$ straights in $G(10,6)$. This is the case $m-n\equiv4\pmod6$, where the upper bound satisfies $h(m,n)=g(m,n)+2$.}
        \label{10x6ub}
    \end{figure}
    
    \begin{figure}[H]
        \centering
        \begin{tikzpicture}
\draw [red, very thick] (9, 0) -- (9, 6);

\draw[step=1cm,gray,thin,dashed] (0,0) grid (11,6);
\draw [black, thin] (0, 0) -- (11, 0) -- (11, 6) -- (0, 6) -- cycle;
\draw [gray, thin] (1.1, 0.1) -- (1.9, 0.9);
\draw [gray, thin] (1.9, 0.1) -- (1.1, 0.9);
\draw [gray, thin] (2.1, 0.1) -- (2.9, 0.9);
\draw [gray, thin] (2.9, 0.1) -- (2.1, 0.9);
\draw [gray, thin] (3.1, 0.1) -- (3.9, 0.9);
\draw [gray, thin] (3.9, 0.1) -- (3.1, 0.9);
\draw [gray, thin] (10.1, 1.1) -- (10.9, 1.9);
\draw [gray, thin] (10.9, 1.1) -- (10.1, 1.9);
\draw [gray, thin] (2.1, 2.1) -- (2.9, 2.9);
\draw [gray, thin] (2.9, 2.1) -- (2.1, 2.9);
\draw [gray, thin] (4.1, 2.1) -- (4.9, 2.9);
\draw [gray, thin] (4.9, 2.1) -- (4.1, 2.9);
\draw [gray, thin] (7.1, 2.1) -- (7.9, 2.9);
\draw [gray, thin] (7.9, 2.1) -- (7.1, 2.9);
\draw [gray, thin] (3.1, 3.1) -- (3.9, 3.9);
\draw [gray, thin] (3.9, 3.1) -- (3.1, 3.9);
\draw [gray, thin] (10.1, 4.1) -- (10.9, 4.9);
\draw [gray, thin] (10.9, 4.1) -- (10.1, 4.9);
\draw [gray, thin] (1.1, 5.1) -- (1.9, 5.9);
\draw [gray, thin] (1.9, 5.1) -- (1.1, 5.9);
\draw [gray, thin] (4.1, 5.1) -- (4.9, 5.9);
\draw [gray, thin] (4.9, 5.1) -- (4.1, 5.9);
\draw [gray, thin] (7.1, 5.1) -- (7.9, 5.9);
\draw [gray, thin] (7.9, 5.1) -- (7.1, 5.9);
\draw [black, thick] (0.5, 0.5) -- (1.5, 0.5);
\draw [black, thick] (1.5, 0.5) -- (2.5, 0.5);
\draw [black, thick] (2.5, 0.5) -- (3.5, 0.5);
\draw [black, thick] (3.5, 0.5) -- (4.5, 0.5);
\draw [black, thick] (5.5, 0.5) -- (6.5, 0.5);
\draw [black, thick] (7.5, 0.5) -- (8.5, 0.5);
\draw [black, thick] (9.5, 0.5) -- (10.5, 0.5);
\draw [black, thick] (0.5, 1.5) -- (1.5, 1.5);
\draw [black, thick] (2.5, 1.5) -- (3.5, 1.5);
\draw [black, thick] (4.5, 1.5) -- (5.5, 1.5);
\draw [black, thick] (6.5, 1.5) -- (7.5, 1.5);
\draw [black, thick] (8.5, 1.5) -- (9.5, 1.5);
\draw [black, thick] (0.5, 2.5) -- (1.5, 2.5);
\draw [black, thick] (3.5, 2.5) -- (4.5, 2.5);
\draw [black, thick] (4.5, 2.5) -- (5.5, 2.5);
\draw [black, thick] (6.5, 2.5) -- (7.5, 2.5);
\draw [black, thick] (7.5, 2.5) -- (8.5, 2.5);
\draw [black, thick] (9.5, 2.5) -- (10.5, 2.5);
\draw [black, thick] (0.5, 3.5) -- (1.5, 3.5);
\draw [black, thick] (2.5, 3.5) -- (3.5, 3.5);
\draw [black, thick] (3.5, 3.5) -- (4.5, 3.5);
\draw [black, thick] (5.5, 3.5) -- (6.5, 3.5);
\draw [black, thick] (7.5, 3.5) -- (8.5, 3.5);
\draw [black, thick] (9.5, 3.5) -- (10.5, 3.5);
\draw [black, thick] (0.5, 4.5) -- (1.5, 4.5);
\draw [black, thick] (2.5, 4.5) -- (3.5, 4.5);
\draw [black, thick] (4.5, 4.5) -- (5.5, 4.5);
\draw [black, thick] (6.5, 4.5) -- (7.5, 4.5);
\draw [black, thick] (8.5, 4.5) -- (9.5, 4.5);
\draw [black, thick] (0.5, 5.5) -- (1.5, 5.5);
\draw [black, thick] (1.5, 5.5) -- (2.5, 5.5);
\draw [black, thick] (3.5, 5.5) -- (4.5, 5.5);
\draw [black, thick] (4.5, 5.5) -- (5.5, 5.5);
\draw [black, thick] (6.5, 5.5) -- (7.5, 5.5);
\draw [black, thick] (7.5, 5.5) -- (8.5, 5.5);
\draw [black, thick] (9.5, 5.5) -- (10.5, 5.5);
\draw [black, thick] (0.5, 0.5) -- (0.5, 1.5);
\draw [black, thick] (4.5, 0.5) -- (4.5, 1.5);
\draw [black, thick] (5.5, 0.5) -- (5.5, 1.5);
\draw [black, thick] (6.5, 0.5) -- (6.5, 1.5);
\draw [black, thick] (7.5, 0.5) -- (7.5, 1.5);
\draw [black, thick] (8.5, 0.5) -- (8.5, 1.5);
\draw [black, thick] (9.5, 0.5) -- (9.5, 1.5);
\draw [black, thick] (10.5, 0.5) -- (10.5, 1.5);
\draw [black, thick] (1.5, 1.5) -- (1.5, 2.5);
\draw [black, thick] (2.5, 1.5) -- (2.5, 2.5);
\draw [black, thick] (3.5, 1.5) -- (3.5, 2.5);
\draw [black, thick] (10.5, 1.5) -- (10.5, 2.5);
\draw [black, thick] (0.5, 2.5) -- (0.5, 3.5);
\draw [black, thick] (2.5, 2.5) -- (2.5, 3.5);
\draw [black, thick] (5.5, 2.5) -- (5.5, 3.5);
\draw [black, thick] (6.5, 2.5) -- (6.5, 3.5);
\draw [black, thick] (8.5, 2.5) -- (8.5, 3.5);
\draw [black, thick] (9.5, 2.5) -- (9.5, 3.5);
\draw [black, thick] (1.5, 3.5) -- (1.5, 4.5);
\draw [black, thick] (4.5, 3.5) -- (4.5, 4.5);
\draw [black, thick] (7.5, 3.5) -- (7.5, 4.5);
\draw [black, thick] (10.5, 3.5) -- (10.5, 4.5);
\draw [black, thick] (0.5, 4.5) -- (0.5, 5.5);
\draw [black, thick] (2.5, 4.5) -- (2.5, 5.5);
\draw [black, thick] (3.5, 4.5) -- (3.5, 5.5);
\draw [black, thick] (5.5, 4.5) -- (5.5, 5.5);
\draw [black, thick] (6.5, 4.5) -- (6.5, 5.5);
\draw [black, thick] (8.5, 4.5) -- (8.5, 5.5);
\draw [black, thick] (9.5, 4.5) -- (9.5, 5.5);
\draw [black, thick] (10.5, 4.5) -- (10.5, 5.5);

        \end{tikzpicture}
        \caption{$12$ straights in $G(11,6)$. This is the case $m-n\equiv5\pmod6$, where the upper bound satisfies $h(m,n)=g(m,n)+2$.}
        \label{11x6ub}
    \end{figure}

As can be seen above, the upper bound is not always equal to the lower bound; there are several cases where tight constructions have been found that do not follow from above. Constructions for $G(6a,6)$, $G(6a+5,6)$, and $G(6a+4,10)$ that attain the lower bound exist, by exhibiting unfoldable (but not extensible) constructions for $G(10,6)$, $G(11,6)$, and $G(16,10)$ respectively. Furthermore, it is possible to prove via considering connectedness that the lower bound given can be improved by two in the case of $m = n+4$.

If the true maximum number of turns is $f_{\max}(m,n)$, then the difference between this number and the upper bound is $$f_d(m,n)=h(m,n)-f_{\max}(m,n),$$
satisfying $f_d(m,n) \in \{0,2\}$.

This result, along with Claim \ref{max4|n}, completes the proof of Theorem \ref{final result}.

\section{Conclusion}
The results in this paper should also be applicable to grids bounded by shapes other than rectangles; the key lemmas in the introduction and the half-form argument can still be used in other subsets of $\mathbb{Z}^2$. Constructions for the maximum case given for rectangles in subsection \ref{2mod4} are either tight or within $2$ of a tight bound; resolving which of the two holds in the remaining cases is left open, though the answer seems somewhat erratic.

\begin{question}
For which $n \equiv 2 \pmod 4$ and $m > n$ does $f_d(m,n) = 2$
hold?
\end{question}

We conjecture that for sufficiently large rectangles (formally, $m,m-n>C$ for some constant $C$) that if $m-n\equiv m'-n'\pmod6$, that $f_d(m,n)=f_d(m',n'),$ that is, that the answer is ``eventually periodic modulo $6$".

\begin{question}
For which triples of positive integers $k,m,n$ do there exist a Hamilton cycle in $G(m,n)$ with exactly $k$ turns?
\end{question}

We conjecture that for all $m\ge n$, and any even $k$ between the minimum and maximum possible number of turns on $G(m,n)$, a cycle with $k$ turns exists, unless $n=4$ and $k=f_{\max}(m,n)-2$, in which case it can be proved to be unattainable. Extensibility and half-forms will perhaps be helpful tools in finding constructions.

\begin{acknowledgement}
Many thanks to Imre Leader for providing feedback and suggestions, and to Nikolai Beluhov for pointing us towards some references on previous work on the problem.
\end{acknowledgement}

\printbibliography

\end{document}